\documentclass[11pt,a4paper]{article}
\usepackage{amsfonts,amsmath,amssymb,amsthm}
\usepackage{verbatim}

\usepackage{hyperref}

\usepackage{accents}
\usepackage{enumerate}

\usepackage{graphicx}
\usepackage{xcolor}

\newcommand{\dis}{\displaystyle}

\newcommand{\noi}{\noindent}
\newcommand{\halmos}{\rule{1ex}{1.4ex}}
\newcommand{\QED}{\nopagebreak{\hspace*{\fill}$\halmos$\medskip}}
\newcommand{\med}{\medskip}
\newcommand{\quand}{\quad\mbox{and}\quad}

\newtheoremstyle{mythm}
  {}
  {}
  {\itshape}
  {}
  {\bfseries}
  {}
  {.5em}
  {#1 #2 \thmnote{(#3)}}

\theoremstyle{mythm}
\newtheorem{theorem}{Theorem}
\newtheorem{proposition}[theorem]{Proposition}
\newtheorem{lemma}[theorem]{Lemma}

\newtheorem{exercise}[theorem]{Exercise}
\newtheorem{corollary}[theorem]{Corollary}
\newtheorem{conjecture}[theorem]{Conjecture}

\newtheorem{counterex}[theorem]{Counterexample}
\newtheorem{definition}[theorem]{Definition}

\newcommand{\bt}{\begin{theorem}}
\newcommand{\et}{\end{theorem}}
\newcommand{\bl}{\begin{lemma}}
\newcommand{\el}{\end{lemma}}
\newcommand{\bp}{\begin{proposition}}
\newcommand{\ep}{\end{proposition}}
\newcommand{\bcor}{\begin{corollary}}
\newcommand{\ecor}{\end{corollary}}
\newcommand{\br}{\begin{remark}\rm}
\newcommand{\er}{\end{remark}}
\newcommand{\bcon}{\begin{conjecture}}
\newcommand{\econ}{\end{conjecture}}
\newcommand{\bex}{\begin{exercise}}
\newcommand{\eex}{\end{exercise}}
\newcommand{\bcou}{\begin{counterex}}
\newcommand{\ecou}{\end{counterex}}

\newtheorem{remark}[theorem]{Remark}

\newenvironment{Proof}[1][]{\noi\textbf{Proof #1}}{\QED}
\newcommand{\bpro}{\begin{Proof}}
\newcommand{\epro}{\end{Proof}}

\newcommand{\be}{\begin{equation}}
\newcommand{\ee}{\end{equation}}
\newcommand{\ba}{\begin{array}}
\newcommand{\ea}{\end{array}}
\newcommand{\bac}{\begin{array}{r@{\,}c@{\,}l}}
\newcommand{\bc}{\be\begin{array}{r@{\,}c@{\,}l}}
\newcommand{\ec}{\end{array}\ee}

\newcommand{\al}{\alpha}

\newcommand{\ga}{\gamma}
\newcommand{\Ga}{\Gamma}
\newcommand{\de}{\delta}
\newcommand{\De}{\Delta}
\newcommand{\eps}{\varepsilon}

\newcommand{\tet}{\theta}
\newcommand{\om}{\omega}
\newcommand{\Om}{\Omega}

\newcommand{\si}{\ensuremath{\sigma}}

\newcommand{\Fi}{{\cal F}}

\newcommand{\Ki}{{\cal K}}

\newcommand{\Mi}{{\cal M}}

\newcommand{\Pc}{{\cal P}}

\newcommand{\Si}{{\cal S}}
\newcommand{\Ti}{{\cal T}}

\newcommand{\p}{\mathbb{P}}

\newcommand{\R}{{\mathbb R}}
\newcommand{\N}{{\mathbb N}}

\renewcommand{\S}{{\mathbb S}}
\newcommand{\T}{{\mathbb T}}
\newcommand{\U}{{\mathbb U}}
\newcommand{\E}{{\mathbb E}}
\newcommand{\F}{{\mathbb F}}

\newcommand{\I}{{\mathbb I}}
\renewcommand{\P}{{\mathbb P}}
\newcommand{\Ob}{{\mathbb O}}
\newcommand{\A}{{\mathbb A}}
\newcommand{\B}{{\mathbb B}}

\newcommand{\volgt}{\ensuremath{\Rightarrow}}
\newcommand{\up}{\uparrow}
\newcommand{\down}{\downarrow}
\newcommand{\sub}{\subset}
\newcommand{\beh}{\backslash}

\newcommand{\isd}{\stackrel{\scriptstyle{\rm d}}{=}}

\newcommand{\ti}{\tilde}

\newcommand{\ov}{\overline}
\newcommand{\un}{\underline}

\newcommand{\lvec}[1]{\accentset{\leftarrow}{#1}}

\newcommand{\pa}{\partial}
\newcommand{\ffrac}[2]{{\textstyle\frac{{#1}}{{#2}}}}

\newcommand{\nab}{\nabla}

\newcommand{\di}{\mathrm{d}}
\newcommand{\half}{{[0,\infty)}}

\newcommand{\ha}{\ffrac{1}{2}}

\newcommand{\ibf}{\mathbf{i}}
\newcommand{\jbf}{\mathbf{j}}
\newcommand{\kbf}{\mathbf{k}}

\newcommand{\wurz}{\varnothing}

\newcommand{\percol}[1]{\overset{{#1}}{\longrightarrow}}

\begin{document}

\makeatletter\@addtoreset{equation}{section}
\makeatother\def\theequation{\thesection.\arabic{equation}}

\renewcommand{\labelenumi}{{\rm (\roman{enumi})}}
\renewcommand{\theenumi}{\roman{enumi}}

\title{A phase transition between endogeny and nonendogeny}

\author{Bal\'azs R\'ath\footnote{Department of Stochastics, Institute of Mathematics,
Budapest University of Technology and Economics, MTA-BME Stochastics Research Group,
M\H{u}egyetem rkp. 3., H-1111 Budapest, Hungary. Alfr\'ed R\'enyi Institute of Mathematics, Re\'altanoda utca 13-15, 1053 Budapest, Hungary.
 rathb@math.bme.hu},
Jan~M.~Swart\footnote{The Czech Academy of Sciences,
  Institute of Information Theory and Automation,
  Pod vod\'arenskou v\v e\v z\' i 4,
  18200 Praha 8,
  Czech Republic.
  swart@utia.cas.cz},
and M\'arton Sz\H{o}ke\footnote{Department of Stochastics, Institute of Mathematics,
Budapest University of Technology and
M\H{u}egyetem rkp. 3., H-1111 Budapest, Hungary. Alfr\'ed R\'enyi Institute of Mathematics, Re\'altanoda utca 13-15, 1053 Budapest, Hungary.
 szokemarton3@gmail.com  }}

\date{\today}
\maketitle

\begin{abstract}\noi

The Marked Binary Branching Tree (MBBT) is the family tree of a rate one
binary branching process, on which points have been generated according to a
rate one Poisson point process, with i.i.d.\ uniformly distributed activation
times assigned to the points. In frozen percolation on the MBBT, initially,
all points are closed, but as time progresses points can become either frozen
or open. Points become open at their activation times provided they have not
become frozen before. Open points connect the parts of the tree below and
above it and one says that a point percolates if the tree above it is
infinite. We consider a version of frozen percolation on the MBBT in which at
times of the form $\tet^n$, all points that percolate are frozen. The limiting
model for $\tet\to 1$, in which points freeze as soon as they percolate, has
been studied before by R\'ath, Swart, and Terpai. We extend their results by
showing that there exists a $0<\tet^\ast<1$ such that the model is endogenous
for $\tet\leq\tet^\ast$ but not for $\tet>\tet^\ast$. This means that for
$\tet\leq\tet^\ast$, frozen percolation is a.s.\ determined by the MBBT but
for $\tet>\tet^\ast$ one needs additional randomness to describe it.
\end{abstract}
\vspace{.5cm}

\noi
%
 \\[10pt]
{\it Acknowledgements:} The authors would like to thank both anonymous referees for their numerous constructive comments that improved the
quality of this paper.
 The work of B.~R\'ath was partially supported by grant NKFI-FK-123962 of NKFI (National Research, Development and Innovation Office), the Bolyai Research Scholarship of the Hungarian Academy of Sciences, the \'UNKP-20-5-BME-5 New National Excellence Program of the Ministry for Innovation and Technology, and the ERC Synergy under Grant No. 810115 - DYNASNET. J.M.~Swart is supported by grant 20-08468S of the Czech Science Foundation (GA CR). The work of M\'arton Sz\H{o}ke is partially supported by the ERC Consolidator Grant 772466 ``NOISE''.

\newpage

{\setlength{\parskip}{-2pt}\tableofcontents}

\newpage

\section{Introduction and main results}

\subsection{Introduction}

The concept of frozen percolation was introduced by Aldous \cite{Ald00}. In
it, i.i.d.\ activation times that are uniformly distributed on $[0,1]$ are
assigned to the edges of an infinite, unoriented graph. Initially, all edges
are closed. At its activation time, an edge opens, provided it is not frozen.
Here, by definition, an edge freezes as soon as one of its endvertices becomes
part of an infinite open cluster. For general graphs, the existence of a
process satisfying this description is not obvious. Indeed, Benjamini and
Schramm observed that on the square lattice, frozen percolation does not exist
(see \cite[Section~3]{BT01} for an account of the argument).

On the other hand, Aldous \cite{Ald00} showed that frozen percolation on the
infinite 3-regular tree does exist. Under natural additional assumptions, such
a process is even unique in law. This was partially already observed in
\cite{Ald00} and made more precise in \cite[Thm~2]{RST19}. The problem of
almost sure uniqueness stayed open for 19 years, but has recently been solved
negatively in \cite[Thm~3]{RST19}, where it is shown that the question
whether a given edge freezes cannot be decided only by looking at the activation
times of all edges.


The proof of \cite[Thm~3]{RST19} depends on detailed calculations that are
specific to the details of the model. As a result, the question of almost sure
uniqueness is still open for frozen percolation on $n$-regular trees with
$n>3$. This raises the question whether model specific calculations are
necessary, or whether the absence of almost sure uniqueness can alternatively
be demonstrated by more general, ``soft'' arguments that have so far been
overlooked.

The results in the present paper suggest that this is not the case and model
specific calculations are, to some degree, unavoidable. We look at a modified
model in which edges can freeze only at a certain countable set of times. For
the resulting model, which depends on a parameter $0<\tet<1$, we show that
under the same natural additional assumptions that guarantee uniquess in law,
there exists a nontrivial critical value $\tet^\ast$ such that almost sure
uniqueness holds for $\tet\leq\tet^\ast$ but not for $\tet>\tet^\ast$.

It turns out that it is mathematically simpler to formulate our results for
frozen percolation on a certain oriented tree,
the Marked Binary Branching Tree (MBBT), a random oriented continuum tree
introduced in \cite{RST19}. Using methods of Section~3 of that paper, our
results can also be translated into results for the unoriented 3-regular
tree. For brevity, we omit the details of the latter step and stick for the
remainder of the paper to the oriented (rather than the unoriented) setting on the MBBT (rather than the 3-regular tree).

\subsection{Frozen percolation on the MBBT}

Let $\T$ be the set of all finite words $\ibf=i_1\cdots i_n$ $(n\geq 0)$ made
up from the alphabet $\{1,2\}$. We call $|\ibf|:=n$ the length of the word
$\ibf$ and denote the word of length zero by $\wurz$, which we distinguish
notationally from the empty set $\emptyset$. The concatenation of two words
$\ibf=i_1\cdots i_n$ and $\jbf=j_1\cdots j_m$ is denoted by
$\ibf\jbf:=i_1\cdots i_nj_1\cdots j_m$. We view $\T$ as an oriented tree with
root $\wurz$, in which each point $\ibf$ has two offspring $\ibf 1$ and
$\ibf 2$, and each point $\ibf=i_1\cdots i_n$ except for the root has one
parent $\lvec\ibf:=i_1\cdots i_{n-1}$. In pictures, we draw the root at
the bottom and we draw the descendants of a point above their predecessor.
By definition, a \emph{rooted subtree} of $\T$ is a subset $\U\sub\T$ such
that $\lvec\ibf\in\U$ for all $\ibf\in\U\beh\{\wurz\}$. We call
$\pa\U:=\{\ibf\in\T\beh\U:\lvec\ibf\in\U\}$ the \emph{boundary} of $\U$,
and we use the convention that $\partial\U=\{\varnothing\}$ if $\U=\emptyset$.

Let $(\tau_\ibf,\kappa_\ibf)_{\ibf\in\T}$ be i.i.d.\ uniformly distributed on
$[0,1]\times\{1,2\}$. We interpret $\tau_\ibf$ as the \emph{activation time}
of $\ibf$ and $\kappa_\ibf$ as its \emph{number of legal offspring}. If
$\kappa_\ibf=1$, then we call $\ibf 1$ and $\ibf 2$ the legal and illegal
offspring of $\ibf$, respectively. Points $\ibf\in\T$ with $\kappa_\ibf=1$
or $=2$ are called \emph{internal points} and \emph{branching points},
respectively. We denote the corresponding sets as
$\I:=\{\ibf\in\T:\kappa_\ibf=1\}$ and $\B:=\{\ibf\in\T:\kappa_\ibf=2\}$. Only
activation times of internal points matter; activation times of branching
points will not be used. For any $\ibf\in\T$ and $A\sub\T$, we write
$\ibf\percol{A}\infty$ if there exist $(j_k)_{k\geq 1}$ such that
\be\label{percol}
{\rm(i)}\quad j_{k+1}\leq\kappa_{\ibf j_1\cdots j_k}
\quand
{\rm(ii)}\quad \ibf j_1\cdots j_k\in A\quad\mbox{for all }k\geq 0.
\ee
In words, this says that there is an infinite open upwards path through $A$
starting at $\ibf$ such that each next point is a legal offspring of its
parent.

We will be interested in frozen percolation on $\T$ with the following
informal description. At any time, points can be \emph{closed}, \emph{frozen},
or \emph{open}. Once a point is frozen or open, it stays that way. Initially,
all branching points are open and all internal points are closed. Branching
points stay open for all time. An internal
point $\ibf$ becomes open at its activation time $\tau_\ibf$ provided that, by this
time, it has not yet become frozen. The rules for freezing points are as
follows. We fix a set $\Xi\sub(0,1]$ that is closed w.r.t.\ the relative
topology of $(0,1]$. Letting $\Ob^t$ denote the set of
open points at time $t$, we decree that up to and including its
activation time, a closed internal point $\ibf$ becomes frozen at the
first time in $\Xi$ when its legal offspring percolates, i.e., when
$\ibf 1\percol{\Ob^t}\infty$.


Let
\be\label{Tt}
\T^t:=\{\ibf\in\I:\tau_\ibf\leq t\big\}\cup\B
\qquad(0\leq t\leq 1)
\ee
denote the set of all points at time $t$ that are either an internal point that has already been activated
or a branching point. Let
$\F$ denote the set of internal points that eventually become frozen.
Since once a point opens or freezes, it stays open or frozen for the remaining time,
the set of open points at time $t$ is given by $\Ob^t=\T^t\beh\F$. In view of this,
we make our informal description precise
by saying that a random subset $\F$ of $\T$ \emph{solves the frozen percolation
  equation} for the \emph{set of possible freezing times} $\Xi$ if
\be\label{frozdef}
\mbox{$\ibf\in\F$ if and only if $\kappa_\ibf=1$ and $\ibf
  1\percol{\T^t\beh\F}\infty$ for some $t\in\Xi\cap(0,\tau_\ibf]$,}
\ee
which says that the points that eventually become frozen are those internal
points $\ibf$ for which $\ibf 1$ percolates at some time in $\Xi$ before or
at the activation time of $\ibf$.

It turns out that solutions to (\ref{frozdef}) always exist, but the question
of uniqueness is more subtle. To get at least uniqueness in law, we impose
additional conditions. We write $\om_\ibf:=(\tau_\ibf,\kappa_\ibf)$
$(\ibf\in\T)$ and for any $\jbf\in\T$, we let
\be\label{Om}
\Om_\jbf:=\big(\om_{\jbf\ibf}\big)_{\ibf\in\T}
\ee
denote the i.i.d.\ randomness that resides in the subtree of $\T$ rooted at
$\jbf$. In particular, we write $\Om:=\Om_\wurz$. If $\F$ is a solution to the
frozen percolation equation, then for each $\jbf\in\T$, we define a random
subset $\F_\jbf$ of $\T$ by
\be
\F_\jbf:=\{\ibf\in\T:\jbf\ibf\in\F\}.
\ee
We say that a solution $\F$ to the frozen percolation equation (\ref{frozdef})
is \emph{stationary} if the law of $(\Om_\jbf,\F_\jbf)$ does not depend on
$\jbf\in\T$. We say that $\F$ is \emph{adapted} if for each finite rooted
subtree $\U\sub\T$, the collection of random variables
$(\Om_\jbf,\F_\jbf)_{\jbf\in\pa\U}$ is independent of
$(\om_\ibf)_{\ibf\in\U}$. Finally, we say that $\F$ \emph{respects the tree
  structure} if $(\Om_\jbf,\F_\jbf)_{\jbf\in\pa\U}$ is a collection of
independent random variables for each finite rooted subtree $\U\sub\T$.

With these definitions, we can formulate our first result about existence and
uniqueness in law of solutions to the frozen percolation equation
(\ref{frozdef}). In the special case that $\Xi=(0,1]$, the following theorem
has been proved before in (in a somewhat different guise) in
\cite[Thm~2]{RST19}.

\bt[Uniqueness in law of frozen percolation]
Let\label{T:frozen} $\Xi$ be a closed subset of $(0,1]$ (w.r.t.\ the relative
topology). Then there exists a solution $\F$ of the frozen percolation
equation (\ref{frozdef}). This solution can be chosen so that it is
stationary, adapted, and respects the tree structure. Subject to these
additional conditions, the joint law of $\Om$ and $\F$ is uniquely determined.
\et

We will prove Theorem~\ref{T:frozen} in Subsection~\ref{S:uni}.
As we will see in the coming subsections, the question of almost sure
uniqueness of solutions to the frozen percolation equation is subtle and the
answer depends on the choice of the closed set~$\Xi$.

In the remainder of the present subsection, which can be skipped at a first
reading, we explain how our set-up relates to the definition of the Marked
Binary Branching Tree (MBBT) introduced in \cite{RST19}. Let
\be
\S:=\big\{i_1\cdots i_n\in\T:i_m\leq\kappa_{i_1\cdots i_{m-1}}
\ \forall 1\leq m\leq n\big\}
\ee
denote the random rooted subtree of $\T$ consisting of all legal descendants
of the root. Then $\S$ is the family tree of a branching process in which each
individual has one or two offspring, with equal probabilities. For any rooted
subtree $\U\sub\S$, we call
\be
\nab\U:=\pa\U\cap\S
\ee
the \emph{boundary} of $\U$ \emph{relative to} $\S$.

Let $(\ell_\ibf)_{\ibf\in\T}$ be i.i.d.\ exponentially distributed random
variables with mean $1/2$, independent of $\Om$. We interpret $\ell_\ibf$ as
the lifetime of the individual $\ibf$ and let
\be\label{birthdeath}
b_{i_1\cdots i_n}:=\sum_{k=0}^{n-1}\ell_{i_1\cdots i_k}
\quand
d_{i_1\cdots i_n}:=\sum_{k=0}^{n}\ell_{i_1\cdots i_k}
\ee
with $b_\wurz:=0$ and $d_\wurz:=\ell_\wurz$ denote the birth and death
times of $i_1\cdots i_n\in\T$. For $h\geq 0$, we let
\be\ba{l@{\qquad}l}\label{nabS}
\dis\T_h:=\big\{\ibf\in\T:d_\ibf\leq h\big\},
&\dis\pa\T_h=\big\{\ibf\in\T:b_\ibf\leq h<d_\ibf\big\},\\[5pt]
\dis\S_h:=\T_h\cap\S,
&\dis\nab\S_h=\pa\T_h\cap\S
\ec
denote the sets of individuals in $\T$ or $\S$ that have died by time $h$ and
those that are alive at time $h$, respectively. Note that the former are
a.s.\ finite rooted subtrees of $\T$ and $\S$, respectively, and the latter
are their boundaries relative to $\T$ or $\S$. Now
\be\label{branch}
(\nab\S_h)_{h\geq 0}
\ee
is a continuous-time branching process subject to the following dynamics:
\begin{itemize}
\item each individual $\ibf$ is with rate 1 replaced by two new individuals
  $\ibf 1$ and $\ibf 2$,
\item each individual $\ibf$ is with rate 1 replaced by one new individual
  $\ibf 1$.
\end{itemize}
Let $(\nab\S_{h-})_{h\geq 0}$ denote the left-continuous modification of the
branching process in (\ref{branch}) and let
\be
\Ti:=\big\{(\ibf,h):\ibf\in\nab\S_{h-}, h\geq 0\big\}.
\ee
As in \cite[Subsection~1.5]{RST19}, we equip $\Ti$ with a metric $d$
by setting $d\big((\ibf,h),(\jbf,g)\big):=h+g-\tau$, where $\tau$ is
the last time before $h\wedge g$ when there existed a common ancestor
of $\ibf$ and $\jbf$. Then $\Ti$ is a random continuum tree. We can
think of $\Ti$ as the family tree of a rate one binary branching
process. Recall that $\I=\{\ibf\in\T:\kappa_\ibf=1\}$ denotes the set
of internal points of $\T$. Let
\be
\Pi_0:=\big\{(\ibf,d_\ibf):\ibf\in\I\cap\S\big\}
\quand
\Pi:=\big\{(\ibf,d_\ibf,\tau_\ibf):\ibf\in\I\cap\S\big\}.
\ee
In words, $\Pi_0$ consists of all points $z=(\ibf,d_\ibf)$ in the continuum tree $\Ti$ at which an individual $\ibf$ dies and is replaced by a single new individual $\ibf 1$, and $\Pi$ consists of all pairs $(z,\tau_z)$ where $z\in\Pi_0$ and $\tau_z$ is the activation time of the individual that dies at this point. Then the pair $(\Ti,\Pi)$ is a \emph{Marked Binary Branching Tree} (MBBT) as defined in \cite[Subsection~1.5]{RST19}. As explained in \cite[Subsection~1.7]{RST19}, the MBBT naturally arises as the near-critical scaling limit of percolation on a wide class of oriented trees.

If we forget about the specific labeling of elements of $\Ti$, i.e., if we are only interested in $\Ti$ as a metric space where we view two metric spaces as equal if they are isometric, then we can no longer recognise from $\Ti$ at which points a single individual is replaced by a single individual with a different label. In such a setting one can check that conditional on $\Ti$, the set $\Pi$ is a Poisson point process of intensity one on $\Ti\times[0,1]$. In particular, $\Pi_0$ is a Poisson point process of intensity one on $\Ti$ and conditionally on $(\Ti,\Pi_0)$, there is an independent, uniformly distributed activation time $\tau_z$ attached to each point $z\in\Pi_0$.

Frozen percolation on the MBBT has been introduced in
\cite[Subsection~1.6]{RST19}. Our earlier definitions, translated into the
language of the MBBT, result in a process with the following informal
description. Initially, all points $z\in\Pi_0$ are closed. Such
points open at their activation time $\tau_z$, provided that by this time
they have not yet become frozen. A point $z\in\Pi_0$ freezes at
the first time in $\Xi$ before or at its activation time when the open
component of $\Ti$ that sits just above the point has infinite size.

\subsection{Burning times}\label{S:burn}

Let $\Xi\sub(0,1]$ be a relatively closed set of possible freezing times and let $\F$ be
a solution to the frozen percolation equation (\ref{frozdef}). We define the
\emph{burning time} of a point $\ibf\in\T$ as
\be\label{YF}
Y_\ibf:=\inf\big\{t\in\Xi:
\ibf\percol{\T^t\beh\F}\infty\big\}\qquad(\ibf\in\T),
\ee
with the convention that $\inf\emptyset:=\infty$. The choice of the term
``burning time'' is motivated by a certain analogy with forest fire models.
The following lemma implies that if $Y_\ibf\leq 1$, then the infimum in
(\ref{YF}) is in fact a minimum.

\bl[Percolation times]
For\label{L:perctime} any random subset $\A\sub\T$ and $\ibf\in\T$, the set
$\{t\in[0,1]:\ibf\percol{\T^t\beh\A}\infty\}$ is a.s.\ closed.
\el

We will prove Lemma~\ref{L:perctime} and Lemma~\ref{L:FY} below in
Section~\ref{S:uni}. By formula (\ref{YF}), the burning times
$(Y_\ibf)_{\ibf\in\T}$ are a.s.\ uniquely determined by the set $\F$ and the
i.i.d.\ randomness $\Om$. The following lemma shows that conversely, given
$\Om$ and $(Y_\ibf)_{\ibf\in\T}$, one can recover $\F$.

\bl[Frozen points]
Let\label{L:FY} $\F$ be a solution to the frozen percolation equation
(\ref{frozdef}) and let $(Y_\ibf)_{\ibf\in\T}$ be defined by (\ref{YF}). Then
\be\label{FY}
\F=\big\{\ibf\in\I:Y_{\ibf 1}\leq\tau_\ibf\big\}.
\ee
\el

\begin{remark}
If $\F$ is adapted, then $Y_{\ibf 1}$ is independent of $\tau_\ibf$ and hence $\P[Y_{\ibf 1}=\tau_\ibf]=0$ for each $\ibf\in\T$. According to our definitions, the point $\ibf$ freezes when $Y_{\ibf 1}=\tau_\ibf$, but as long as we only discuss adapted solutions, it in fact does not matter how things are defined in this case.
\end{remark}

Let $I:=[0,1]\cup\{\infty\}$. If $\F$ is a solution to the frozen percolation
equation (\ref{frozdef}), then it is not hard to see that the burning times
$(Y_\ibf)_{\ibf\in\T}$ satisfy the inductive relation
\be\label{Yind}
Y_\ibf=\chi[\tau_\ibf,\kappa_\ibf](Y_{\ibf 1},Y_{\ibf 2}),
\ee
where $\chi:[0,1]\times\{1,2\}\times I^2\to I$ is the function
\be\label{chi_def}
\chi[\tau,\kappa](x,y):=\left\{\ba{ll}
x\quad&\mbox{if }\kappa=1,\ x>\tau,\\[5pt]
\infty\quad&\mbox{if }\kappa=1,\ x\leq\tau,\\[5pt]
x\wedge y\quad&\mbox{if }\kappa=2.\ea\right.
\ee

Assume that $\F$ is stationary, adapted, and respects the tree
structure. Then the law of $Y_\wurz$ satisfies the Recursive Distributional
Equation (RDE)
\be\label{uni_RDE}
Y_\wurz\isd\chi[\om](Y_1,Y_2),
\ee
where $\isd$ denotes equality in distribution, $Y_1,Y_2$ are i.i.d.\ copies of
$Y_\wurz$, and $\om$ is an independent uniformly distributed random variable
on $[0,1]\times\{1,2\}$. Proposition~37 of \cite{RST19} classifies all
solutions of the RDE (\ref{uni_RDE}). Expanding on that result, we can prove the
following lemma, which is the basis of our proof of Theorem~\ref{T:frozen}.

\bl[Law of burning times]
For\label{L:burnlaw} each set $\Xi\sub(0,1]$ that is closed w.r.t.\ the
relative topology of $(0,1]$, there exists a unique
probability measure $\rho_\Xi$ on $I$ such that
\begin{enumerate}
\item $\rho_\Xi$ solves the RDE (\ref{uni_RDE}),
\item $\rho_\Xi$ is concentrated on $\Xi\cup\{\infty\}$,
\item $\rho_\Xi\big([0,t]\big)\geq\ha t$ for all $t\in\Xi$.
\end{enumerate}
Assume that $\F$ solves the frozen percolation equation (\ref{frozdef}) for the
set of possible freezing times $\Xi$ and that $\F$ is stationary, adapted, and
respects the tree structure. Then the burning time of the root $Y_\wurz$,
defined in (\ref{YF}), has law $\rho_\Xi$.
\el

We will prove Lemma~\ref{L:burnlaw} together with Lemma~\ref{L:genRDE} below
in Section~\ref{S:uni}. The following lemma shows that every solution of the
RDE (\ref{uni_RDE}) is of the form $\rho_\Xi$ for some closed set
$\Xi\sub(0,1]$. Below, ${\rm supp}(\mu)$ denotes the support of a measure~$\mu$.

\bl[General solutions to the RDE]
If\label{L:genRDE} $\rho$ solves the RDE (\ref{uni_RDE}), then $\rho=\rho_\Xi$
with $\Xi:=(0,1]\cap{\rm supp}(\rho)$.
\el

By condition~(ii) of Lemma~\ref{L:burnlaw}, for a general closed subset $\Xi\sub(0,1]$, we have $(0,1]\cap{\rm supp}(\rho_\Xi)\sub\Xi$. This inclusion may be strict,\footnote{For example, if $\Xi=\{s,t\}$ with $0<s<t\leq 1$ and $t\leq 2s$, then using Lemma~\ref{L:RDEint} below it is easy to check that $\rho_\Xi=s\de_s+(1-s)\de_\infty$.} however, so the correspondence between solutions of the RDE (\ref{uni_RDE}) and sets of possible freezing times is not one-to-one.

\subsection{Almost sure uniqueness}\label{S:as}

Recall from (\ref{Om}) that $\Om=(\om_\ibf)_{\ibf\in\T}$ with
$\om_\ibf=(\tau_\ibf,\kappa_\ibf)$. For a given set $\Xi\sub(0,1]$ of possible
freezing times, we say that solutions to the frozen percolation equation
(\ref{frozdef}) are \emph{almost surely unique} if, whenever $\F$ and $\F'$
solve (\ref{frozdef}) relative to the same $\Om$, one has $\F=\F'$ a.s.

Let us first note that it is easy to show that if $\Xi$ is a finite subset of $(0,1]$ then the solutions of (\ref{frozdef}) are almost surely unique. Indeed, if $\Xi=\{ t_1,...,t_n \}$ with $0<t_1<\dots<t_n \leq 1$ then one proves by induction on $k=1,\dots,n$  that the set of vertices that burn at time $t_k$ is determined by $\Omega$. This implies that the burning time $Y_\ibf$ of each vertex  $\ibf\in\T$ is determined by $\Omega$, hence the set $\F$ is also determined by $\Omega$ using Lemma~\ref{L:FY}.

For the remainder of the paper, we will mostly focus our attention on a
one-parameter family of sets of possible burning times. For $0<\tet<1$, we
define $\Xi_\tet:=\{\tet^n:n\in\N\}$ (with $\N:=\{0,1,2,\ldots\}$) and we set
$\Xi_1:=(0,1]$, which can naturally be viewed as the limit of $\Xi_\tet$ as
$\tet\to 1$. As a straightforward application of \cite[Prop~37]{RST19},
one can check that for these sets, the probability laws $\rho_\Xi$ from
Lemma~\ref{L:burnlaw} are given by
\bc\label{muXi}
\dis\rho_{\Xi_\tet}(\di t)
&=&\dis\frac{1-\tet}{1+\tet}\sum_{k=0}^\infty\tet^k\de_{\tet^k}(\di t)
+\frac{\tet}{1+\tet}\de_\infty(\di t)\qquad(0<\tet<1),\\[15pt]
\dis\rho_{\Xi_1}(\di t)&=&\dis\ha\di t+\ha\de_\infty(\di t).
\ec
It is not hard to see that $\rho_{\Xi_\tet}$ converges weakly to $\rho_{\Xi_1}$
as $\tet\to 1$.

We conjecture that for the sets $\Xi_\tet$ with $0<\tet < \tfrac{1}{2}$, solutions to the frozen percolation equation are almost surely unique. We have not been able to prove this, but we can prove that there exists a $\tfrac{1}{2}<\tet^\ast<1$ such that almost sure uniqueness does not hold for $\tet>\tet^\ast$ and almost sure uniqueness holds under additional assumptions for $\tet\leq\tet^\ast$.

To explain this in more detail, fix $\Om=(\om_\ibf)_{\ibf\in\T}$, let $\F$ be a solution to the frozen percolation equation (\ref{frozdef}) that is stationary, adapted, and respects the tree structure, and let $(Y_\ibf)_{\ibf\in\T}$ be the burning times defined in (\ref{YF}). Then
\begin{enumerate}
\item For each finite rooted subtree $\U\sub\T$, the r.v.'s $(Y_\ibf)_{\ibf\in\pa\U}$ are i.i.d.\ and independent of $(\om_\ibf)_{\ibf\in\U}$.
\item $\dis Y_\ibf=\chi[\om_\ibf](Y_{\ibf 1},Y_{\ibf 2})\qquad(\ibf\in\T)$.
\end{enumerate}
This means that $(\om_\ibf,Y_\ibf)_{\ibf\in\T}$ is a \emph{Recursive Tree Process} (RTP) as defined in \cite{AB05}. Note that since by Theorem~\ref{T:frozen}, the joint law of $(\Om,\F)$ is uniquely determined, the same is true for the law of the RTP $(\om_\ibf,Y_\ibf)_{\ibf\in\T}$. Following a definition from \cite{AB05}, one says that such an RTP is \emph{endogenous} if $Y_\wurz$ is measurable w.r.t.\ the \si-field generated by the collection of random variables $\Om=(\om_\ibf)_{\ibf\in\T}$. We make the following observation.

\bl[Endogeny and almost sure uniqueness]
Let $(\om_\ibf,Y_\ibf)_{\ibf\in\T}$ be the RTP defined above. Then the\label{L:endas} following claims are equivalent:
\begin{enumerate}
\item The RTP $(\om_\ibf,Y_\ibf)_{\ibf\in\T}$ is endogenous.
\item If $\F$ and $\F'$ solve (\ref{frozdef}) relative to the same $\Om$, and
  moreover $\F$ and $\F'$ are stationary, adapted, and respect the tree
  structure, then $\F=\F'$ a.s.
\end{enumerate}
\el

\bpro
Fix $\Om=(\om_\ibf)_{\ibf\in\T}$ and let $\F$ be a solution to (\ref{frozdef}) relative to $\Om$ that is stationary, adapted, and respects the tree structure. By Theorem~\ref{T:frozen}, such a solution exists (perhaps on an extended probability space) and the joint law of $(\Om,\F)$ is uniquely determined. Let $(\om_\ibf,Y_\ibf)_{\ibf\in\T}$ be the corresponding RTP defined by (\ref{YF}). Endogeny says that $Y_\wurz$ is measurable w.r.t.\ the \si-field generated by $\Om$. Since $(\om_{\jbf\ibf},Y_{\jbf\ibf})_{\ibf\in\T}$ is equally distributed with $(\om_\ibf,Y_\ibf)_{\ibf\in\T}$, endogeny implies that $Y_\jbf$ is measurable w.r.t.\ the \si-field generated by $\Om$ for each $\jbf\in\T$. This shows that endogeny is equivalent to the statement that $(Y_\ibf)_{\ibf\in\T}$ is measurable w.r.t.\ the \si-field generated by $\Om$. Since by (\ref{YF}) and Lemma~\ref{L:FY}, given $\Om$, the set $\F$ and collection of random variables $(Y_\ibf)_{\ibf\in\T}$ determine each other a.s.\ uniquely, this is in turn equivalent to the statement that $\F$ is measurable w.r.t.\ the \si-field generated by $\Om$. Equivalently, this says that the conditional law of $\F$ given $\Om$ is a delta-measure. This shows that (i) implies (ii). Conversely, if (i) does not hold, let us construct a random variable $\F'$ such that  $\F'$ is conditionally independent of $\F$
given $\Om$, moreover the conditional distributions of $\F$  and $\F'$ are the same if we condition on $\Om$. In particular, $\F'$ then also solves (\ref{frozdef}) relative to $\Om$ and is stationary, adapted, and respects the tree structure. Since the conditional law of $\F$ given $\Om$ is not a delta-measure, we then have $\F\neq\F'$ with positive probability, showing that (ii) does not hold.
\epro

It follows from Lemma~\ref{L:endas} that if the RTP
$(\om_\ibf,Y_\ibf)_{\ibf\in\T}$ is nonendogenous, then solutions to the frozen
percolation equation (\ref{frozdef}) are not almost surely unique.
We pose the converse implication as an open problem:
\begin{quote}
\textbf{Question~1} Does endogeny of the RTP $(\om_\ibf,Y_\ibf)_{\ibf\in\T}$
imply almost sure uniqueness of solutions to the frozen percolation equation
(\ref{frozdef})?
\end{quote}
In other words, Question~1 asks whether in part~(ii) of
Lemma~\ref{L:endas}, one can remove the conditions that $\F$ and $\F'$
are stationary, adapted, and respect the tree structure.

We now address the question of endogeny. To state our main result, we need one
technical lemma, which introduces a parameter $\tet^\ast$. Numerically, we
find that $\theta^* \approx 0.636$.

\bl[The critical parameter]
Let\label{L:tetdef} $g:(0,1)\to\R$ be defined as \begin{equation}
g(\tet):=2(1+\tet)
-\sum_{\ell=0}^{\infty}\frac{\tet^{2\ell}(1-\tet^2)}{2/(1+\tet)-\tet^{\ell}}.
\end{equation}
Then $g$ is a strictly decreasing continuous function that changes sign at a
point $\tet^\ast \in (\tfrac{1}{2},1)$.
\el

The following theorem is the main result of our paper. For $\tet=1$,
the result has been proved in \cite[Thm~12]{RST19} but the result is
new in the regime $0<\tet<1$.

\bt[Endogeny]
Let\label{T:endog} $\tet^\ast$ be as in Lemma~\ref{L:tetdef}.
Let $0<\tet\leq 1$, and for the set of possible freezing times
$\Xi_\tet$, let $(\Om,\F)$ be defined as in Theorem~\ref{T:frozen}. Let
$(\om_\ibf,Y_\ibf)_{\ibf\in\T}$ be the corresponding RTP of burning times
defined in (\ref{YF}). This RTP is endogenous for $0<\tet\leq\tet^\ast$ but
not for $\tet^\ast<\tet\leq 1$.
\et

In the Subsections \ref{S:scale} and \ref{S:scRD} below, we eleborate
a bit on our methods for proving Theorem~\ref{T:endog}. We use the
remainder of the present subsection to make a few additional comments
on Question~1 posed above.

Set $\F_0:=\emptyset$ and define inductively for $k\geq 1$
\be\label{Fk}
Y^k_\ibf:=\inf\big\{t\in\Xi:\ibf\percol{\T^t\beh\F_{k-1}}\infty\big\}
\quad(\ibf\in\T)\quand
\F_k:=\big\{\ibf\in\I:Y^k_{\ibf 1}\leq\tau_\ibf\big\},
\ee
with the usual convention that $\inf\emptyset:=\infty$. Then it is not hard to see that
\be\ba{r@{\ }l@{\quad}r@{\ }l}\label{incdec}
{\rm(i)}&\dis\F_{2n}\sub\F_{2n+1}\quad
&{\rm(ii)}&\dis\F_{2n+1}\supset\F_{2n+2},\\[5pt]
{\rm(iii)}&\dis\F_{2n}\sub\F_{2n+2}\quad
&{\rm(iv)}&\dis\F_{2n+1}\supset\F_{2n+3},
\ea\qquad(n\in\N).
\ee
Moreover, if $\F$ solves the frozen percolation equation (\ref{frozdef}), then
\be\label{Fpm2}
\F_{2n}\sub\F\sub\F_{2n+1}\qquad(n\in\N).
\ee
For the sets of possible burning times $\Xi_\tet$ with $0<\tet\leq 1$, it is possible to verify by calculation that $\F_2=\emptyset$ a.s.\ if and only if $\tet\geq 1/2$. In particular, if $\tet<1/2$, then there are points that must freeze in any solution to the frozen percolation equation (\ref{frozdef}). We conjecture that in fact, for any $\tet<1/2$, the sets $\bigcup_{n\in\N}\F_{2n}$ and $\bigcap_{n\in\N}\F_{2n+1}$ are a.s.\ equal and as a result, solutions to the frozen percolation equation (\ref{frozdef}) are a.s.\ unique for all $\tet<1/2$. Note that even if this conjecture is correct, it does not not fully settle Question~1, since the parameter $\tet^\ast$ from Lemma~\ref{L:tetdef} is strictly larger than $1/2$.

\subsection{Scale invariance}\label{S:scale}

We fix a set of possible burning
times $\Xi$, construct a frozen percolation process $(\Om,\F)$ as in
Theorem~\ref{T:frozen} and let $(\om_\ibf,Y_\ibf)_{\ibf\in\T}$ be the
corresponding RTP of burning times defined in (\ref{YF}). Conditional on
$\Om=(\om_\ibf)_{\ibf\in\T}$, let $(Y'_\ibf)_{\ibf\in\T}$ be an independent
copy of $(Y_\ibf)_{\ibf\in\T}$. Then endogeny is equivalent to the statement
that $Y_\wurz=Y'_\wurz$ a.s. An easy argument, which can be found in
\cite[Appendix~B]{MSS18}, shows that the joint law of
$(Y_\wurz,Y'_\wurz)$ solves the \emph{bivariate RDE}
\be\label{bivar_RDE}
(Y_\wurz,Y'_\wurz)\isd\big(\chi[\om](Y_1,Y_2),\chi[\om](Y'_1,Y'_2)\big),
\ee
where $(Y_1,Y'_1)$ and $(Y_2,Y'_2)$ are independent copies of
$(Y_\wurz,Y'_\wurz)$ and $\om$ is an independent uniformly distributed random
variable on $[0,1]\times\{1,2\}$. We define probability laws on $I^2$ by
\be\label{unnu}
\un\rho^{(2)}_\Xi:=\P\big[(Y_\wurz,Y'_\wurz)\in\,\cdot\,\big]
\quand
\ov\rho^{(2)}_\Xi:=\P\big[(Y_\wurz,Y_\wurz)\in\,\cdot\,\big].
\ee
The marginals of these measures are the measure $\rho_\Xi$ defined in Lemma~\ref{L:burnlaw}. General theory for RTPs yields the following:

\bp[Bivariate uniqueness]
The\label{P:bivar} following statements are equivalent:
\begin{enumerate}
\item The RTP $(\om_\ibf,Y_\ibf)_{\ibf\in\T}$ is endogenous.
\item $\un\rho^{(2)}_\Xi=\ov\rho^{(2)}_\Xi$.
\item The measure $\ov\rho^{(2)}_\Xi$ is the only solution of the bivariate RDE
  (\ref{bivar_RDE}) in the space of symmetric probability measures on $I^2$ with
  marginals given by $\rho_\Xi$.
\end{enumerate}
\ep

\bpro
The equivalence of (i) and (ii) follows immediately from the definitions in (\ref{unnu}), the equivalence of (i) and (iii) is proved in \cite[Thm~11]{AB05} (see also \cite[Thm~1]{MSS18}), and the implication (iii)$\volgt$(ii) is trivial.
\epro

Proposition~\ref{P:bivar} is our main tool for proving
Theorem~\ref{T:endog}, but in order to be able to successfully apply
Proposition~\ref{P:bivar}, we need one more idea. For a general set of
possible burning times $\Xi$, it is difficult to find all solutions of
the bivariate RDE (\ref{bivar_RDE}) in the space of symmetric
probability measures with marginals given by $\rho_\Xi$. For the
special sets $\Xi_\tet$ with $0<\tet\leq 1$, however, it turns out to
be sufficient to look only at scale invariant solutions of the
bivariate RDE. As we will explain below, this leads to a significant
simplification of the problem, which allows us to prove
Theorem~\ref{T:endog} for the sets $\Xi_\tet$, but not for general
$\Xi$.

It has been proved in \cite[Prop~9]{RST19} that the law of the MBBT is
invariant under a certain scaling relation. This is ultimately the consequence
of the fact that the MBBT is itself the scaling limit of near-critical
percolation on trees of finite degree. We will not repeat the scaling property
of the MBBT here but instead formulate scaling properties of solutions to the
RDE (\ref{uni_RDE}) and bivariate RDE (\ref{bivar_RDE}) that are consequences
of the scaling of the MBBT.

For each $t>0$, we define a scaling map $\psi_t:I\to I$ by
\be\label{psit}
\psi_t(y):=\left\{\ba{ll}
\dis t^{-1}y\quad&\mbox{ if }y\leq t,\\
\dis\infty\quad&\mbox{otherwise.}\ea\right.
\ee
We let $\Mi^{(1)}$ denote the space of all probability measures $\rho$ on
$I=[0,1]\cup\{\infty\}$ that satisfy $\rho\big([0,t]\big)\leq t$ for all $0\leq
t\leq 1$, and we define scaling maps $\Ga_t$ by
\be\label{Ga1}
\Ga_t\rho:=t^{-1}\rho\circ\psi_t^{-1}+(1-t^{-1})\de_\infty
\qquad(\rho\in\Mi^{(1)},\ t>0),
\ee
where $\de_\infty$ denotes the delta-measure at $\infty$. It is not hard to
see that $\Ga_t$ maps the space $\Mi^{(1)}$ into itself. In particular, for
$0<t<1$, the assumption $\rho\big([0,t]\big)\leq t$ guarantees that
$\Ga_t\rho$ puts nonnegative mass at $\infty$. The following lemma
says that the set of solutions to the RDE (\ref{uni_RDE}) is invariant under
the scaling maps $\Ga_t$. Below, $\rho_\Xi$ denotes the measure defined in
Lemma~\ref{L:burnlaw}.

\bl[Scale invariance of the RDE]
Let\label{L:scaleRDE} $\rho$ be a solution to the RDE (\ref{uni_RDE}). Then
$\rho\in\Mi^{(1)}$, and for each $t>0$, the measure $\Ga_t\rho$ is also a
solution to the RDE (\ref{uni_RDE}). In particular, if $\Xi$ is a relatively
closed subset of $(0,1]$, then 
\be\label{scaleRDE}
\Ga_t\rho_\Xi=\rho_{\Xi'}\quad\mbox{with}\quad
\Xi':=\{t^{-1}y:y\in\Xi\}\cap[0,1]\qquad(t>0).
\ee
\el

For the bivariate RDE, a result similar to Lemma \ref{L:scaleRDE} holds, which
we formulate now. We say that a probability measure on $I^2$ is
\emph{symmetric} if it is invariant under the map $(y_1,y_2)\mapsto(y_2,y_1)$.
Let $\Mi^{(2)}$ denote the space of all symmetric probability measures
$\rho^{(2)}$ on $I^2$ that satisfy
\be\label{Mi2def}
\rho^{(2)}\big(
[0,t]\times I\cup I\times[0,t]\big)\leq t
\qquad\forall 0\leq t\leq 1.
\ee
We define $\psi^{(2)}_t:I^2\to I^2$ by
$\psi^{(2)}_t(y,y'):=\big(\psi_t(y),\psi_t(y')\big)$ and
we define $\Ga^{(2)}_t:\Mi^{(2)}\to\Mi^{(2)}$ by
\be\label{Ga2}
\Ga^{(2)}_t\rho:=t^{-1}\rho\circ(\psi^{(2)}_t)^{-1}+(1-t^{-1})\de_{(\infty,\infty)}
\qquad(\rho\in\Mi^{(2)},\ t>0).
\ee
We will prove that $\rho \in \Mi^{(2)} $ indeed implies $\Ga^{(2)}_t\rho \in \Mi^{(2)}$ in Section \ref{S:scaleprf}.
With the above definitions, we have the following lemmas, which are analogous to
Lemma~\ref{L:scaleRDE}. The measures $\un\rho^{(2)}_\Xi$ and $\ov\rho^{(2)}_\Xi$
that occur in Lemma~\ref{L:scalenu} are defined in (\ref{unnu}).

\bl[Scale invariance of bivariate RDE]
Let\label{L:scalebivRDE} $\rho^{(2)}$ be a symmetric solution to the bivariate
RDE (\ref{bivar_RDE}). Then $\rho^{(2)}\in\Mi^{(2)}$, and for each $t>0$,
the measure $\Ga^{(2)}_t\rho^{(2)}$ is also a solution to (\ref{bivar_RDE}).
\el


\bl[Scale invariance of special solutions]
Let\label{L:scalenu} $\Xi\sub(0,1]$ be relatively closed. Then, for each $t>0$,
\be
\Ga^{(2)}_t\un\rho^{(2)}_\Xi=\un\rho^{(2)}_{\Xi'}
\quand
\Ga^{(2)}_t\ov\rho^{(2)}_\Xi=\ov\rho^{(2)}_{\Xi'}
\quad\mbox{with}\quad
\Xi':=\{t^{-1}y:y\in\Xi\}\cap[0,1].
\ee
\el

\subsection{Scale invariant solutions to the bivariate RDE}\label{S:scRD}

In the present subsection, we explain how scale invariance helps us prove our
main result Theorem~\ref{T:endog}. It follows from Lemma~\ref{L:scaleRDE},
and can also easily be checked by direct calculation using
formula (\ref{muXi}), that the measures $\rho_{\Xi_\tet}$ are invariant under
scaling by $\tet$, and hence also by $\tet^n$ for each $n\geq 0$. Likewise,
$\rho_{\Xi_1}$ is invariant under scaling by any $0<t\leq 1$, so we have
\be\label{rhosca}
\Ga_t\rho_{\Xi_\tet}=\rho_{\Xi_\tet}\qquad\big(0<\tet\leq 1,\ t\in\Xi_\tet).
\ee
Motivated by this, for $0<\tet\leq 1$, we let $\Mi^{(2)}_\tet$ denote the
space of probability measures $\rho^{(2)}$ on $I^2$ such that:
\begin{enumerate}
\item $\rho^{(2)}\in\Mi^{(2)}$,
\item the marginals of $\rho^{(2)}$ are given by $\rho_{\Xi_\tet}$,
\item $\Ga^{(2)}_t\rho^{(2)}=\rho^{(2)}$ for all $t\in\Xi_\tet$.
\end{enumerate}

Let $0<\tet\leq 1$, and for the set of possible freezing times $\Xi_\tet$, let
$(\Om,\F)$ be defined as in Theorem~\ref{T:frozen}. Let
$(\om_\ibf,Y_\ibf)_{\ibf\in\T}$ be the corresponding RTP of burning times
defined in (\ref{YF}). It follows from Proposition~\ref{P:bivar} and
Lemma~\ref{L:scalenu} that the RTP $(\om_\ibf,Y_\ibf)_{\ibf\in\T}$ is
endogenous if and only if $\ov\rho^{(2)}_{\Xi_\tet}$ is the only solution of the
bivariate RDE (\ref{bivar_RDE}) in the space $\Mi^{(2)}_\tet$.
In view of this, Theorem~\ref{T:endog} is implied by the following theorem.

\begin{theorem}[Scale invariant solutions of the bivariate RDE]
Let\label{the_theorem} $\tet^\ast$ be as in Lemma~\ref{L:tetdef}
and let $0<\tet\leq 1$. Then:
\begin{enumerate}
\item If $\theta \le \theta^*$ then $\ov\rho^{(2)}_{\Xi_\tet}$ is the only
  solution of the bivariate RDE (\ref{bivar_RDE}) in the space
  $\Mi^{(2)}_\tet$.
\item If $\theta^*<\theta$, then there exists a measure $\hat{\rho}^{(2)}
  \in \Mi^{(2)}_\tet$ with $\hat{\rho}^{(2)}\neq\ov\rho^{(2)}_{\Xi_\tet}$ that
  solves (\ref{bivar_RDE}).
\end{enumerate}
\end{theorem}

We call $\ov\rho^{(2)}_{\Xi_\tet}$ the \emph{diagonal} solution of the
bivariate RDE since it is concentrated on $\{(y,y):y\in I\}$ (see
(\ref{unnu})). In the special case $\tet=1$, Theorem~\ref{the_theorem} has
been proved in \cite[Thm~12]{RST19}, where it is moreover shown that the
bivariate RDE (\ref{bivar_RDE}) has precisely two solutions in the
space $\Mi^{(2)}_1$. We conjecture that this holds more generally.
In Remark~\ref{remark:conj_unique} below, we present numerical evidence for
the following conjecture.

\bcon[Uniqueness of the nondiagonal solution]
For\label{conj:unique} all $\tet^\ast<\tet\leq 1$, the measures
$\ov\rho^{(2)}_{\Xi_\tet}$ and $\un\rho^{(2)}_{\Xi_\tet}$ defined in
(\ref{unnu}) are the only solutions of the bivariate RDE (\ref{bivar_RDE}) in
the space $\Mi^{(2)}_\tet$.
\econ

The main advantage of scale invariance is that it reduces the number of
parameters. In general, we can characterise a measure on $[0,1]^2$ by its
distribution function, which is a real function of two variables. However,
using scale invariance, we can characterise a measure
$\rho^{(2)}\in\Mi^{(2)}_\tet$ using a real function of one variable only, see
Definition~\ref{def:rho->f} below. This significantly simplifies the
calculations.

We can in fact be a little more general. Generalizing the definition
above (\ref{muXi}), for $0<\tet<1$ and $0<\al\leq 1$, let us define
$\Xi_{\tet,\al}:=\{\al\tet^n:n\in\N\}$. Then Lemma~\ref{L:scaleRDE} implies
that $\rho_{\Xi_{\tet,\al}}=\Ga_{1/\al}\rho_{\Xi_\tet}$. Moreover,
Proposition~\ref{P:bivar} and Lemma~\ref{L:scalenu} imply that the RTP
corresponding\footnote{The precise definition of an RTP corresponding to a solution to an RDE can be found below formula (\ref{RDE}) below.} to $\rho_{\Xi_{\tet,\al}}$ is endogenous if and only if the RTP
corresponding to $\rho_{\Xi_\tet}$ is endogenous. Since this does not
conceptually add anything new, for simplicity, we have formulated our main
results only for the set of possible burning times $\Xi_\tet$.

\section{Frozen percolation on the MBBT}

\subsection{Existence and uniqueness in law}\label{S:uni}

In this subsection, we prove Theorem~\ref{T:frozen} and
Lemmas~\ref{L:perctime}, \ref{L:FY}, \ref{L:burnlaw}, and \ref{L:genRDE}.\med

\bpro[of Lemma~\ref{L:perctime}]
It suffices to prove the claim for $\ibf=\wurz$. Let $P:=\{t\in[0,1]:
\wurz\percol{\T^t\beh\A}\infty\}$. Similar to the definition in
(\ref{percol}), for any $\ibf,\kbf\in\T$ and $A\sub\T$, we write
$\ibf\percol{A}\kbf$ if there exist a $\jbf=j_1\cdots j_n\in\T$ such
that $\kbf=\ibf\jbf$ and
\be\label{percolij}
{\rm(i)}\quad j_{k+1}\leq\kappa_{\ibf j_1\cdots j_k}\quad(0\leq k<n)
\quand
{\rm(ii)}\quad \ibf j_1\cdots j_k\in A\quad(0\leq k\leq n).
\ee
By (\ref{Tt}), the set $T(\ibf):=\{t\in[0,1]:\ibf\in\T^t\}$ is closed for each $\ibf\in\T$, so for each finite $n$, the set
\be
P_n:=\big\{t\in[0,1]:\wurz\percol{\T^t\beh\A}\jbf
\mbox{ for some }\jbf\in\T\mbox{ with }|\jbf|=n\big\},
\ee
being a finite intersection and union of sets of the form $T(\ibf)$, is also closed. It follows that the same is true for $P=\bigcap_{n\geq 0}P_n$.
\epro

\bpro[of Lemma~\ref{L:FY}]
By Lemma~\ref{L:perctime}, for each $\ibf\in\T$, the set
\[
T(\ibf):=\big\{t\in[0,1]:\ibf\percol{\T^t\beh\F}\infty\big\}
\]
is a random closed subset of $[0,1]$. By Lemma~\ref{L:perc} below, $\P[\ibf\percol{\T^t\beh\F}\infty]\leq\P[\ibf\percol{\T^t}\infty]=t$ for all $t\in(0,1]$, so there a.s.\ exists a random $\eps>0$ such that $T(\ibf)\sub[\eps,1]$, i.e., $T(\ibf)$ is a random compact subset of $(0,1]$. Since $\Xi$ is a closed subset of $(0,1]$ this implies that on the event that $Y_\ibf\leq 1$, the infimum in (\ref{YF}) is in fact a minimum and $\ibf\percol{\T^t\beh\F}\infty$ for $t=Y_\ibf$.

Let $\ibf\in\I$. If $Y_{\ibf 1}>\tau_\ibf$, then clearly there exists no $t\in(0,\tau_\ibf]$ such that $\ibf 1\percol{\T^t\beh\F}\infty$, and hence by (\ref{frozdef}) $\ibf\not\in\F$. On the other hand, if $Y_{\ibf 1}\leq\tau_\ibf$, then by what we have just proved, setting $t:=Y_{\ibf 1}$ we have $t\in\Xi\cap(0,\tau_\ibf]$ and $\ibf 1\percol{\T^t\beh\F}\infty$, which by (\ref{frozdef}) shows that $\ibf\in\F$.
\epro

To prepare for the proof of Lemma~\ref{L:burnlaw}, we need a bit of
theory. Let ${\rm BV}$ denote the space of functions $F:\R\to\R$ that are
locally of bounded variation. For each $F\in{\rm BV}$, the right and
left limits $F(t+):=\lim_{s\down t}F(s)$ and $F(t-):=\lim_{s\up t}F(s)$ exist
for each $t\in\R$, and $F$ defines a signed measure $\di F$ on $\R$ by any of
the equivalent formulas
\be
\di F\big((s,t]\big)=F(t+)-F(s+)\quad(s<t)\quand
\di F\big([s,t)\big)=F(t-)-F(s-)\quad(s<t).
\ee
For $G,F\in{\rm BV}$, we let $G\di F$ denote the signed measure obtained by
weighting $\di F$ with the density $G$. For $F\in{\rm BV}$, we define
\be
\ov F(t):=\ha\big(F(t-)+F(t+)\big)\qquad(t\in\R).
\ee
It is well-known that a right-continuous function with left limits makes at
most countably many jumps, and hence $\ov F(t)\neq F(t)$ for at most countably
many values of $t$. We will need the following simple fact.

\bl[Product rule]
For\label{L:prodrul} $F,G\in{\rm BV}$, one has $FG\in{\rm BV}$ and
$\di(FG)=\ov F\di G+\ov G\di F$.
\el

\bpro
The statement is well-known if $F$ and $G$ are continuous. Therefore, since
our formula is linear in $F$ and $G$ and since each measure can be decomposed
into an atomic and nonatomic part, it suffices to prove the statement only
when $\di F$ and $\di G$ are purely atomic. Using again linearity and a simple
limit argument, it suffices to prove the statement only in the case that $\di
F=\de_s$ and $\di G=\de_t$ for some $s,t\in\R$. If $s\neq t$, the statement is
trivial. If $s=t$, then the statement follows from the observation that
\bc
\dis F(t+)G(t+)-F(t-)G(t-)
&=&\dis\ha\big(F(t+)+F(t-)\big)\big(G(t+)-G(t-)\big)\\[5pt]
&&\dis+\ha\big(G(t+)+G(t-)\big)\big(F(t+)-F(t-)\big).
\ec
\epro

We cite the following lemma from \cite[Lemma~38]{RST19}.

\bl[Integral formulation of RDE]
A\label{L:RDEint} probability measure $\rho$ on $I$ solves the RDE
(\ref{uni_RDE}) if and only if
\be\label{RDEint}
\int_{[0,t]}\rho(\di s)s=\rho\big([0,t]\big)^2\qquad\big(t\in[0,1]\big).
\ee
\el

The following lemma is just a simple rewrite of the previous one. Below, we
let $\mu\big|_A$ denote the restriction of a (signed) measure $\mu$ to a
measurable set $A$, defined as $\mu\big|_A(B):=\mu(A\cap B)$.

\bl[Differential formulation of RDE]
Let\label{L:RDEdif} $T$ denote the identity function $T(t):=t$ $(t\in\R)$. Assume that $F\in{\rm BV}$ is right-continuous and nondecreasing and satisfies $F(t)=0$ $(t<0)$, $F(t)=F(1)$ $(t>1)$, and
\be\label{RDEdif}
T\di F=2\ov F\di F.
\ee
Then there exists a unique solution $\rho$ to the RDE (\ref{uni_RDE})
such that
\be\label{rhoF}
\rho\big([0,t]\big)=F(t)\qquad\big(t\in[0,1]\big),
\ee
and each solution $\rho$ to the RDE (\ref{uni_RDE}) arises in this way.
\el

\bpro
Let $\rho$ be a solution of the RDE (\ref{uni_RDE}) and let $F:[0,1]\to\R$ be defined as in (\ref{rhoF}). Extend $F$ to a function in ${\rm BV}$ by setting $F(t):=0$ for $t<0$ and $F(t):=F(1)$ for $t>1$. Then by Lemma~\ref{L:RDEint}, $\int_{(0,t]}T\di F=F(t)^2$ $(t\in[0,1])$, which by Lemma~\ref{L:prodrul} implies that $F$ solves (\ref{RDEdif}).

Assume, conversely, that $F\in{\rm BV}$ is right-continuous and nondecreasing and satisfies $F(t)=0$ $(t<0)$ and (\ref{RDEdif}). Then clearly $F\geq 0$. Formula (\ref{RDEdif}) implies that for a.e.\ $t$ w.r.t.\ $\di F$, we have $\ov F(t)=\ha t$, which by the fact that $F\geq 0$ implies $F(t)\leq t$. It follows that setting $\rho([0,t]):=F(t)$ $(t\in[0,1])$ defines a subprobability measure on $[0,1]$, which can uniquely be extended to a probability measure on $[0,1]\cup\{\infty\}$. Lemma~\ref{L:prodrul} implies that $\int_{(0,t]}T\di F=F(t)^2$ $(t\in[0,1])$, so using the fact that $\rho(\{0\})=F(0)-F(0-)=0$ and Lemma~\ref{L:RDEint}, we conclude that $\rho$ solves the RDE (\ref{uni_RDE}).
\epro

Let $T\in{\rm BV}$ denote the identity function $T(t):=t$ $(t\in\R)$. For a given closed set $\Xi\sub\R$, we will be interested in right-continuous functions $F\in{\rm BV}$ that solve the differential equation
\be\label{RDEdif2}
{\rm(i)}\ T\di F=2\ov F\di F,\quad{\rm(ii)}\ \di F\big|_\Xi=\di F
\quad{\rm(iii)}\ F(t)\geq\ha t\quad(t\in\Xi).
\ee
Note that condition~(ii) says that the signed measure $\di F$ is concentrated on $\Xi$. Our first lemma says that the distance between two solutions of (\ref{RDEdif2}) is a nonincreasing function of time.

\bl[Distance between two solutions]
Let\label{L:Fdist} $\Xi\sub\R$ be closed and for
$i=1,2$, let $F_i\in{\rm BV}$ be right-continuous solutions to
the differential equation (\ref{RDEdif2}).  Then
$\big|F_1(t)-F_2(t)\big|\leq\big|F_1(s)-F_2(s)\big|$ $(s\leq t)$.
\el

\bpro
We observe that by (i), we have $\ov F_i(t)=\ha t$ for a.e.\ $t$ w.r.t.\ $\di
F_i$. In particular, $\ov F_i(t)=\ha t$ whenever $F_i(t-)\neq F_i(t)$, which
we can combine with condition~(iii) to get
\[
\mbox{(iii)'}\;\ \ov F_i(t)\geq\ha t\quad(t\in\Xi).
\]
We now use Lemma~\ref{L:prodrul} to calculate
\be\ba{l}\label{diF}
\ha\di(F_1-F_2)^2=(\ov F_1-\ov F_2)(\di F_1-\di F_2)\\[5pt]
\dis\quad=\ov F_1\di F_1-\ov F_1\di F_2-\ov F_2\di F_1+\ov F_2\di F_2
=(\ha T-\ov F_2)\di F_1+(\ha T-\ov F_1)\di F_2,
\ec
where in the last step we have used (i). Using moreover
(ii) and (iii)', we see that the right-hand side of (\ref{diF}) is nonpositive,
so the claim of the lemma follows by integration.
\epro

Let $\Fi$ denote the space of all right-continuous, nondecreasing functions $F:\R\to\R$ that satisfy $0\leq F(t)\leq 0\vee t$ $(t\in\R)$. In other words, these are the distribution functions of nonnegative measures $\di F$ on $\half$ that satisfy $\di F([0,t])\leq t$ for all $t\geq 0$. We equip $\Fi$ with a topology that corresponds to vague convergence of the measures $\di F$. Then $\Fi$ is a compact, metrisable space and $F_n\to F$ in the topology on $\Fi$ if and only if $F_n(t)\to F(t)$ for each continuity point $t$ of $F$. Our aim is to prove that each closed set $\Xi\sub\half$, there exists a unique $F\in\Fi$ that solves (\ref{RDEdif2}). Uniqueness follows from Lemma~\ref{L:Fdist}, so it remains to prove existence. We will use an approximation argument. We start by proving the statement for finite $\Xi$.

\bl[Finite sets]
For\label{L:finex} each finite set $\Xi\sub(0,\infty)$, there exists an $F\in\Fi$ that solves (\ref{RDEdif2}).
\el

\bpro
Let $\Xi=\{t_1,\ldots,t_n\}$ with $0=:t_0<t_1<\cdots<t_n$. We inductively define $F$ so that it is constant on each of the intervals $(-\infty,t_1)$, $[t_1,t_2)$,\ldots $[t_{n-1},t_n)$, and $[t_n,\infty)$, satisfies $F(0)=0$, and
\be\label{Fkind}
F(t_k):=F(t_{k-1})\vee\big(t_k-F(t_{k-1})\big)\qquad(1\leq k\leq n).
\ee
Note that the average of $F(t_{k-1})$ and $t_k-F(t_{k-1})$ is $\ha t_k$, so their maximum is $\geq\ha t_k$. In view of this, $F$ clearly satisfies (\ref{RDEdif2})~(ii) and (iii). Moreover, for each $1\leq k\leq n$, we have either $F(t_k)=F(t_{k-1})$ or $F(t_k)=t_k-F(t_{k-1})$. In either case,
\be
t_k\big(F(t_k)-F(t_{k-1})\big)=2\cdot\ha\big(F(t_k)+F(t_{k-1})\big)\big(F(t_k)-F(t_{k-1})\big),
\ee
which shows that $F$ satisfies (\ref{RDEdif2})~(i). It is clear that $F$ is right-continuous, nonnegative, and nondecreasing, and by induction (\ref{Fkind}) also implies that $F(t)\leq t$ for all $t\geq 0$, showing that $F\in\Fi$.
\epro

Let $d$ be any metric generating the topology on $[0,\infty]$ and let $\Ki[0,\infty]$ denote the space of all closed subsets of $[0,\infty]$. For each $A\in\Ki[0,\infty]$ and $\eps>0$, we set
\be\label{Hau1}
A_\eps:=\big\{t\in[0,\infty]:d(t,A)<\eps\big\}
\quad\mbox{where}\quad
d(t,A):=\inf_{s\in A}d(t,s).
\ee
We equip $\Ki[0,\infty]$ with the \emph{Hausdorff metric}
\be\label{Hau2}
d_{\rm H}(A,B):=\inf\big\{\eps>0:A\sub B_\eps\mbox{ and }B\sub A_\eps\big\}.
\ee
By \cite[Lemma~B.1]{SSS14}, the topology generated by $d_{\rm H}$ does not depend on the choice of the metric $d$ generating the topology on $[0,\infty]$. The following lemma lists some elementary properties of the space $\Ki[0,\infty]$.

\bl[Properties of the Hausdorff metric]
The\label{L:Haus} space $\Ki[0,\infty]$ is compact and the set of all finite subsets of $(0,\infty)$ is dense in $\Ki[0,\infty]$.
\el

\bpro
Since $[0,\infty]$ is homeomorphic to $[0,1]$, we may equivalently show that $\Ki[0,1]$ is compact and the set of all finite subsets of $(0,1)$ is dense in $\Ki[0,1]$, where the Hausdorff metric on $\Ki[0,1]$ is defined in the same way as in (\ref{Hau1})--(\ref{Hau2}), with $d(x,y):=|x-y|$ the usual metric on $[0,1]$. The fact that $\Ki(E)$ is compact if $E$ is compact is well-known, see, e.g., \cite[Lemma~B.4]{SSS14}. If $\Xi\sub[0,1]$ is closed, then it is easy to see that the sets $\Xi_n:=\{k/n:1<k<n,\ d(k/n,\Xi)\leq 1/n\}$ converge to $\Xi$ in the Hausdorff metric. This shows that the set of finite subsets of $(0,1)$ is dense in $\Ki[0,1]$.
\epro

Our next lemma will allow us to construct solutions to (\ref{RDEdif2}) for general $\Xi$ by approximation with finite $\Xi$.

\bl[Limits of solutions]
Let\label{L:limsol} $F,F_n\in\Fi$ and $\Xi_n,\Xi\in\Ki[0,\infty]$ satisfy $F_n\to F$ and $\Xi_n\to\Xi$. Assume that $F_n$ solves (\ref{RDEdif2}) relative to $\Xi_n\cap\half$ for each $n$. Then $F$ solves (\ref{RDEdif2}) relative to $\Xi\cap\half$.
\el

\bpro
Recall that $F_n\to F$ means that $\di F_n\to\di F$ vaguely, or equivalently, $F_n(t)\to F(t)$ for each continuity point $t$ of $F$. Since $T$ is a continuous function, the vague convergence $\di F_n\to\di F$ implies that also $T\di F_n\to T\di F$ vaguely. By Lemma~\ref{L:prodrul}, $2\ov F\di F=\di F^2$. Now if $F_n(t)\to F(t)$ for each continuity point $t$ of $F$, then also  $F^2_n(t)\to F^2(t)$ for each continuity point $t$ of $F^2$, so taking the value limit on the left- and right-hand sides of the equation, we see that $F$ solves (\ref{RDEdif2}) ~(i). Since $\Xi_n\to\Xi$, we easily obtain that $\di F$ is concentrated on $\Xi_\eps$ for each $\eps>0$,  and hence $F$ satisfies (\ref{RDEdif2}) ~(ii). To see that $F$ also satisfies (\ref{RDEdif2}) ~(iii), fix $t\in\Xi$. Since $\Xi_n\to\Xi$ we can find $t_n\in\Xi_n$ such that $t_n\to t$. Then for each $s>t$, we have $F_n(s)\geq F_n(t_n)\geq\ha t_n$ for all $n$ large enough. Taking the limit, it follows that $F(s)\geq\ha t$ for each $s\geq t$ that is a continuity point of $F$, and hence $F(t)\geq\ha t$ by right-continuity.
\epro

We can now prove existence of solutions to (\ref{RDEdif2}) for general $\Xi$.

\bl[Existence of solutions to the RDE]
For\label{L:exist} each closed set $\Xi\sub\half$, there exists a function $F\in\Fi$ that solves (\ref{RDEdif2}).
\el

\bpro
By Lemma~\ref{L:Haus}, for each closed $\Xi\sub[0,\infty]$, there exist finite $\Xi_n\sub(0,\infty)$ such that $\Xi_n\to\Xi$. By Lemma~\ref{L:finex}, for each $n$ there exists an $F_n\in\Fi$ so that $F_n$ solves (\ref{RDEdif2}) relative to $\Xi_n$. Since $\Fi$ is compact, by going to a subsequence if necessary we can assume that $F_n\to F$ for some $F\in\Fi$. Then Lemma~\ref{L:limsol} tells us that $F$ solves (\ref{RDEdif2}) relative to $\Xi\cap\half$.
\epro

Before we prove Lemma~\ref{L:burnlaw}, we recall the general definition of an
RTP. Let $\T$ denote the space of all finite words $\ibf=i_1\cdots i_n$
$(n\geq 0)$ made up from the alphabet $\{1,\ldots,d\}$, where $d\geq 1$ is
some fixed integer. All previous notation involving the binary tree
generalizes in a straightforward manner to the $d$-ary tree $\T$. Let $I$ and
$\Om$ be Polish spaces, let $\ga:\Om\times I^d\to I$ be a measurable function,
and let $(\om_\ibf)_{\ibf\in\T}$ be i.i.d.\ $\Om$-valued random variables. Let
$\nu$ be a probability law on $I$ that solves the Recursive Distributional
Equation (RDE)
\be\label{RDE}
X_\wurz\isd\ga[\om_\wurz](X_1,\ldots,X_d),
\ee
where $\isd$ denotes equality in distribution, $X_\wurz$ has law $\nu$, and
$X_1,\dots, X_d$ are copies of $X_\wurz$, independent of each other and of
$\om_\wurz$. A simple argument based on Kolmogorov's extension theorem (see
\cite[Lemma~1.9]{MSS20}) tells us that the i.i.d.\ random variables
$(\om_\ibf)_{\ibf\in\T}$ can be coupled to $I$-valued random
variables $(X_\ibf)_{\ibf\in\T}$ in such a way that:
\begin{enumerate}
\item For each finite rooted subtree $\U\sub\T$, the r.v.'s
  $(X_\ibf)_{\ibf\in\pa\U}$ are i.i.d.\ with common law $\nu$ and independent
  of $(\om_\ibf)_{\ibf\in\U}$.
\item $\dis X_\ibf=\ga[\om_\ibf](X_{\ibf 1},\ldots,X_{\ibf d})\qquad(\ibf\in\T)$.
\end{enumerate}
Moreover, these conditions uniquely determine the joint law of
$(\om_\ibf,X_\ibf)_{\ibf\in\T}$. We call the latter the
\emph{Recursive Tree Process} (RTP) corresponding to the maps $\ga$ and
solution $\nu$ of the RDE (\ref{RDE}).\med

\bpro[of Lemma~\ref{L:burnlaw}]
Let $\Xi\sub(0,1]$ be relatively closed and let $\ov\Xi:=\Xi\cup\{0\}$. By Lemma~ \ref{L:exist}, there exists a solution $F\in\Fi$ of the differential equation (\ref{RDEdif2}) relative to $\ov\Xi$. Set $\rho_\Xi([0,t]):=F(t)$ $(t\in[0,1])$ and $\rho_\Xi(\{\infty\}):=1-F(1)$ (which is $\geq 0$ since $F(1)\leq 1$ by the definition below (\ref{diF}) of the class $\Fi$) and observe that $\rho_\Xi(\{0\})=0$. Then $\rho_\Xi$ is a probability measure on $I$ that satisfies conditions (ii) and (iii) of Lemma~\ref{L:burnlaw}, and by Lemma~\ref{L:RDEdif} also condition~(i). Assume, conversely, that $\rho_\Xi$ satisfies conditions (i)--(iii) of Lemma~\ref{L:burnlaw}, and set $F(t):=\rho_\Xi([0,t])$ $(t\in[0,1])$, $F(t):=0$ $(t<0)$, $F(t):=F(1)$ $(t>1)$. Then by Lemma~\ref{L:RDEdif}, $F$ solves the differential equation (\ref{RDEdif2}) subject to the initial condition $F(t):=0$ $(t<0)$. By Lemma~\ref{L:Fdist}, these conditions uniquely determine $F$ and hence also $\rho_\Xi$.

Assume that $\F$ solves the frozen percolation equation (\ref{frozdef}) for
the set of possible freezing times $\Xi$ and that $\F$ is stationary, adapted,
and respects the tree structure. Generalising (\ref{YF}), for any $\De\sub(0,1]$
that is relatively closed, we set
\be\label{YFwhole}
Y_\ibf^\De:=\inf\big\{t\in\De:
\ibf\percol{\T^t\beh\F}\infty\big\}\qquad(\ibf\in\T).
\ee
Then in particular,  $Y_\ibf^{\Xi}$ is the burning time $Y_\ibf$ defined in (\ref{YF}). As in (\ref{Om}), we write $\om_\ibf=(\tau_\ibf,\kappa_\ibf)$ $(\ibf\in\T)$.Since $\F$ is stationary, adapted, and respects the tree structure, the random
variables $(\om_\ibf,Y^\Xi_\ibf)_{\ibf\in\T}$ form an RTP corresponding to the
map $\chi$ in (\ref{chi_def}) and some solution $\rho_\Xi$ to the RDE
(\ref{uni_RDE}). To complete the proof, we need to show that $\rho_\Xi$
also satisfies conditions (ii) and (iii) of Lemma~\ref{L:burnlaw}.
Since $\rho_\Xi$ is the law of $Y^\Xi_\ibf$ $(\ibf\in\T)$,
it clearly satisfies condition~(ii) of Lemma~\ref{L:burnlaw}. To also
prove (iii), we use that by Lemma~\ref{L:FY}, we have
$\F=\big\{\ibf\in\I:Y^\Xi_{\ibf 1}\leq\tau_\ibf\big\}$,
which allows us to  apply \cite[Prop.~39]{RST19}, which tells us that
\be
\P\big[Y^{(0,1]}_\ibf\leq t\big]=F(t)\vee\big(t-F(t)\big)
\qquad\big(\ibf\in\T,\ t\in[0,1]\big),
\ee
where $F(t):=\rho_\Xi\big([0,t]\big)$ $\big(t\in[0,1]\big)$. Since
\be\label{YXiY}
Y^\Xi_\ibf=\inf\big\{t\in\Xi:t\geq Y^{(0,1]}_\ibf\big\}\qquad(\ibf\in\T),
\ee
it follows that
\be\label{FtF}
F(t)=\rho_\Xi\big([0,t]\big)=\P\big[Y^\Xi_\ibf\leq t\big]
=\P\big[Y^{(0,1]}_\ibf\leq t\big]=F(t)\vee\big(t-F(t)\big)
\qquad(t\in\Xi),
\ee
where the two probabilities are equal by (\ref{YXiY}) and the fact that $t\in\Xi$. This proves that $\rho_\Xi$ satisfies condition~(iii) of Lemma~\ref{L:burnlaw}.
\epro

\bpro[of Lemma~\ref{L:genRDE}]
If $\rho$ solves the RDE (\ref{uni_RDE}), then by Lemma~\ref{L:RDEdif}, the
function $F\in{\rm BV}$ defined in (\ref{rhoF}) is right-continuous and
nondecreasing with $F(0)=0$ and satisfies (\ref{RDEdif}). Let
$\Xi:={\rm supp}(\di F)\cap(0,1]$. Then (\ref{RDEdif}) implies that
$\ov F(t)=\ha t$ for a.e.\ $t$ w.r.t.\ $\di F$. Since $F$ is right-continuous
with left limits, this implies that $\ov F(t)=\ha t$ for all $t\in\Xi$, and
hence $F(t)\geq\ha t$ for all $t\in\Xi$. It follows that $\rho$ satisfies
conditions (i)--(iii) of Lemma~\ref{L:burnlaw} and hence $\rho=\rho_\Xi$.
\epro

The following lemma settles the existence part of Theorem~\ref{T:frozen}.

\bl[Frozen points]
Let\label{L:YFY} $\Xi\sub(0,1]$ be closed w.r.t.\ the relative topology of
$(0,1]$, let $(\om_\ibf,Y_\ibf)_{\ibf\in\T}$ be the RTP corresponding to the
solution $\rho_\Xi$ to the RDE (\ref{uni_RDE}) defined in Lemma~\ref{L:burnlaw},
and let $\F$ be defined by (\ref{FY}). Then $\F$ solves the frozen
percolation equation (\ref{frozdef}) for the set of possible freezing times
$\Xi$ and $\F$ is stationary, adapted, and respects the tree structure.
Moreover, the $Y_\ibf$ are given by (\ref{YF}).
\el

\bpro
It follows from the properties of an RTP that $\F$, defined by (\ref{FY}), is
stationary, adapted, and respects the tree structure. The inductive relation
(\ref{Yind}) implies that if $Y_\ibf<\infty$, then there exist
$(j_k)_{k\geq 1}$ such that $\ibf j_1\cdots j_n$ is a legal descendant of
$\ibf j_1\cdots j_{n-1}$ and $Y_\ibf=Y_{\ibf j_1\cdots j_n}$ for all $n\geq
1$. For all $n\geq 0$ such that $\kappa_{\ibf j_1\cdots j_n}=1$, the fact that
$Y_{\ibf j_1\cdots j_n}<\infty$ and (\ref{Yind}) moreover imply that
$Y_{\ibf j_1\cdots j_n1}>\tau_\ibf$. Therefore, we have that
\be
\ibf\percol{\T^t\beh\F}\infty\quad\mbox{if }t=Y_\ibf<\infty.
\ee
Since $Y_\ibf$ takes values in $\Xi\cup\{\infty\}$, it follows that
\be\label{Ygeqinf}
Y_\ibf\geq\inf\big\{t\in\Xi:
\ibf\percol{\T^t\beh\F}\infty\big\}\qquad(\ibf\in\T).
\ee
To prove that this is actually an equality, let $Y'_\ibf$ denote the
right-hand side of (\ref{Ygeqinf}). Since $\F$ is
stationary, adapted, and respects the tree structure, as pointed out
in Section~\ref{S:burn}, the random variables
$(\om_\ibf,Y'_\ibf)_{\ibf\in\T}$ form an RTP corresponding to the map $\chi$ in
(\ref{chi_def}) and some solution $\rho$ to the RDE (\ref{uni_RDE}). By
Lemma~\ref{L:burnlaw}, $\rho=\rho_\Xi$, so $Y'_\ibf$ and $Y_\ibf$ are equal in
law, which by (\ref{Ygeqinf}) implies that they are a.s.\ equal.
This proves that the $Y_\ibf$ are given by (\ref{YF}). By assumption,
$\F$ is defined by (\ref{FY}). Inserting
(\ref{YF}) into (\ref{FY}), we see that $\F$ solves the
frozen percolation equation (\ref{frozdef}).
\epro

\bpro[of Theorem~\ref{T:frozen}]
Lemma~\ref{L:YFY} proves existence of a solution $\F$ of the frozen
percolation equation (\ref{frozdef}) for the set of possible freezing times
$\Xi$ that is stationary, adapted, and respects the tree structure.
It remains to prove uniqueness in law. Set $\om_\ibf:=(\tau_\ibf,\kappa_\ibf)$
$(\ibf\in\T)$ and let $\Om=(\om_\ibf)_{\ibf\in\T}$. Let $Y_\ibf$ $(\ibf\in\T)$
be the burning times defined in (\ref{YF}). Since $\F$ is stationary, adapted,
and respects the tree structure, as pointed out
in Section~\ref{S:burn}, the random variables
$(\om_\ibf,Y_\ibf)_{\ibf\in\T}$ form an RTP corresponding to the map $\chi$ in
(\ref{chi_def}) and some solution $\rho$ to the RDE (\ref{uni_RDE}). By
Lemma~\ref{L:burnlaw}, $\rho=\rho_\Xi$, and hence by
\cite[Lemma~1.9]{MSS20} the law of $(\om_\ibf,Y_\ibf)_{\ibf\in\T}$ is uniquely
determined. By Lemma~\ref{L:FY}, this implies that the joint law of
$(\Om,\F)$ is also uniquely determined.
\epro

\subsection{Scale invariance}\label{S:scaleprf}

In this subsection, we prove Lemmas \ref{L:scaleRDE}, \ref{L:scalebivRDE}, and \ref{L:scalenu} about invariance of solutions of the (bivariate) RDE under the scaling maps $\Ga_t$ and $\Ga^{(2)}_t$. We will generalise a bit and define scaling maps $\Ga^{(n)}_t$ for any $1\leq n\leq\infty$, where the case $n=\infty$ will play an important role in the proof of  Lemma~\ref{L:scalenu}.

Recall the definition of the scaling maps $\psi_t:I\to I$ $(t>0)$ in (\ref{psit}). For $t>0$, we define a cut-off map $c_t:I\to I$ by
\be\label{ct}
c_t(y):=\left\{\ba{ll}
y\quad&\mbox{if }y\leq t,\\[5pt]
\infty\quad&\mbox{otherwise.}
\ea\right.
\ee
Note that $c_t$ is the identity map when $t\geq 1$. It is easy to check that
\be\label{psic}
\psi_{1/t}\circ\psi_t=c_t\qquad(t>0).
\ee
For $1\leq n<\infty$, we write $[n]:=\{1,\ldots,n\}$ and we set $[\infty]:=\N_+$. We denote a generic element of $I^n$ by $\vec y=(y^k)_{k\in[n]}$ and we define $\psi^{(n)}_t:I^n\to I^n$ and $c^{(n)}_t:I^n\to I^n$ in a coordinatewise way by $\psi^{(n)}_t(\vec y):=\big(\psi_t(y^k)\big)_{k\in[n]}$ and $c^{(n)}_t(\vec y):=\big(c_t(y^k)\big)_{k\in[n]}$.

We say that a probability measure on $I^n$ is \emph{symmetric} if it is invariant under a permutation of the coordinates. Generalising the definitions of $\Mi^{(1)}$ and $\Mi^{(2)}$ in Subsection~\ref{S:scale}, for any $0<t\leq 1$ and $1\leq n\leq\infty$, we let $\Mi^{(n)}$ denote the space of symmetric probability measures $\rho^{(n)}$ on $I^n$ such that
\be\label{infyk}
\rho^{(n)}\big(J^n[t]\big)\leq t\qquad(0<t\leq 1)
\quad\mbox{with}\quad
J^n[t]:=\big\{\vec y\in I^n:\exists k\in[n]\mbox{ s.t.\ }y^k\leq t\big\}.
\ee
Note that $J^n[1]=I^n\beh\{\vec\infty\}$, where $\vec\infty$ denotes the element $\vec y\in I^n$ with $y^k:=\infty$ for all $k\in[n]$. Generalising the definitions of $\Ga_t^{(1)}$ and $\Ga_t^{(2)}$ in Subsection~\ref{S:scale}, for each $1\leq n\leq\infty$ and $t>0$, we define
\be\label{Gan}
\Ga^{(n)}_t\rho^{(n)}:=t^{-1}\rho^{(n)}\circ(\psi^{(n)}_t)^{-1}
+(1-t^{-1})\de_{\vec\infty}\qquad\big(\rho^{(n)}\in\Mi^{(n)}).
\ee
We also define cut-off maps $C^{(n)}_t$ by
\be\label{Can}
C^{(n)}_t\rho^{(n)}:=\rho^{(n)}\circ(c^{(n)}_t)^{-1}\qquad\big(\rho^{(n)}\in\Mi^{(n)})
\ee
and in particular set $C_t:=C^{(1)}_t$. Finally,  for all $1\leq n\leq\infty$, we define a map $T^{(n)}$ acting on probability measures on $I^n$ by
\be\label{Tndef}
T^{(n)}\rho^{(n)}:=
\mbox{ the law of }\big(\chi[\om](Y^k_1,Y^k_2)\big)_{k\in[n]},
\ee
where $(Y^k_1)_{k\in[n]}$ and $(Y^k_2)_{k\in[n]}$ are independent random variables with law $\rho^{(n)}$, and $\om$ is an independent random variable that is uniformly distributed on $[0,1]\times\{1,2\}$. We call the equation
\be\label{nvar}
T^{(n)}\rho^{(n)}=\rho^{(n)}
\ee
the \emph{$n$-variate RDE}. In particular, for $n=1$ this is the RDE (\ref{uni_RDE}) and for $n=2$ this is bivariate RDE (\ref{bivar_RDE}). The following lemma, which will be proved below, shows that all these maps are well-defined on the space $\Mi^{(n)}$.

\bl[Maps are well-defined]
For\label{L:mapdef} each $1\leq n\leq\infty$ and $t>0$, the maps $\Ga^{(n)}_t$, $C^{(n)}_t$, and $T^{(n)}$ map the space $\Mi^{(n)}$ into itself.
\el

The following lemma says that as long as we are interested in symmetric solutions of the $n$-variate RDE (\ref{nvar}), it suffices to look for solutions in the space $\Mi^{(n)}$.

\bl[Solutions to the RDE are scalable]
If\label{L:scalable} a symmetric probability measure $\rho^{(n)}$ on $I^n$ solves the $n$-variate RDE (\ref{nvar}), then $\rho^{(n)}\in\Mi^{(n)}$.
\el

The following lemma is the central result of this subsection.

\bl[Commutation relation]
For\label{L:commut} each $1\leq n\leq\infty$, one has
\be\label{commut}
\Ga^{(n)}_tT^{(n)}\rho^{(n)}=tT^{(n)}\Ga^{(n)}_t\rho^{(n)}
+(1-t)\Ga^{(n)}_t\rho^{(n)}
\qquad(t>0,\ \rho^{(n)}\in\Mi^{(n)}).
\ee
\el

We first show how Lemmas \ref{L:mapdef}--\ref{L:commut} imply Lemmas \ref{L:scaleRDE} and \ref{L:scalebivRDE}, and then prove Lemmas \ref{L:mapdef}--\ref{L:commut}. We start by proving a more general statement.\med

\bl[Scale invariance of $n$-variate RDE]
Let\label{L:scalenRDE} $1\leq n\leq\infty$ and let $\rho^{(n)}$ be a
symmetric solution to the $n$-variate RDE (\ref{nvar}). Then
$\rho^{(n)}\in\Mi^{(n)}$, and for each $t>0$, the measure $\Ga^{(n)}_t\rho^{(n)}$
is also a solution to (\ref{nvar}).
\el

\bpro
Let $1\leq n\leq\infty$ and let $\rho^{(n)}$ be a symmetric solution
to the $n$-variate RDE (\ref{nvar}). Then $\rho^{(n)}\in\Mi^{(n)}$
by Lemma~\ref{L:scalable}. Moreover, for each $t>0$, Lemma~\ref{L:commut}
and the fact that $T^{(n)}\rho^{(n)}=\rho^{(n)}$ imply that
\be
\Ga^{(n)}_t\rho^{(n)}=tT^{(n)}\Ga^{(n)}_t\rho^{(n)}
+(1-t)\Ga^{(n)}_t\rho^{(n)}
\ee
which shows that $T^{(n)}\Ga^{(n)}_t\rho^{(n)}=\Ga^{(n)}_t\rho^{(n)}$, i.e.,
the measure $\Ga^{(n)}_t\rho^{(n)}$ solves the $n$-variate RDE (\ref{nvar}).
\epro

\bpro[of Lemmas \ref{L:scaleRDE} and \ref{L:scalebivRDE}]
Most of the statements of Lemmas \ref{L:scaleRDE} and \ref{L:scalebivRDE} follow by specialising Lemma~\ref{L:scalenRDE} to the cases $n=1$ and $n=2$, respectively. Apart from this, we only need to prove (\ref{scaleRDE}). Let $\Xi\sub(0,1]$ be relatively closed and let $\Xi'$ be as in (\ref{scaleRDE}). Then $\Ga_t\rho_\Xi$ solves the RDE (\ref{uni_RDE}) by Lemma~\ref{L:scalenRDE}. Using the definition of $\Ga_t$, it is easy to see that the fact that $\rho_\Xi$ has properties (ii) and (iii) of Lemma~\ref{L:burnlaw} implies that $\Ga_t\rho_\Xi$ has these same properties with $\Xi$ replaced by $\Xi'$, i.e., $\Ga_t\rho_\Xi$ is concentrated on $\Xi'\cup\{\infty\}$, and $\Ga_t\rho_\Xi\big([0,t']\big)\geq\ha t'$ for all $t'\in\Xi'$. Now Lemma~\ref{L:burnlaw} allows us to identify $\Ga_t\rho_\Xi$ as $\rho_{\Xi'}$.
\epro

We now provide the proofs of Lemmas \ref{L:mapdef}--\ref{L:commut}.\med

\bpro[of Lemma~\ref{L:mapdef} (partially)]
Let $\rho^{(n)}\in\Mi^{(n)}$. It is clear that the right-hand side of (\ref{Gan}) defines a signed measure that is symmetric with respect to a permutation of the coordinates and that satisfies $\Ga^{(n)}_t\rho^{(n)}(I^n)=1$. We observe that by (\ref{psit}),
\be\ba{l}
\dis(\psi^{(n)}_t)^{-1}(J^n[s])
=\big\{\vec y\in I^n:\exists k\in[n]\mbox{ s.t.\ }\psi_t(y^k)\leq s\big\}\\[5pt]
\quad\dis=\big\{\vec y\in I^n:\exists k\in[n]\mbox{ s.t.\ }y^k\leq t
\mbox{ and }t^{-1}y^k\leq s\big\}
=\big\{\vec y\in I^n:\exists k\in[n]\mbox{ s.t.\ }y^k\leq st\wedge t\big\},
\ec
and hence
\be\label{MiMi}
\Ga^{(n)}_t\rho^{(n)}(J^n[s])
=t^{-1}\rho^{(n)}\circ(\psi^{(n)}_t)^{-1}(J^n[s])
=t^{-1}\rho^{(n)}(J^n[st\wedge t])\leq t^{-1}(st\wedge t)\leq s.
\ee
for each $t>0$ and $0<s\leq 1$. Applying this with $s=1$ and using the fact that $I^n=J^n[1]\cup\{\vec\infty\}$ we see that $\Ga^{(n)}_t\rho^{(n)}$ is a probability measure. More generally, (\ref{MiMi}) shows that $\Ga^{(n)}_t$ maps the space $\Mi^{(n)}$ into itself.

It is clear that $C^{(n)}_t\rho^{(n)}$, defined in (\ref{Can}) , is a (symmetric) probability measure on $I^n$ whenever $\rho^{(n)}$ is. If moreover $\rho^{(n)}\in\Mi^{(n)}$, then
\be
C^{(n)}_t\rho^{(n)}(J^n[s])
=\rho^{(n)}\circ(c^{(n)}_t)^{-1}(J^n[s])
=\rho^{(n)}(J^n[s\wedge t])\leq s
\ee
for each $t>0$ and $0<s\leq 1$, which shows that $C^{(n)}_t$ maps the space $\Mi^{(n)}$ into itself.

This proves the claims for $\Ga^{(n)}_t$ and $C^{(n)}_t$. We postpone the proof of claim for $T^{(n)}$ until the proof of Lemma~\ref{L:commut}, where it will follow as a side result of the main argument.
\epro

The proof of Lemma~\ref{L:scalable} uses the following simple lemma, which we cite from \cite[Lemma~8]{RST19}.

\bl[Percolation probability]
One\label{L:perc} has
$\dis\P\big[\wurz\percol{\T^t}\infty\big]=t\quad(0\leq t\leq 1)$.
\el

\bpro[of Lemma~\ref{L:scalable}]
We will prove the following, somewhat stronger statement. Let $1\leq
n\leq\infty$ and let $\rho^{(n)}$ be a solution to the $n$-variate RDE
(\ref{nvar}). Then we will show that $\rho^{(n)}$ satisfies
(\ref{infyk}). In particular, if $\rho^{(n)}$ is symmetric, this
implies that $\rho^{(n)}\in\Mi^{(n)}$.

Let $\rho^{(n)}$ be a solution to the $n$-variate RDE (\ref{nvar}) and for
$k\in[n]$, let $\rho_k$ denote the $n$-th marginal of $\rho^{(n)}$. It is
clear from (\ref{nvar}) that $\rho_k$ solves the RDE (\ref{uni_RDE}), so
by Lemma~\ref{L:genRDE}, for each $k\in[n]$, there exists a relatively closed
set $\Xi_k\sub(0,1]$ such that $\rho_k=\rho_{\Xi_k}$. Since $\rho^{(n)}$
solves the $n$-variate RDE (\ref{nvar}), by \cite[Lemma~1.9]{MSS20},
we can construct an $n$-variate RTP
\be
\big(\om_\ibf,\vec Y_\ibf\big)_{\ibf\in\T}
\ee
where $\vec Y_\ibf=(Y^k_\ibf)_{k\in[n]}$ are $I^n$-valued random variables
such that $\vec Y_\ibf$ is inductively given in terms of $\vec Y_{\ibf 1}$ and
$\vec Y_{\ibf 2}$ as in (\ref{nvar}). In particular, for each $k\in[n]$,
$(\om_\ibf,Y^k_\ibf)_{\ibf\in\T}$ is an RTP corresponding to $\rho_k=\rho_{\Xi_k}$.
Let
\be\label{FYk}
\F^k=\big\{\ibf\in\I:Y^k_{\ibf 1}\leq\tau_\ibf\big\}\qquad(k\in[n]).
\ee
Then Lemma~\ref{L:YFY} tells us that
\be\label{YFk}
Y^k_\ibf:=\inf\big\{t\in\Xi:
\ibf\percol{\T^t\beh\F^k}\infty\big\}\qquad(\ibf\in\T,\ k\in[n]),
\ee
with the convention that $\inf\emptyset:=\infty$. Using Lemma~\ref{L:perc}, we
can now estimate
\be
\P\big[\inf_{k\in[n]}Y^k_\wurz\leq t\big]
\leq\P\big[\wurz\percol{\T^t}\infty\big]=t\qquad(0<t\leq 1),
\ee
which proves that $\rho^{(n)}$ satisfies (\ref{infyk}).
\epro

\bpro[of Lemma~\ref{L:commut}]
We will first prove (\ref{commut}) for $0<t\leq 1$, and on the way also
establish that $T^{(n)}$ maps $\Mi^{(n)}$ into itself, which is the missing part
of Lemma~\ref{L:mapdef} that still remains to be proved.
Fix $1\leq n\leq\infty$, $0<t\leq 1$, and $\rho^{(n)}\in\Mi^{(n)}$. Let $\vec
Y=(Y^k)_{k\in[n]}$ be a random variable with law $\rho^{(n)}$. It follows from
(\ref{infyk}) that we can couple $\vec Y$ to a Bernoulli random variable $B$
such that $\P[B=1]=t$ and $\P[B=1\,|\,\inf_{k\in[n]}Y^k\leq t]=1$. More
formally, there exists a probability measure $\mu$ on $I^n\times\{0,1\}$
whose first marginal is $\rho^{(n)}$, whose second marginal is the Bernoulli
distribution with parameter $t$, and that is concentrated on $\{(\vec
y,b):b=1\mbox{ if }\inf_{k\in[n]}y_k\leq t\}$. Such a measure $\mu$ is not
unique, but we fix one from now on. Let $(\vec Y_1,B_1)$ and $(\vec
Y_2,B_2)$ be independent random variables with law $\mu$. Then
\be\label{indB}
{\rm(i)}\quad\P\big[B_i=1\,\big|\,\inf_{k\in[n]}Y^k_i\leq t\big]=1,
\qquad
{\rm(ii)}\quad\P[B_i=1]=t.\qquad(i=1,2).
\ee
Let $\om=(\tau,\kappa)$ be an independent random variable that is uniformly
distributed on $[0,1]\times\{1,2\}$. We define
\be\ba{lr@{\,}c@{\,}l}\label{wurzdef}
\dis{\rm(i)}\quad&
\dis\vec Y_\wurz&:=&\dis\big(\chi[\om](Y^k_1,Y^k_2)\big)_{k\in[n]},\\[5pt]
\dis{\rm(ii)}\quad&\dis B_\wurz&:=&\dis1_{\{\kappa=1\}}1_{\{\tau\leq t\}}B_1
+1_{\{\kappa=2\}}(B_1\vee B_2).
\ec
The second definition is motivated by the heuristic idea that if $B_i$ is the event that the subtree rooted at $i$ percolates at time $t$ (neglecting the freezing), then $B_\wurz$ is the event that the whole tree percolates at time $t$. We claim that
\be\label{Bprop}
{\rm(i)}\quad\P\big[B_\wurz=1\,\big|\,\inf_{k\in[n]}Y^k_\wurz\leq t\big]=1,
\qquad
{\rm(ii)}\quad\P[B_\wurz=1]=t.
\ee
Indeed, if $Y^k_\wurz\leq t$ for some $k$, then by the definition of $\chi$ in
(\ref{chi_def}), a.s.\ either $\kappa=1$ and $\tau<Y^k_1\leq t$, or $\kappa=2$
and $Y^k_1\wedge Y^k_2\leq t$. In either case, it follows that $B_\wurz=1$,
proving part~(i) of (\ref{Bprop}). Part~(ii) follows by writing
\be
\P[B_\wurz=1]=\ha t\P[B_1=1]+\ha\P[B_1\vee B_2=1]=\ha t^2+\ha[1-(1-t)^2]=t.
\ee
We next claim that for each measurable subset $A\sub I^n$,
\be\ba{lr@{\,}c@{\,}l}\label{RGarep}
\dis{\rm(i)}\quad&\dis T^{(n)}\rho^{(n)}(A)
&=&\dis\P\big[\vec Y_\wurz\in A\big],\\[5pt]
\dis{\rm(ii)}\quad&\dis\Ga^{(n)}_t\rho^{(n)}(A)
&=&\dis\P\big[\psi^{(n)}_t(\vec Y_i)\in A\,\big|\,B_i=1\big]\qquad(i=1,2),\\[5pt]
\dis{\rm(iii)}\quad&\dis\Ga^{(n)}_tT^{(n)}\rho^{(n)}(A)
&=&\dis\P\big[\psi^{(n)}_t(\vec Y_\wurz)\in A\,\big|\,B_\wurz=1\big].
\ec
Part~(i) of (\ref{RGarep}) is immediate from (\ref{wurzdef})~(i). Since both sides of the equation are probability measures, it suffices to prove part~(ii) for $A\sub I^n\beh\{\vec\infty\}$. Then $\psi^{(n)}_t(\vec Y_i)\in A$ implies $\inf_{k\in[n]}Y^k_i\leq t$ which by (\ref{indB})~(i) in turn implies $B_i=1$. It follows that $\P[\psi^{(n)}_t(\vec Y_i)\in A]=\P[\psi^{(n)}_t(\vec Y_i)\in A,\ B_i=1]=t\P[\psi^{(n)}_t(\vec Y_i)\in A\,|\,B_i=1]$ $(i=1,2)$. Comparing with the definition of $\Ga^{(n)}_t$ in (\ref{Gan}), we see that part~(ii) holds. Formulas (\ref{RGarep})~(i) and (\ref{Bprop}) imply that
\be
T^{(n)}\rho^{(n)}(J^n[t])
=\P\big[\inf_{k\in[n]}Y^k_\wurz\leq t\big]\leq\P\big[B_\wurz=1\big]=t.
\ee
Since this holds for general $0<t\leq 1$, and since $T^{(n)}$ also clearly
preserves the symmetry of $\rho^{(n)}$, we conclude that $T^{(n)}$ maps the
space $\Mi^{(n)}$ into itself. In particular, this shows that
$\Ga^{(n)}_tT^{(n)}\rho^{(n)}$ is well-defined. Using (\ref{Bprop}), part~(iii) of (\ref{RGarep}) now follows by the same argument as part~(ii), but applied to $\vec Y_\wurz$ which by part~(i) has law $T^{(n)}\rho^{(n)}$. We next prove (\ref{commut}) for $0<t\leq 1$. We set
\be
B_\circ:=1_{\{\kappa=1\}}1_{\{\tau\leq t\}}B_1
+1_{\{\kappa=2\}}(B_1\wedge B_2).
\ee
We claim that for each measurable subset $A\sub I^n$,
\be\label{Bcirc}
{\rm(i)}\quad\P\big[B_\circ=1\big]=t^2,\quad
{\rm(ii)}\quad T^{(n)}\Ga^{(n)}_t\rho^{(n)}(A)
=\P\big[\psi^{(n)}_t(\vec Y_\wurz)\in A\,\big|\,B_\circ=1\big].
\ee
Part~(i) is a consequence of the independence of $\tau,\kappa,B_1$ and $B_2$
which yields $\P[B_\circ=1]=\ha\cdot t\cdot t+\ha\cdot t\cdot t$. To prove
part~(ii), we introduce the function
\be\label{Phidef}
\Phi[s](y):=\left\{\ba{ll}
y\quad&\mbox{if }s<y,\\
\infty\quad&\mbox{if }y\leq s,
\ea\right.\qquad\big(s\in[0,1],\ y\in I\big),
\ee
with the help of which we can write the function $\chi$ from (\ref{chi_def}) as
\be\label{chiPhi}
\chi[\tau,\kappa](y_1,y_2):=\left\{\ba{ll}
\Phi[\tau](y_1)\quad&\mbox{if }\kappa=1,\\[5pt]
y_1\wedge y_2\quad&\mbox{if }\kappa=2.\ea\right.
\ee
We define $\Phi^{(n)}[s](\vec y)$ and $\vec y_1\wedge\vec y_2$ in a
componentwise way, i.e., $\Phi^{(n)}[s](\vec y):=(\Phi[s](y^k))_{k\in[n]}$ and
$\vec y_1\wedge\vec y_2:=(y^k_1\wedge y^k_2)_{k\in[n]}$. Using the facts that
\be\ba{r@{\ }l}\label{psiprop}
{\rm(i)}&\dis\psi_t(\Phi[s](y))=\Phi[t^{-1}s](\psi_t(y))\qquad(s\leq t),\\[5pt]
{\rm(ii)}&\dis\psi_t(y_1\wedge y_2)=\psi_t(y_1)\wedge\psi_t(y_2),
\ea\ee
we can write
\be\ba{l}
\dis\P\big[\psi^{(n)}_t(\vec Y_\wurz)\in A\,\big|\,B_\circ=1\big]\\[5pt]
\dis\quad=\P\big[\kappa=1,\ \psi^{(n)}_t(\Phi^{(n)}[\tau](\vec Y_1))\in A
\,\big|\,B_\circ=1\big]
+\P\big[\kappa=2,\ \psi^{(n)}_t(\vec Y_1\wedge\vec Y_2)\in A
\,\big|\,B_\circ=1\big]\\[5pt]
\dis\quad=\ha\P\big[\Phi^{(n)}[t^{-1}\tau](\psi^{(n)}_t(\vec Y_1))\in A
\,\big|\,\tau\leq t,\ B_1=1\big]\\[5pt]
\dis\quad\phantom{=}
+\ha\P\big[(\psi^{(n)}_t(\vec Y_1)\wedge\psi^{(n)}_t(\vec Y_2))\in A
\,\big|\,B_1=1=B_2\big].
\ec
Using (\ref{RGarep})~(ii) and the fact that conditional on $\tau\leq t$, the random variable $t^{-1}\tau$ is uniformly distributed on $[0,1]$, we can rewrite this as
\be
\ha\P\big[\Phi[\ti\tau](\vec Z_1)\in A\big]
+\ha\P\big[(\vec Z_1\wedge\vec Z_2)\in A\big],
\ee
where $\vec Z_1,\vec Z_2$ are independent with law $\Ga^{(n)}_t(\rho^{(n)})$ and $\ti\tau$ is an independent random variable that is uniformly distributed on $[0,1]$. In view of (\ref{chiPhi}) and the definition of $T^{(n)}$ in (\ref{Tndef}), we arrive at (\ref{Bcirc})~(ii).

Formulas (\ref{Bprop})~(ii), (\ref{RGarep})~(iii), and (\ref{Bcirc}) give,
for any measurable subset $A\sub I^n$,
\bc
\dis t\Ga^{(n)}_tT^{(n)}\rho^{(n)}(A)
&=&\dis\P\big[\psi^{(n)}_t(\vec Y_\wurz)\in A,\ B_\wurz=1\big],\\[5pt]
\dis t^2T^{(n)}\Ga^{(n)}_t\rho^{(n)}(A)
&=&\dis\P\big[\psi^{(n)}_t(\vec Y_\wurz)\in A,\ B_\circ=1\big].
\ec
We observe that the event $\{B_\circ=1\}$ is contained in the event
$\{B_\wurz=1\}$ and the difference of these events is the event that
$\kappa=2$ and precisely one of the random variables $B_1$ and $B_2$ is one.
On the event that $\kappa=2$ by (\ref{psiprop})~(ii) we have $\psi^{(n)}_t(\vec Y_\wurz)=\psi^{(n)}_t(\vec Y_1\wedge\vec Y_2)=\psi^{(n)}_t(\vec Y_1)\wedge\psi^{(n)}_t(\vec Y_2)$. Moreover on the event that $B_i=0$, by (\ref{indB})~(i) we have $\inf_{k\in[n]}Y^k_i>t$ and hence $\psi^{(n)}_t(\vec Y_i)=\vec\infty$. In view of this, for any measurable subset $A\sub I^n$,
\be\ba{l}
\dis t\Ga^{(n)}_tT^{(n)}\rho^{(n)}(A)-t^2T^{(n)}\Ga^{(n)}_t\rho^{(n)}(A)\\[5pt]
\dis\quad=\P\big[\psi^{(n)}_t(\vec Y_1)\in A,\ \kappa=2,\ B_1=1,\ B_2=0\big]
+\P\big[\psi^{(n)}_t(\vec Y_2)\in A,\ \kappa=2,\ B_1=0,\ B_2=1\big]\\[5pt]
\dis\quad=\ha t(1-t)\P\big[\psi^{(n)}_t(\vec Y_1)\in A\,\big|\,B_1=1\big]
+\ha t(1-t)\P\big[\psi^{(n)}_t(\vec Y_2)\in A\,\big|\,B_2=1\big]\\[5pt]
\dis\quad= t(1-t)\Ga^{(n)}_t\rho^{(n)}(A),
\ec
where in the last step we have used (\ref{RGarep})~(ii). Dividing by $t$, we see
that (\ref{commut}) holds for $0<t\leq 1$.

To derive (\ref{commut}) also for $t\geq 1$, replacing $t$ by $1/t$, we may equivalently show that
\be\label{commut2}
\Ga^{(n)}_{1/t}T^{(n)}\rho^{(n)}=t^{-1}T^{(n)}\Ga^{(n)}_{1/t}\rho^{(n)}
+(1-t^{-1})\Ga^{(n)}_{1/t}\rho^{(n)}
\qquad(0<t\leq 1,\ \rho^{(n)}\in\Mi^{(n)}).
\ee
For $0<t\leq 1$, we set
\be\label{Mint}
\Mi^{(n)}[t]:=\big\{\rho^{(n)}\in\Mi^{(n)}:\rho^{(n)}
\mbox{ is concentrated on }I^n_t\big\}
\quad\mbox{with}\quad I_t:=[0,t]\cup\{\infty\}.
\ee
We observe that $\psi_t:I_t\to I$ is a bijection and $\psi_{1/t}$ is its inverse. As a result, $\Ga^{(n)}_t:\Mi^{(n)}[t]\to\Mi^{(n)}$ is a bijection and $\Ga^{(n)}_{1/t}$ is its inverse. Using this and applying (\ref{commut}) for $0<t\leq 1$ to the measure $\Ga^{(n)}_{1/t}\rho^{(n)}$, we conclude that
\bc\label{invcal}
\dis\Ga^{(n)}_tT^{(n)}\Ga^{(n)}_{1/t}\rho^{(n)}
&=&\dis tT^{(n)}\Ga^{(n)}_t\Ga^{(n)}_{1/t}\rho^{(n)}
+(1-t)\Ga^{(n)}_t\Ga^{(n)}_{1/t}\rho^{(n)}\\[5pt]
&=&\dis tT^{(n)}\rho^{(n)}+(1-t)\rho^{(n)}.
\ec
By our earlier remarks, we have $\Ga^{(n)}_{1/t}\rho^{(n)}\in\Mi^{(n)}[t]$,
which is easily seen to imply that also
$T^{(n)}\Ga^{(n)}_{1/t}\rho^{(n)}\in\Mi^{(n)}[t]$.
Using this, we can apply $\Ga^{(n)}_{1/t}$ from the left to (\ref{invcal})
and multiply by $t^{-1}$ to obtain
\be
t^{-1}T^{(n)}\Ga^{(n)}_{1/t}\rho^{(n)}
=\Ga^{(n)}_{1/t}T^{(n)}\rho^{(n)}+(t^{-1}-1)\Ga^{(n)}_{1/t}\rho^{(n)},
\ee
which proves (\ref{commut2}).
\epro

The rest of this subsection is devoted to the proof of Lemma~\ref{L:scalenu}. The maps $\Ga^{(\infty)}_t$, $C^{(\infty)}_t$, and $T^{(\infty)}$ will play an important role in the proof. Symmetric probability measures on $I^\infty$ are also known as exchangeable probability measures. We will use De Finetti's theorem to associate the space $\Mi^{(\infty)}$ with a subspace $\Mi^\ast$ of the space of all probability measures on the space of probability measures on $I$. The space $\Mi^\ast$ is naturally equipped with a special kind of stochastic order, called the convex order, and we will use a characterisation, proved in \cite{MSS18}, of the measures $\un\rho^{(2)}_\Xi$ and $\ov\rho^{(2)}_\Xi$ from (\ref{unnu}) in terms of the convex order.

We now give the precise definitions. We let $\Pc(I^n)$ denote the space of all probability measures on $I^n$ and denote the subspace of symmetric probability measures by $\Pc_{\rm sym}(I^n)$. We equip $\Pc(I)$ with the topology of weak convergence and the associated Borel-\si-field and let $\Pc(\Pc(I))$ denote the space of all probability measures on $\Pc(I)$. Each $\nu\in\Pc(\Pc(I))$ is the law of a $\Pc(I)$-valued random variable, i.e., we can construct a random probability measure $\xi$ on $I$ such that $\nu=\P[\xi\in\,\cdot\,]$ is the law of $\xi$.  By definition, for $1\leq n\leq\infty$,
\be
\nu^{(n)}:=\E\big[\underbrace{\xi\otimes\cdots\otimes\xi}_{\mbox{$n$ times}}\big]
\ee
is called the \emph{$n$-th moment measure} of $\nu$. Here
$\xi\otimes\cdots\otimes\xi$ denotes the product measure of $n$ identical
copies of $\xi$ and the expectation of a random measure $\mu$ on a Polish
space $\Om$ is the deterministic measure $E[\mu]$ defined by
$\int_\Om\phi\di E[\mu]:=E\big[\int_\Om\phi\,\di\mu\big]$ for all bounded
measurable $\phi:\Om\to\R$. Let $\xi$ be a $\Pc(I)$-valued
random variable with law $\nu$, and conditional on $\xi$, let $(Y^k)_{k\in[n]}$
be i.i.d.\ with law $\xi$. Then it is easy to see (compare
\cite[formula~(4.1)]{MSS18}) that the unconditional law of $(Y^k)_{k\in[n]}$ is
given by $\nu^{(n)}$, i.e.,
\be\label{momeas}
\nu^{(n)}=\P\big[(Y^k)_{k\in[n]}\in\,\cdot\,\big]
\quad\mbox{where}\quad
\P\big[(Y^k)_{k\in[n]}\in\,\cdot\,\big|\,\xi\big]
=\underbrace{\xi\otimes\cdots\otimes\xi}_{\mbox{$n$ times}}.
\ee
We observe that $\nu^{(n)}\in\Pc_{\rm sym}(I^n)$ for all $\nu\in\Pc(\Pc(I))$ and $1\leq n\leq\infty$. In fact, by De Finetti's theorem, the map $\nu\mapsto\nu^{(\infty)}$ is a bijection from $\Pc(\Pc(I))$ to $\Pc_{\rm sym}(I^\infty)$. This allows us to identify the space $\Mi^{(\infty)}$ with a subspace of $\Pc(\Pc(I))$. We set (compare (\ref{infyk}))
\be\ba{l}\label{Mist}
\dis\Mi^\ast:=\big\{\nu\in\Pc(\Pc(I)):
\nu(J^\ast[t])\leq t\ \forall 0<t\leq 1\big\}\\[5pt]
\dis\quad\mbox{with}\quad
J^\ast[t]:=\big\{\xi\in\Pc(I):\xi([0,t])>0\big\}
\qquad(0<t\leq 1).
\ec
Note that $J^\ast[1]=\Pc(I)\beh\{\de_\infty\}$, where $\de_\infty$ denotes the delta-measure at $\infty$. The following lemma identifies $\Mi^{(\infty)}$ with $\Mi^\ast$.

\bl[Probability measures on probability measures]
The\label{L:Miast} map $\nu\mapsto\nu^{(\infty)}$ is a bijection from $\Mi^\ast$ to
$\Mi^{(\infty)}$.
\el

In order to expose the main line of the argument, we postpone the proof of this and some of the following lemmas till later. It follows immediately from Lemmas \ref{L:mapdef} and \ref{L:Miast} that there exist  unique maps $\Ga^\ast_t$, $C^\ast_t$, and  $T^\ast$, mapping the space $\Mi^\ast$ into itself, such that
\be\label{high}
(\Ga^\ast_t\nu)^{(\infty)}=\Ga^{(\infty)}_t\nu^{(\infty)},\quad
(C^\ast_t\nu)^{(\infty)}=C^{(\infty)}_t\nu^{(\infty)},\quand
(T^\ast\nu)^{(\infty)}=T^{(\infty)}\nu^{(\infty)}
\ee
for any $t>0$ and $\nu\in\Mi^\ast$. The equation $T^\ast\nu=\nu$ has been called the \emph{higher-level RDE} in \cite{MSS18} and we will refer to  $\Ga^\ast_t$, $C^\ast_t$, and  $T^\ast$ as \emph{higher-level maps}. The following lemma gives a more explicit description of  $\Ga^\ast_t$ and $C^\ast_t$.

\bl[Higher-level scaling and cut-off maps]
Let\label{L:hiGa} $\nu\in\Mi^\ast$ and let $\xi$ be a $\Pc(I)$-valued
random variable with law $\nu$. Then for each $t>0$, the maps
$\Ga^\ast_t$ and $C^\ast_t$ defined in (\ref{high}) are given by
\be\ba{lr@{\,}c@{\,}l}\label{hiGa}
{\rm(i)}&\dis\Ga^\ast_t\nu&=&\dis t^{-1}\P\big[\xi\circ\psi_t^{-1}\in\,\cdot\,\big]
+(1-t^{-1})\de_{\de_\infty},\\[5pt]
{\rm(ii)}&\dis C^\ast_t\nu&=&\dis\P\big[\xi\circ c_t^{-1}\in\,\cdot\,\big],
\ec
where $\de_{\de_\infty}\in\Pc(\Pc(I))$ denotes the delta measure at the point
$\de_\infty\in\Pc(I)$.
\el

The following lemma, which we cite from \cite[Lemma~2]{MSS18},
identifies the map $T^\ast$ more explicitly. Recall the definition of
the map $\chi[\om]:I^2\to I$ in (\ref{chi_def}). In line with earlier
notation, in (\ref{czT}) below, $\xi_1\otimes\xi_2\circ\chi[\om]^{-1}$
denotes the image of the random measure $\xi_1\otimes\xi_2$ under the
random map $\chi[\om]$.

\bl[Higher-level RDE map]
Let\label{L:hilev} $\nu\in\Mi^\ast$, let $\xi_1,\xi_2$ be independent
$\Pc(I)$-valued random variables with law $\nu$, and let $\om$ be an
independent random variable that is uniformly distributed on
$[0,1]\times\{1,2\}$. Then the map $T^\ast$ defined in (\ref{high}) is given by
\be\label{czT}
T^\ast(\nu)=\P\big[\xi_1\otimes\xi_2\circ\chi[\om]^{-1}\in\,\cdot\,\big].
\ee
\el

We equip the space $\Pc(\Pc(I))$ with the \emph{convex order}, which we denote
by $\leq_{\rm cv}$. Two measures $\nu_1,\nu_2\in\Pc(\Pc(I))$ satisfy
$\nu_1\leq_{\rm cv}\nu_2$ if and only if the following two equivalent
conditions are satisfied, see \cite[Thm~13]{MSS18}:
\begin{enumerate}
\item $\dis\int\phi\,\di\nu_1\leq\int\phi\,\di\nu_2$ for all convex continuous
  functions $\phi:\Pc(I)\to\R$.
\item There exists an $I$-valued random variable $Y$ defined on a probability
  space $(\Om,\Fi,\P)$ and \si-fields $\Fi_1\sub\Fi_2\sub\Fi$ such that
  $\nu_i=\P\big[\P[Y\in\,\cdot\,|\Fi_i]\in\,\cdot\,\big]$ $(i=1,2)$.
\end{enumerate}
The convex order is a partial order, in particular,
$\nu_1\leq_{\rm cv}\nu_2\leq_{\rm cv}\nu_1$ implies $\nu_1=\nu_2$, see
\cite[Lemma~15]{MSS18}. The following lemma says that the scaling maps
$\Ga^\ast_t$ preserve the convex order.

\bl[Monotonicity with respect to the convex order]
Let\label{L:cvmon} $t>0$ and let $\Ga^\ast_t$ be defined in (\ref{high}).
Then $\nu_1,\nu_2\in\Mi^\ast$ and $\nu_1\leq_{\rm cv}\nu_2$
imply $\Ga^\ast_t\nu_1\leq_{\rm cv}\Ga^\ast_t\nu_2$.
\el

Let $(\om_\ibf,Y_\ibf)_{\ibf\in\T}$ be an RTP corresponding to a solution $\rho$ to the RDE (\ref{uni_RDE}) and let $\Om:=(\om_\ibf)_{\ibf\in\T}$. We define $\un\rho,\ov\rho\in\Pc(\Pc(I))$ by
\be\label{unov}
\un\rho:=\P\big[\P[Y_\wurz\in\,\cdot\,|\,\Om]\in\,\cdot\,\big]
\quand
\ov\rho:=\P\big[\de_{Y_\wurz}\in\,\cdot\,\big].
\ee
We observe that the second moment measures of  $\un\rho$ and $\ov\rho$ are given by
\be\label{unnu2}
\un\rho^{(2)}=\P\big[(Y_\wurz,Y'_\wurz)\in\,\cdot\,\big]
\quand
\ov\rho^{(2)}=\P\big[(Y_\wurz,Y_\wurz)\in\,\cdot\,\big],
\ee
where $(Y'_\ibf)_{\ibf\in\T}$ is conditionally independent of $(Y_\ibf)_{\ibf\in\T}$ given $\Om$ and conditionally equally distributed with $(Y_\ibf)_{\ibf\in\T}$. In particular, if $\rho=\rho_\Xi$, then $\un\rho^{(2)}$ and $\ov\rho^{(2)}$ are the measures defined in (\ref{unnu}). The following proposition, which we cite from \cite[Props~3 and 4]{MSS18}, says that $\un\rho$ and $\ov\rho$ are the minimal and maximal solutions, with respect to the convex order, of the higher-level RDE $T^\ast(\nu)=\nu$.

\bp[Minimal and maximal solutions]
Let\label{P:minmax} $\rho$ be a solution to the RDE (\ref{uni_RDE}). Then the set
\be\label{Sir}
\Si_\rho:=\big\{\nu\in\Pc(\Pc(I)):T^\ast(\nu)=\nu,\ \nu^{(1)}=\rho\big\}
\ee
has a unique minimal element $\un\rho$ and maximal element $\ov\rho$ with respect to the convex order, and these are the measures defined in (\ref{unov}).
\ep

We will derive Lemma~\ref{L:scalenu} from the following, stronger statement. We will first give the proofs of Lemmas \ref{L:CGa} and \ref{L:scalenu} and then provide the proofs of the remaining lemmas.

\bl[Scaling of minimal and maximal solutions]
Let\label{L:CGa} $\rho$ be a solution to the RDE (\ref{uni_RDE}) and let $t>0$. Then
\be\label{CGa}
{\rm(i)}\ C^\ast_t\un\rho=\un{C_t\rho},\quad
{\rm(ii)}\  C^\ast_t\ov\rho=\ov{C_t\rho},\quad
{\rm(iii)}\ \Ga^\ast_t\un\rho=\un{\Ga_t\rho},\quad
{\rm(iv)}\ \Ga^\ast_t\ov\rho=\ov{\Ga_t\rho}.
\ee
\el

\bpro
We first prove (\ref{CGa})~(i) and (ii). Recall from (\ref{Can}) that $C_t\rho:=\rho\circ c_t^{-1}$ where $c_t$ is the cut-off map defined in (\ref{ct}). Let $(\om_\ibf,Y_\ibf)_{\ibf\in\T}$ be the RTP corresponding to $\rho$. Then it is easy to see that  $(\om_\ibf,c_t(Y_\ibf))_{\ibf\in\T}$ is the RTP corresponding to $C_t\rho$. Applying the definition in (\ref{unov}), it follows that
\be
\un{C_t\rho}=\P\big[\P[c_t(Y_\wurz)\in\,\cdot\,|\,\Om]\in\,\cdot\,\big]
=\P\big[\P[Y_\wurz\in\,\cdot\,|\,\Om]\circ c_t^{-1}\in\,\cdot\,\big]=C^\ast_t\un\rho,
\ee
where in the last equlity we have used Lemma~\ref{L:hiGa}. This proves (\ref{CGa})~(i). The proof of (\ref{CGa})~(ii) is similar, using the fact that $\de_{c_t(Y_\wurz)}=\de_{Y_\wurz}\circ c_t^{-1}$.

We next prove (\ref{CGa})~(iii) and (iv). We start by observing that (\ref{high}) implies that
\be\label{fimo}
(\Ga^\ast_t\nu)^{(1)}=\Ga_t\nu^{(1)}\qquad(\nu\in\Mi^\ast,\ t>0).
\ee
We moreover claim that
\be\label{GGC}
\Ga^\ast_{1/t}\circ\Ga^\ast_t\nu=C^\ast_t\nu\qquad(\nu\in\Mi^\ast,\ t>0).
\ee
To see this, define $\psi^\ast_t:\Pc(I)\to\Pc(I)$ by $\psi^\ast_t(\xi):=\xi\circ\psi_t^{-1}$. Then (\ref{hiGa})~(i) says that $\Ga^\ast_t\nu=t^{-1}\nu\circ(\psi^\ast_t)^{-1}+(1-t^{-1})\de_{\de_\infty}$. A simple calculation using the fact that $\psi^\ast_t(\de_\infty)=\de_\infty$ then gives
\be
\Ga^\ast_s\Ga^\ast_t\nu=(st)^{-1}\nu\circ(\psi^\ast_t)^{-1}\circ(\psi^\ast_s)^{-1}+(1-(st)^{-1})\de_{\de_\infty}\qquad(s,t>0).
\ee
Applying this with $s=1/t$, using (\ref{psic}), and (\ref{hiGa})~(ii), it follows that if $\xi$ is a $\Pc(I)$-valued random variable with law $\nu$, then
\be
\Ga^\ast_{1/t}\circ\Ga^\ast_t\nu=\P\big[\xi\circ\psi_t^{-1}\circ\psi_{1/t}^{-1}\in\,\cdot\,\big]=\P\big[\xi\circ c_t^{-1}\in\,\cdot\,\big]=C^\ast_t\nu,
\ee
which proves (\ref{GGC}). Let $\rho$ be a solution to the RDE (\ref{uni_RDE}), let $0<t\leq 1$, and let $\rho':=\Ga_{1/t}\rho$. Then, by Lemma~\ref{L:scaleRDE}, $\rho'$ also solves the RDE (\ref{uni_RDE}). Moreover, $\rho'$ is concentrated on $[0,t]\cup\{\infty\}$ and hence in view of (\ref{psic}) $\Ga_t\rho'=\rho$. Let us write
\be
\Mi^\ast_\rho:=\big\{\nu\in\Mi^\ast:\nu^{(1)}=\rho\big\}
\ee
and let $\Mi^\ast_{\rho'}$ be defined similarly with $\rho$ replaced by $\rho'$. Since $t\leq 1$, the cut-off map $c_{1/t}$ and hence also $C^\ast_{1/t}$ are the identity maps and hence it follows from (\ref{GGC}) with $1/t$ instead of $t$ and from (\ref{fimo}) that
\be\label{tinvt}
\mbox{$\Ga^\ast_{1/t}$ is a bijection from $\Mi^\ast_\rho$ to $\Mi^\ast_{\rho'}$ and that $\Ga^\ast_t$ is its inverse.}
\ee

It follows from Lemma~\ref{L:scalable} and our identification of $\Mi^{(\infty)}$ with $\Mi^\ast$ in Lemma~\ref{L:Miast} and (\ref{high}) that the sets $\Si_\rho$ and $\Si_{\rho'}$ defined in (\ref{Sir}) are subsets of $\Mi^\ast_\rho$ and $\Mi^\ast_{\rho'}$, respectively. Using moreover Lemma~\ref{L:scalenRDE}, we see  that $\Ga^\ast_{1/t}$ maps $\Si_\rho$ into $\Si_{\rho'}$ and that $\Ga^\ast_t$ maps $\Si_{\rho'}$ into $\Si_\rho$. By (\ref{tinvt}), we conclude that $\Ga^\ast_{1/t}$ is a bijection from $\Si_\rho$ to $\Si_{\rho'}$. By Lemma~\ref{L:cvmon}, the map $\Ga^\ast_{1/t}$ is monotone with respect to the convex order. By Proposition~\ref{P:minmax}, the set $\Si_\rho$ has unique minimal and maximal elements with respect to the convex order, which are $\un\rho$ and $\ov\rho$. Likewise, $\un\rho'$ and $\ov\rho'$ are the unique minimal and maximal elements of $\Si_{\rho'}$. Since $\Ga^\ast_{1/t}$ is a monotone bijection from $\Si_\rho$ to $\Si_{\rho'}$, it must map $\un\rho$ and $\ov\rho$ to $\un\rho'$ and $\ov\rho'$, respectively. Recalling that $\rho'=\Ga_{1/t}\rho$, this shows that
\be
\Ga^\ast_{1/t}\un\rho=\un{\Ga_{1/t}\rho},\quad
\Ga^\ast_{1/t}\ov\rho=\ov{\Ga_{1/t}\rho},
\ee
which proves (\ref{CGa})~(iii) and (iv) in the special case that $t\geq 1$.

To prove (\ref{CGa})~(iii) for $0<t\leq 1$, let $\rho''$ be a solution to the RDE (\ref{uni_RDE}), let $\rho:=\Ga_t\rho''$, and as before let $\rho'=\Ga_{1/t}\rho$. Then, by  (\ref{psic}), $\rho'=\Ga_{1/t}\circ\Ga_t\rho''=C_t\rho''$. Our previous arguments show that $\Ga^\ast_{1/t}$ maps $\un\rho$ into $\un\rho'$ and hence the inverse map $\Ga^\ast_t$ maps $\un\rho'$ into $\un\rho$, i.e.,
\be\label{acG}
\Ga^\ast_t\un\rho'=\un\rho.
\ee
Formulas (\ref{GGC}) and (\ref{CGa})~(i) tell us that
\be
\Ga^\ast_{1/t}\circ\Ga^\ast_t\un\rho''=C^\ast_t\un\rho''=\un{C_t\rho''}=\un\rho'.
\ee
Applying $\Ga^\ast_t$ from the left, using (\ref{tinvt}) and (\ref{acG}), we obtain that
\be
\Ga^\ast_t\un\rho''=\Ga^\ast_t\un\rho'=\un\rho.
\ee
Since $\rho=\Ga_t\rho''$, this proves (\ref{CGa})~(iii) for $0<t\leq 1$. The proof of (\ref{CGa})~(iv) for $0<t\leq 1$ goes exactly in the same way.
\epro

\bpro[of Lemma~\ref{L:scalenu}]
It follows from Lemma~\ref{L:CGa} and (\ref{high}) that $\Ga^{(2)}_t\un\rho^{(2)}_\Xi=\un{\Ga_t\rho_\Xi}^{(2)}$ and  $\Ga^{(2)}_t\ov\rho^{(2)}_\Xi=\ov{\Ga_t\rho_\Xi}^{(2)}$, where  $\Ga_t\rho_\Xi=\rho_{\Xi'}$ by Lemma~\ref{L:scaleRDE}.
\epro

We cited Lemma~\ref{L:hilev} and Proposition~\ref{P:minmax} from \cite{MSS18}, so to complete the proofs of this subsection, it only remains to provide the proofs of Lemmas \ref{L:Miast}, \ref{L:hiGa}, and \ref{L:cvmon}.\med

\bpro[of Lemma~\ref{L:Miast}]
By De Finetti's theorem, the map $\nu\mapsto\nu^{(\infty)}$ is a bijection from $\Pc(\Pc(I))$ to $\Pc_{\rm sym}(I^\infty)$, so it suffices to show that for $\nu\in\Pc(\Pc(I))$, one has $\nu\in\Mi^\ast$ if and only if $\nu^{(\infty)}\in\Mi^{(\infty)}$. Let $\xi$ be a $\Pc(I)$-valued random variable with law $\nu$ and conditional on $\xi$, let $(Y^k)_{k\in\N_+}$ be i.i.d.\ with law $\xi$. Then the unconditional law of $(Y^k)_{k\in\N_+}$ is $\nu^{(\infty)}$. By the definition in (\ref{Mist}), $\nu\in\Mi^\ast$ if and only if $\P[\xi([0,t])>0]\leq t$ for all $0<t\leq 1$. The event $\{\xi([0,t])>0\}$ is a.s.\ equal to the event $\{\exists k\in\N_+\mbox{ s.t.\ }Y^k\leq t\}$, so comparing with the definition in (\ref{infyk}) we see that $\nu\in\Mi^\ast$ if and only if $\nu^{(\infty)}\in\Mi^{(\infty)}$.
\epro

\bpro[of Lemma~\ref{L:hiGa}]
We need to show that $\Ga^\ast_t$ and $C^\ast_t$ defined as in (\ref{hiGa}) satisfy (\ref{high}). Conditional on $\xi$, let $(Y^k)_{k\in\N_+}$ be i.i.d.\ with law $\xi$. Then the unconditional law of $(Y^k)_{k\in\N_+}$ is $\nu^{(\infty)}$ and by (\ref{Can}) $C^{(\infty)}_t\nu^{(\infty)}$ is the (unconditional) law of $(c_t(Y^k))_{k\in\N_+}$, which is the same as $(C^\ast_t\nu)^{(\infty)}$. The claim for $\Ga^\ast_t$ follows in the same way, using (\ref{Gan}) and the fact that $\de_{\de_\infty}^{(\infty)}=\de_{\vec\infty}$.
\epro

\bpro[of Lemma~\ref{L:cvmon}]
Assume that $\nu_1,\nu_2\in\Mi^\ast$ satisfy $\nu_1\leq_{\rm cv}\nu_2$. By
characterisation~(ii) of the convex order, we can find a random variable $Y$
and \si-fields $\Fi_1\sub\Fi_2$ such that
$\nu_i=\P\big[\P[Y\in\,\cdot\,|\Fi_i]\in\,\cdot\,\big]$ $(i=1,2)$.
Then, by (\ref{hiGa})~(i),
\be
\Ga^\ast_t\nu_i=t^{-1}\ti\nu_i+(1-t^{-1})\de_{\de_\infty}
\quad\mbox{with}\quad
\ti\nu_i:=\P\big[\P[\psi_t(Y)\in\,\cdot\,|\Fi_i]\in\,\cdot\,\big]
\qquad(i=1,2).
\ee
Since $\Fi_1\sub\Fi_2$, by characterisation~(ii) of the convex order, we see
that $\ti\nu_1\leq_{\rm cv}\ti\nu_2$. Using characterisation~(i) of the
convex order, it follows that $\Ga^\ast_t\nu_1\leq_{\rm cv}\Ga^\ast_t\nu_2$.
\epro

\section{Scale invariant solutions to the bivariate RDE}\label{section_scale_invarint_solutions_of_RDE}

The goal of this section is to prove Theorem \ref{the_theorem}.
Let  $\theta \in (0,1)$.
 Throughout Section \ref{section_scale_invarint_solutions_of_RDE}  we will use the shorthand $\rho$ to denote the probability measure $\rho_{\Xi_\theta}$ defined in \eqref{muXi}.
Let
\begin{equation}\label{x_and_c_def}
x_k :=  \theta^k \quad \text{and} \quad c_k :=  \frac{1-\theta}{1+\theta} \cdot \theta^k, \qquad k \in \mathbb{N}.
\end{equation}
thus we have
\begin{equation}\label{rho}
 \rho = \rho_{\Xi_\theta} =\sum_{k=0}^\infty c_k \de_{x_k} +\frac{\tet}{1+\tet}\de_\infty.
\end{equation}
For simplification we also introduce the notation
\begin{equation}\label{x_minus_one}
x_{-1} := \infty \quad \text{and} \quad c_{-1} := \rho(\{\infty\}) = \frac{\theta}{1+\theta}.
\end{equation}
Using this notation we have $\rho(\{x_k\}) = c_k$ for every $k \ge -1$.

\begin{definition}\label{biv_with_rho_marginals}
Let $\mathcal{P}^{(2)}_\theta$ denote the space of symmetric probability measures on $I \times I$ such that its marginal distributions are $\rho$.
\end{definition}

\subsection{Main lemmas}\label{subsection_main_lemmas}

\paragraph{}In Section \ref{subsection_main_lemmas} we state the key lemmas of Section \ref{section_scale_invarint_solutions_of_RDE} and prove Theorem \ref{the_theorem} using them.

\begin{definition}[The signature of a scale invariant measure]\label{def:rho->f}
Let $\theta \in (0,1)$. The signature of a scale invariant measure $\rho^{(2)} \in \mathcal{M}^{(2)}_{\theta}$ is the function $f_{\rho^{(2)}}:\mathbb{N}\to\mathbb{R}$
defined by
\begin{equation}\label{rho->f}
f_{\rho^{(2)}}(n) :=  \rho^{(2)} \big( \{ [0,x_n] \times I\} \cup \{ I \times [0,1] \} \big), \qquad n \in \mathbb{N}.
\end{equation}
\end{definition}

The signature of the diagonal measure $\overline{\rho}^{(2)}$ (c.f.\ \eqref{unnu}) is
\begin{equation}\label{signature_of_diagonal}
f_{\overline{\rho}^{(2)}}(n) = \P\big[Y_\wurz \leq x_n \text{ or } Y_\wurz \leq 1  \big]=  \sum \limits_{k = 0}^\infty c_k = \frac{1}{1+\theta}, \quad n \in \mathbb{N}.
\end{equation}

\begin{lemma}[Conditions for $f$ to be a signature]\label{lemma:f_conditions}
If $\theta \in (0,1)$ and $f: \mathbb{N} \to \mathbb{R}$ then there exists a (unique) probability measure $\rho^{(2)} \in \mathcal{M}^{(2)}_{\theta}$ such that $f$ is its signature if and only if
\begin{enumerate}
\item $f(0) \le 1$,
\item $\lim \limits_{n \to \infty} f(n) = \frac{1}{1+\theta}$,
\item $f(n)$ is non-increasing,
\item $(1+\theta) \cdot f(0) \le 2f(1)$,
\item $(1+\theta) \cdot f(n) \le \theta \cdot f(n-1) + f(n+1)$ for every $n \ge 1$.
\end{enumerate}
\end{lemma}

\paragraph{}We will prove this lemma in Section \ref{sec:f_conditions}.
Next we define a function $f_{\theta, c}(n), n \in \mathbb{N}$ that will help us identify the signature of a scale invariant solution of the bivariate RDE.

\begin{lemma}[Implicit equation for $f_{\theta, c}(n)$]\label{lemma:imp_eq_for_f_theta_c}
Let $\theta \in (0,1)$ and $c \ge 0$ be arbitrary. The system of equations
\begin{align}
   f_{\theta,c}(0)^2 - \frac{1}{1+\theta}f_{\theta,c}(0) &= 2c, \label{quadratic_eq_for_f_0} \\
 f_{\theta,c}(n-1)^2- f_{\theta,c}(n)^2 &= \theta^{n-1}\left( f_{\theta,c}(n-1)-f_{\theta,c}(n) \right)+c \cdot \theta^{2n-2}(1-\theta^2), \quad n \geq 1, \label{quadratic_eq_for_f_n_and_n_minus_1}  \\
  f_{\theta,c}(0) & >0, \quad f_{\theta,c}(n)> \frac{\theta^{n-1}}{2}, \quad n \geq 1 \label{f_theta_c_n_geq_ugly}
\end{align}
has a unique solution $f_{\theta,c}(n), n \in \mathbb{N}$.
\end{lemma}

\begin{lemma}[Existence and continuity of $f_{\theta, c}(\infty)$]\label{lemma:f_continuous}
If $\theta \in (0,1), c \ge 0$, then the limit
\begin{equation}\label{def:eq:f_theta_infty}
f_{\theta, c}(\infty):= \lim_{n \to \infty} f_{\theta, c}(n)
\end{equation} exists and
the function $c \mapsto f_{\theta, c}(\infty)$ is continuous.
\end{lemma}

 \paragraph{}We will prove Lemmas \ref{lemma:imp_eq_for_f_theta_c} and \ref{lemma:f_continuous}  in Section \ref{sec:f_unique}.
 Note that if $c=0$ then
\begin{equation}\label{f_theta_null}
   f_{\theta, 0}(n) = \frac{1}{1+\theta}, \qquad n \in \mathbb{N}
\end{equation}
is a solution of \eqref{quadratic_eq_for_f_0}-\eqref{f_theta_c_n_geq_ugly} (and it  follows from the uniqueness statement of Lemma \ref{lemma:imp_eq_for_f_theta_c} that \eqref{f_theta_null} is the only solution of of \eqref{quadratic_eq_for_f_0}-\eqref{f_theta_c_n_geq_ugly} in the $c=0$ case).

\begin{remark}\label{remark:cont_case_system}
Note that if we rearrange \eqref{quadratic_eq_for_f_0} and \eqref{quadratic_eq_for_f_n_and_n_minus_1},  we get
\begin{equation}
f_{\theta, c}(0)^2 - \frac{f_{\theta, c}(0)}{1+\theta} = 2c, \qquad
\frac{f_{\theta, c}(n-1) - f_{\theta, c}(n)}{\theta^{n-1} - \theta^n} = \frac{c \cdot \theta^{n-1} \cdot \frac{1+\theta}{2}}{\frac{f_{\theta, c}(n-1) + f_{\theta, c}(n)}{2} - \frac{\theta^{n-1}}{2}}.
\end{equation}
Now if we non-rigorously define the function $f$ by $f(\theta^n) := f_{\theta, c}(n)$ when $\theta$ is very close to $1$, moreover we denote $r:= \theta^n$, then in the $\theta \to 1$ limit we get
\begin{equation} \label{cont_case}
f(1)^2 - \frac 12 f(1) = 2c, \qquad
\frac{\partial}{\partial r} f(r) = \frac{c \cdot r}{f(r) - r/2}
\end{equation}
i.e., conditions (iii) and (i) of equation (2.2) of \cite{RST19}.
We also note that condition (ii) of equation (2.2) of \cite{RST19},i.e., $f(0) = \frac 12$, corresponds to the condition $ f_{\theta, c}(\infty) = \frac{1}{1+\theta}$ in our current discrete setting. We will see that the key question is whether there exists $c>0$ for which $ f_{\theta, c}(\infty) = \frac{1}{1+\theta}$ holds.
\end{remark}

\begin{lemma}[Signature of scale invariant solution of the bivariate RDE]\label{lemma:equivalent_system}
Let $\rho^{(2)} \in \mathcal{M}^{(2)}_{\theta}$ and let $f_{\rho^{(2)}}$ denote its signature.

\begin{enumerate}
\item $\rho^{(2)}$ is a solution of the bivariate RDE \eqref{bivar_RDE} if and only if there exists $c \ge 0$ such that $f_{\rho^{(2)}}(n) = f_{\theta, c}(n)$ holds for every $n \in \mathbb{N}$.

\item If $\rho^{(2)}$ is a solution of the bivariate RDE and $c$ is the parameter for which $f_{\rho^{(2)}}(n) = f_{\theta, c}(n)$ holds for every $n \in \mathbb{N}$, then
\begin{equation}\label{c_upper_bound}
c \le \max \left( 0, \frac{\theta \cdot (2\theta - 1)}{(1+\theta)^2} \right).
\end{equation}
\end{enumerate}
\end{lemma}

\paragraph{}We will prove this lemma in Section \ref{sec:equivalent_system}.
Note that that the diagonal measure $\overline{\rho}^{(2)}$ defined in \eqref{unnu} is a solution of the bivariate RDE \eqref{bivar_RDE}, and indeed  $f_{\overline{\rho}^{(2)}}(n) = \frac{1}{1+\theta} = f_{\theta, 0}(n)$ for every $n \in \mathbb{N}$ by \eqref{signature_of_diagonal} and \eqref{f_theta_null},
in accordance with Lemma \ref{lemma:equivalent_system}.

\begin{definition}[Perturbation of the diagonal signature]\label{def:f_tilde}
Let $\theta \in (0,1)$ and $c \ge 0$ be arbitrary and $f_{\theta, c}$. Let us define
\begin{equation}\label{eq_tilde_f_theta_n}
\tilde{f}_\theta(n) := \left( \frac{\partial}{\partial c} f_{\theta, c}(n) \right) \bigg|_{c=0_+} \quad \text{and} \quad \tilde{f}_\theta(\infty) := \lim \limits_{n \to \infty} \tilde{f}_\theta (n).
\end{equation}
\end{definition}
We will prove in Section \ref{sec:lim_f_c=0} that the limit in \eqref{eq_tilde_f_theta_n} exists. Recall the notion of $\theta^* \in (\tfrac{1}{2},1)$ from Lemma \ref{L:tetdef}.
\begin{lemma}[Critical value]\label{lemma:lim_f_c=0}
We have $\tilde{f}_\theta(\infty) > 0$ for every $\theta \in (0, \theta^*)$, $\tilde{f}_\theta(\infty) < 0$ for every $\theta \in (\theta^*, 1)$ and $\tilde{f}_{\theta^*}(\infty)=0$.
\end{lemma}
We will prove this lemma in Section \ref{sec:lim_f_c=0}.

\begin{lemma}[Solution of the recursion if $\theta \le \theta^*$]\label{lemma:theta<=theta*}
If $\theta \in \left(\frac 12, \theta^*\right]$, then there does not exist $c \in \left( 0, \frac{\theta \cdot (2\theta - 1)}{(1+\theta)^2} \right] $ for which $ f_{\theta, c}(\infty) = \frac{1}{1+\theta}$.
\end{lemma}

\paragraph{}We will prove this lemma in Section \ref{sec:theta<=theta*}.

\begin{lemma}[Solution of the recursion if $\theta > \theta^*$]\label{lemma:theta>theta*}
For any $\theta \in (\theta^*, 1)$ there exists a $\hat{c} > 0$ for which $ f_{\theta, \hat{c}}(\infty) = \frac{1}{1+\theta}$,
moreover $f_{\theta, \hat{c}}$ also satisfies all of the conditions of Lemma \ref{lemma:f_conditions}.
\end{lemma}

\paragraph{}We will prove this lemma in Section \ref{sec:theta>theta*}.

\begin{remark}
 Figure \ref{fig:fix_c} shows  the values of the parameter $c$ for which we have $ f_{\theta, c}(\infty) = \frac{1}{1+\theta}$ for different values of $\theta \in (0.6,1)$. It shows that if $\theta \le \theta^*$ then the only such value is $c=0$, but if $\theta > \theta^*$ then there also exists a positive value $\hat{c}$. We also note that if $\theta \to 1$ then the numerical simulations suggest that $\hat{c} \to 0.01770838$, which coincides with parameter value of $c$ which gives the non-diagonal solution in the case $\Xi_1=(0,1]$, see \cite[Section 1.6]{RST19}. In other words, $c=0.01770838$ is the unique positive value of  $c$ for which the differential equation \eqref{cont_case} together with the boundary condition $f(0)= \frac 12$ has a solution.
\end{remark}

\begin{figure}[!ht]
	\centering
	\includegraphics[scale=0.33]{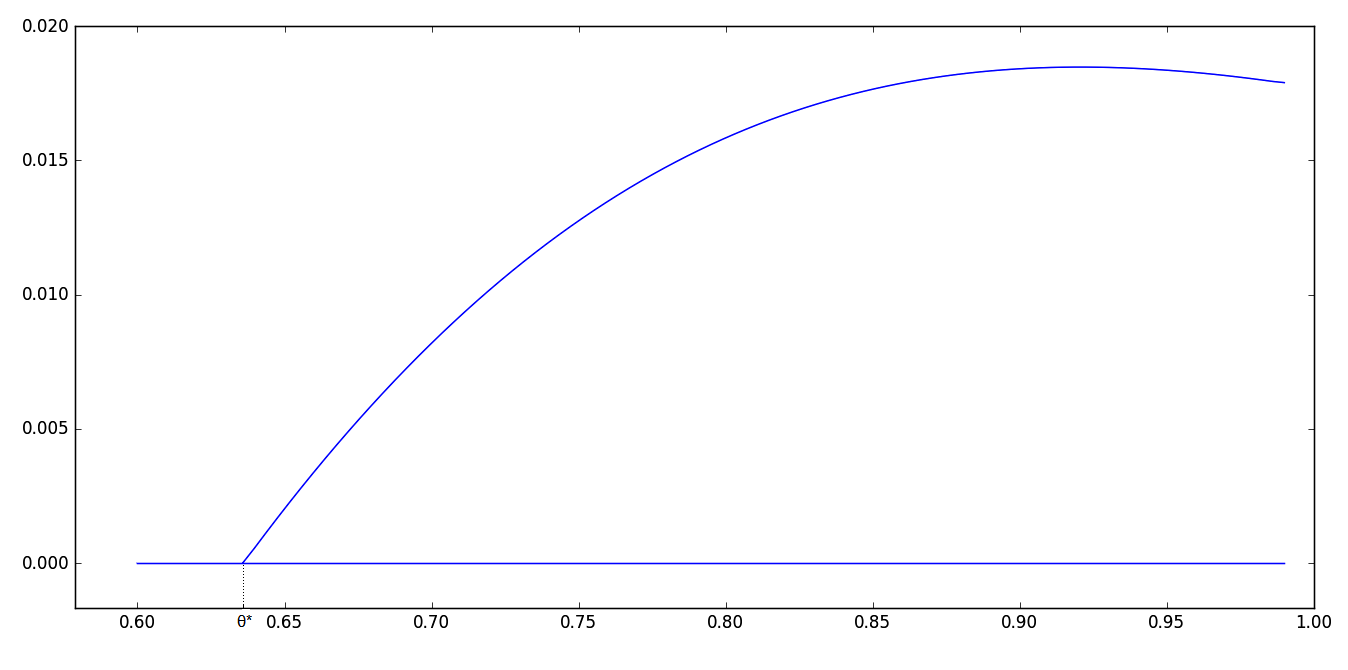}
	\caption{The values of $c$  for which $ f_{\theta, c}(\infty) = \frac{1}{1+\theta}$}
	\label{fig:fix_c}
\end{figure}

\begin{Proof}[of Theorem \ref{the_theorem}]
The diagonal measure $\ov\rho^{(2)}$ defined in \eqref{unnu} is indeed a solution of the bivariate RDE \eqref{bivar_RDE} for every $\theta \in (0,1)$.

If $\theta \le \frac 12$ and $\rho^{(2)} \in \mathcal{M}^{(2)}_{\theta}$ is a solution of the bivariate RDE, then by Lemma \ref{lemma:equivalent_system}(ii) we must have $c=0$, where $c$ is the parameter for which $f_{\rho^{(2)}} \equiv f_{\theta, c}$ (such $c$ exists by Lemma \ref{lemma:equivalent_system}(i)). By \eqref{f_theta_null}
we have $f_{\rho^{(2)}}(n) \equiv \frac{1}{1+\theta}$,  thus by \eqref{signature_of_diagonal} and the uniqueness statement of Lemma~\ref{lemma:f_conditions} we obtain that  there is no scale invariant solution of the bivariate RDE other than the diagonal solution  in the $\theta \le \frac 12$  case.

If $\theta \in \left(\frac 12, \theta^*\right]$ and we assume that $\rho^{(2)} \in \mathcal{M}^{(2)}_{\theta}$ is a solution of the bivariate RDE, then by Lemma \ref{lemma:equivalent_system}  we have $c \in \left[ 0, \frac{\theta \cdot (2\theta - 1)}{(1+\theta)^2} \right]$ for the parameter $c$ which gives $f_{\rho^{(2)}} \equiv f_{\theta, c}$. But by Lemma \ref{lemma:theta<=theta*} we know that there is no $c \in \left( 0, \frac{\theta \cdot (2\theta - 1)}{(1+\theta)^2} \right]$ such that $\lim \limits_{n \to \infty} f_{\theta, c}(n) = \frac{1}{1+\theta}$ holds. Therefore by condition (ii) of Lemma \ref{lemma:f_conditions} we see that again only $c=0$ produces a signature of a scale invariant solution of the bivariate RDE.

If $\theta > \theta^*$ then by Lemmas \ref{lemma:f_conditions} and \ref{lemma:theta>theta*} there exists a measure $\hat{\rho}^{(2)} \in \mathcal{M}^{(2)}_{\theta}$ for which $f_{\theta, \hat{c}}(n) = f_{\hat{\rho}^{(2)}}(n)$ for every $n \in \mathbb{N}$. The measure $\hat{\rho}^{(2)}$ is non-diagonal, as we explain. First note that $\hat{c} \neq 0$ implies $f_{\theta, \hat{c}} \neq f_{\theta, 0}$, thus $f_{\hat{\rho}^{(2)}} \neq f_{\overline{\rho}^{(2)}}$ (since $f_{\overline{\rho}^{(2)}}=f_{\theta, 0}$ by \eqref{signature_of_diagonal} and \eqref{f_theta_null}), and thus we must have $\hat{\rho}^{(2)} \neq \overline{\rho}^{(2)}$, i.e., $\hat{\rho}^{(2)}$ is non-diagonal.

Finally, we observe that $\hat{\rho}^{(2)}$ is a solution of the bivariate RDE \eqref{bivar_RDE} by Lemma \ref{lemma:equivalent_system}(i).
\end{Proof}

\subsection{\texorpdfstring{Conditions for $f$ to be a signature}{Conditions for f to be a signature}}\label{sec:f_conditions}


\paragraph{}In this section we show the necessary and sufficient conditions for a function $f: \mathbb{N} \to \mathbb{R}$ to be the signature of some $\rho^{(2)} \in \mathcal{M}^{(2)}_{\theta}$, i.e., we prove Lemma \ref{lemma:f_conditions}. To do this, first we define the bivariate signature $F_{\rho^{(2)}}$ in Definition \ref{def:rho->F} for each $\rho^{(2)} \in \mathcal{P}^{(2)}_\theta$ (c.f.\ Definition \ref{biv_with_rho_marginals}). In Lemma \ref{lemma:F_characterizes} we prove that this function $F_{\rho^{(2)}}$ characterizes the distribution $\rho^{(2)}$ and  in Lemma \ref{lemma:F_conditions} we prove necessary and sufficient conditions for bivariate functions to be the bivariate signature of some $\rho^{(2)} \in \mathcal{P}^{(2)}_\theta$. After analysing the relation between scale invariant measures and scale invariant bivariate signatures in Lemma \ref{lemma:scale_inv} as well as the relation between $F_{\rho^{(2)}}$ and the univariate signature in Lemma \ref{lemma:F_and_f}, we can easily conclude the proof of Lemma \ref{lemma:f_conditions}.

\begin{definition}[Bivariate signature]\label{def:rho->F}
Given $\rho^{(2)} \in \mathcal{P}^{(2)}_\theta$, we define the bivariate function $F_{\rho^{(2)}} : \{x_k\}_{k=0}^\infty \times \{x_k\}_{k=0}^\infty \to [0,1] $ by
\begin{equation}\label{rho->F}
F_{\rho^{(2)}}(x_k, x_j) := \rho^{(2)} \big( \{ [0,x_k] \times I\} \cup \{ I \times [0,x_j] \} \big), \quad j, k \in \mathbb{N}.
\end{equation}
\end{definition}

Recall the notation from the beginning of Section \ref{section_scale_invarint_solutions_of_RDE}.

\begin{lemma}[Bivariate signature characterizes the measure]\label{lemma:F_characterizes}
The measure $\rho^{(2)} \in \mathcal{P}^{(2)}_\theta$ is uniquely characterized by $F_{\rho^{(2)}}$, in particular, for any $ j, k \in \mathbb{N}$ we have
\begin{align}
\label{first_qq}
\rho^{(2)}\big(\{\infty\}\times\{\infty\}\big)&= 1-F_{\rho^{(2)}}(x_0,x_0),\\
\label{second_qq}
\rho^{(2)}\big([0,x_k]\times\{\infty\}\big)&= F_{\rho^{(2)}}(x_k,x_0) - \frac{1}{1+\theta},\\
\label{third_qq}
\rho^{(2)}\big(\{\infty\}\times[0,x_j]\big)&= F_{\rho^{(2)}}(x_j,x_0) - \frac{1}{1+\theta},\\
\label{F_charact_4}
\rho^{(2)}\big([0,x_k]\times[0,x_j]\big)&= \frac{x_k}{1+\theta} + \frac{x_j}{1+\theta} - F_{\rho^{(2)}}(x_k,x_j).
\end{align}
\end{lemma}

\begin{Proof} The proof of \eqref{first_qq} follows from Definition \ref{def:rho->F} using $x_0=1$ and
$\rho^{(2)}( ([0,1]\cup \{\infty\})^2 )=1$.
 Since the marginal distribution of $\rho^{(2)}$ is $\rho$, for every $j \in \mathbb{N}$ we have
$\rho^{(2)}\big( I \times [0,x_j] \big)=\rho( [0,x_j] )=\sum_{i=j}^\infty c_i=\frac{x_j}{1+\theta}$.
The equalities \eqref{second_qq}, \eqref{third_qq} and \eqref{F_charact_4} readily follow.
The $\rho^{(2)}$ measure of  every atom of $\rho^{(2)}$ can be  determined  using \eqref{first_qq}-\eqref{F_charact_4} and  inclusion-exclusion.
\end{Proof}

\begin{lemma}[Necessary and sufficient conditions on $F$]\label{lemma:F_conditions}
Let $\theta \in (0,1)$. Let us assume that we are given a function $F: \{x_k\}_{k=0}^\infty \times \{x_k\}_{k=0}^\infty \to [0, \infty)$. There exists a unique probability measure $\rho^{(2)} \in \mathcal{P}^{(2)}_\theta$ such that $F \equiv F_{\rho^{(2)}}$ holds (where $F_{\rho^{(2)}}$ is defined in Definition \ref{def:rho->F}) if and only if the following conditions are fulfilled:
\begin{enumerate}[1.]
\item $F(x_0, x_0) \le 1$,
\item $\lim \limits_{k \to \infty} F(x_k, x_j) = \frac{x_j}{1+\theta} \quad \forall \, j \in \mathbb{N}$,
\item $F(x_k, x_0)$ is non-increasing in $k$,
\item $F(x_k, x_j) = F(x_j, x_k) \quad \forall \, j, k \in \mathbb{N}$,
\item $-F(x_k, x_j) + F(x_{k+1}, x_j) + F(x_k, x_{j+1}) - F(x_{k+1}, x_{j+1}) \ge 0 \quad \forall \, j, k \in \mathbb{N}$.
\end{enumerate}
\end{lemma}

\begin{Proof}
If $\rho^{(2)} \in \mathcal{P}^{(2)}_\theta$ and we define $F_{\rho^{(2)}}$ as in \eqref{rho->F}, then conditions 1., 2.\ and 3.\  trivially hold for $F_{\rho^{(2)}}$. Condition 4.\ follows from the symmetry of $\rho^{(2)}$ and finally condition 5.\ also holds, since $-F_{\rho^{(2)}}(x_k, x_j) + F_{\rho^{(2)}}(x_{k+1}, x_j) + F_{\rho^{(2)}}(x_k, x_{j+1}) - F_{\rho^{(2)}}(x_{k+1}, x_{j+1}) = \rho^{(2)}(x_k, x_j)$, where $\rho^{(2)}(x_k, x_j)$ is a shorthand for $\rho^{(2)} \big( \{ (x_k, x_j) \} \big)$.

In the other direction, the uniqueness statement follows from Lemma \ref{lemma:F_characterizes}.

If $F$ is a function such that all of the conditions of the lemma hold, then we will define $\rho^{(2)}$ pointwise on  $\{(x_k, x_j)\}_{k,j=-1}^\infty$ (where $x_{-1}=\infty$, c.f.\ \eqref{x_minus_one}) and prove that $\rho^{(2)}$ is a probability measure, it is in $\mathcal{P}^{(2)}_\theta$ and $F \equiv F_{\rho^{(2)}}$ holds.
For every $j, k \in \mathbb{N}$ let
\begin{align}
&\rho^{(2)}(x_k, x_j) := -F(x_k, x_j) + F(x_{k+1}, x_j) + F(x_k, x_{j+1}) - F(x_{k+1}, x_{j+1}),\\
&\rho^{(2)}(\infty, x_k) := F(x_k, x_0) - F(x_{k+1}, x_0),\\
&\rho^{(2)}(x_k, \infty) := F(x_k, x_0) - F(x_{k+1}, x_0),\\
&\rho^{(2)}(\infty, \infty) := 1 - F(x_0, x_0).
\end{align}
The non-negativity of $\rho^{(2)}$ follows from conditions 1., 3.\ and 5., moreover $\rho^{(2)}$ is trivially symmetric, since $F$ is also symmetric by 4.

The marginals of $\rho^{(2)}$:
\begin{align}
\bullet &\sum \limits_{i=-1}^\infty \rho^{(2)}(x_k, x_i) = F(x_k, x_0) - F(x_{k+1}, x_0) +\nonumber\\
&+ \sum \limits_{i=0}^\infty (-F(x_k, x_i) + F(x_{k+1}, x_i) + F(x_k, x_{i+1}) - F(x_{k+1}, x_{i+1})) =\nonumber\\
&= F(x_k, x_0) - F(x_{k+1}, x_0) - F(x_k, x_0) + \lim \limits_{i \to \infty} F(x_k, x_i) + F(x_{k+1}, x_0) - \nonumber\\
&- \lim \limits_{i \to \infty} F(x_{k+1}, x_i) \stackrel{2.}{=} \frac{(1-\theta) \cdot x_k}{1+\theta} \stackrel{\eqref{x_and_c_def}}{=} c_k, \qquad  k \in \mathbb{N}\\
\bullet &\sum \limits_{i=-1}^\infty \rho^{(2)}(\infty, x_i) = 1 - F(x_0, x_0) + \sum \limits_{k=0}^\infty (F(x_k, x_0) - F(x_{k+1}, x_0)) =\nonumber\\
&= 1 - F(x_0, x_0) + F(x_0, x_0) - \lim \limits_{k \to \infty} F(x_k, x_0) \stackrel{2.}{=}1-\frac{x_0}{1+\theta} = c_{-1}
\end{align}
So the measure $\rho^{(2)}$ has marginal distributions $\rho$ defined as in \eqref{rho}.
In particular, $\rho^{(2)}$ is a probability measure on $I^2$.
We still have to check that $F \equiv F_{\rho^{(2)}}$ holds:
\begin{multline}
 F_{\rho^{(2)}}(x_k, x_j) \stackrel{\eqref{rho->F}}{=} \rho^{(2)} \big( \{ [0,x_k] \times I\} \cup \{ I \times [0,x_j] \} \big)
 = \sum \limits_{i=k}^\infty c_i + \sum \limits_{l=j}^\infty \sum \limits_{i=-1}^{k-1} \rho^{(2)} (x_i,x_l) =  \\
 \frac{x_k}{1+\theta} + \sum \limits_{l=j}^\infty (-F(x_0, x_l) + F(x_k, x_l) + F(x_0, x_{l+1}) - F(x_k, x_{l+1}))
+ F(x_0, x_j) - \lim \limits_{j \to \infty} F(x_0, x_j) = \\ \frac{x_k}{1+\theta} - F(x_0, x_j) + \lim \limits_{l \to \infty} F(x_0, x_l)
+ F(x_k, x_j) - \lim \limits_{l \to \infty} F(x_k, x_l) + F(x_0, x_j) - \lim \limits_{j \to \infty} F(x_0, x_j) \stackrel{2.}{=}\\ F(x_k, x_j), \qquad  j, k \in \mathbb{N}.
\end{multline}
\end{Proof}

\begin{definition}[Scale invariant bivariate function]
 $F: \{x_k\}_{k=0}^\infty \times \{x_k\}_{k=0}^\infty \to [0, \infty)$ is a scale invariant bivariate function if
\begin{equation}\label{F_scale_inv}
F(x_{k+l}, x_{j+l}) = \theta^l F(x_k, x_j), \quad j, k, l \in \mathbb{N}.
\end{equation}
\end{definition}

If $F$ is scale invariant then for every $0 \le j \le k$ we have
\begin{equation}\label{F_scale_inv_remark}
 F(x_k, x_j) = \theta^j F(x_{k-j}, x_0).
\end{equation}

Recall the notation of $\Mi^{(2)}_\tet$ from below (\ref{rhosca}) as well as that of $\mathcal{P}^{(2)}_\theta$ from Definition \ref{biv_with_rho_marginals}.
\begin{lemma}[Scale invariant measures and functions]\label{lemma:scale_inv} Let  $\rho^{(2)} \in \mathcal{P}^{(2)}_\theta$.
$\rho^{(2)} \in \Mi^{(2)}_\tet$ holds if and only if $F_{\rho^{(2)}}$ defined in Definition \ref{def:rho->F} is a scale invariant function.
\end{lemma}
\begin{Proof} First note that if $F_{\rho^{(2)}}$  is a scale invariant function then
\begin{equation}
  \rho^{(2)}\big([0,\theta^n]\times I\cup I\times[0,\theta^n]\big) \stackrel{ \eqref{rho->F}  }{=} F_{\rho^{(2)}}(x_n,x_n) \stackrel{ \eqref{F_scale_inv} }{=}
  \theta^n F_{\rho^{(2)}}(x_0,x_0)  \leq \theta^n, \qquad n \geq 0.
\end{equation}
 Together with our assumption that both of the marginals of $\rho^{(2)}$ are $\rho$, this implies that we have $\rho^{(2)}\big([0,t]\times I\cup I\times[0,t]\big) \leq t $
for all $0 \leq t \leq 1$, thus by \eqref{Mi2def} we have $\rho^{(2)} \in \Mi^{(2)}$.

It remains to show that if $\rho^{(2)} \in \Mi^{(2)} \cap \mathcal{P}^{(2)}_\theta$ then $F_{\rho^{(2)}}$ is scale invariant if and only if $\Ga_\tet^{(2)} \rho^{(2)}=\rho^{(2)}$.
Let  $\hat{\rho}^{(2)}:=\Ga_\tet^{(2)}\rho^{(2)}$. By the scale invariance of the marginal distribution (c.f.\ \eqref{rhosca})  we have
$\hat{\rho}^{(2)} \in \mathcal{P}^{(2)}_\theta$. Thus by Lemma \ref{lemma:F_characterizes} we only need to prove that $F_{\hat{\rho}^{(2)}} \equiv F_{\rho^{(2)}}$ holds if and only
$ \theta^{-1} F_{\rho^{(2)}}(x_{k+1}, x_{j+1}) =  F_{\rho^{(2)}}(x_k, x_j)$ holds for any $ k,j \in \mathbb{N}$ (i.e., $F_{\rho^{(2)}}$ satisfies the $l=1$ case of \eqref{F_scale_inv}).  This equivalence follows as soon as we observe that we have
\begin{multline}
F_{\hat{\rho}^{(2)}}(x_k,x_j)\stackrel{ \eqref{rho->F} }{=} \Ga_\tet^{(2)}\rho^{(2)} \big( \{ [0,x_k] \times I\} \cup \{ I \times [0,x_j] \} \big)
\stackrel{ \eqref{Ga2} }{=} \\
 \theta^{-1} \rho^{(2)} \big( (\psi^{(2)}_\theta)^{-1} \left( \{ [0,x_k] \times I\} \cup \{ I \times [0,x_j] \} \right)   \big) \stackrel{(*)}{=}
  \theta^{-1} \rho^{(2)}\big( \{ [0,\theta x_k] \times I\} \cup \{ I \times [0, \theta x_j] \} \big)\stackrel{\eqref{x_and_c_def}}{=} \\
   \theta^{-1} \rho^{(2)}\big( \{ [0, x_{k+1}] \times I\} \cup \{ I \times [0,  x_{j+1}] \} \big)\stackrel{ \eqref{rho->F} }{=}
   \theta^{-1}F_{\rho^{(2)}}(x_{k+1},x_{j+1}), \qquad  k,j \in \mathbb{N},
\end{multline}
where $(*)$ holds since we defined $\psi^{(2)}_\theta : I^2\to I^2$ by
$\psi^{(2)}_\theta(x,x'):=\big(\psi_\theta(x),\psi_\theta(x')\big)$, where $\psi_\theta(x)= x/\theta$ if $x\leq \theta$ and $\psi_\theta(x)=\infty$ if $x>\theta$, cf.\ \eqref{psit}.
\end{Proof}

\begin{lemma}[Relationship between $F_{\rho^{(2)}}$ and the signature]\label{lemma:F_and_f}
If $\rho^{(2)} \in \mathcal{M}^{(2)}_{\theta}$, $F_{\rho^{(2)}}$ is the function defined in Definition \ref{def:rho->F} and $f_{\rho^{(2)}}$ is the signature of $\rho^{(2)}$ (c.f. Definition \ref{def:rho->f}), then
\begin{equation}\label{F_and_f}
f_{\rho^{(2)}}(n) =  F_{\rho^{(2)}}(x_n, x_0) = \frac{1}{x_{k \wedge j}} F_{\rho^{(2)}}(x_k, x_j)
\end{equation}
holds for every $j, k \in \mathbb{N}$ for which $n = |k-j|$.
\end{lemma}

\begin{Proof}
The first equality is trivial from the definition of $f_{\rho^{(2)}}$ and $F_{\rho^{(2)}}$. The second equality follows from the fact that $F_{\rho^{(2)}}$ is a scale invariant function (by Lemma \ref{lemma:scale_inv}), $x_{k \wedge j } =  \theta^{k \wedge j}$ (cf.\ \eqref{x_and_c_def}), \eqref{F_scale_inv_remark} and symmetry of $F_{\rho^{(2)}}$ (cf. Condition 4. of Lemma \ref{lemma:F_conditions}).
\end{Proof}

\begin{Proof}[of Lemma \ref{lemma:f_conditions}]
First let $\rho^{(2)} \in \mathcal{M}^{(2)}_{\theta}$ and $f_{\rho^{(2)}}$ be its signature.
Let us also define  $F_{\rho^{(2)}}$ as in \eqref{rho->F}, which means $f_{\rho^{(2)}}(n) =  F_{\rho^{(2)}}(x_n, x_0)$ by Lemma \ref{lemma:F_and_f}.

Let us now check that $f_{\rho^{(2)}}$ satisfies the properties (i)-(v) of Lemma \ref{lemma:f_conditions}.
\begin{enumerate}
\item $f_{\rho^{(2)}}(0) =  F_{\rho^{(2)}}(x_0, x_0) \le 1$  using condition 1. of Lemma \ref{lemma:F_conditions}.

\item $ \lim \limits_{n \to \infty} f_{\rho^{(2)}}(n)  = \frac{1}{1+\theta}$ by \eqref{F_and_f} and condition 2. of Lemma \ref{lemma:F_conditions}.

\item We have
$ f_{\rho^{(2)}}(n) = F_{\rho^{(2)}}(x_n, x_0) \ge  F_{\rho^{(2)}}(x_{n+1}, x_0) = f_{\rho^{(2)}}(n+1)$ for any $n \in \mathbb{N}$ by condition 3. of Lemma \ref{lemma:F_conditions}.

\item From condition 5. of Lemma \ref{lemma:F_conditions} we know
\begin{equation}\label{F_cond_from_rho_nonneg}
-F_{\rho^{(2)}}(x_k, x_j) + F_{\rho^{(2)}}(x_{k+1}, x_j) + F_{\rho^{(2)}}(x_k, x_{j+1}) - F_{\rho^{(2)}}(x_{k+1}, x_{j+1}) \ge 0.
\end{equation}

Using $F_{\rho^{(2)}}(x_k, x_j) = x_{k \wedge j} \cdot f_{\rho^{(2)}}(|k-j|)$ (c.f.\ \eqref{F_and_f}) and substituting $j:=k$ into \eqref{F_cond_from_rho_nonneg} we obtain
\begin{equation}
 - x_k \cdot f_{\rho^{(2)}}(0) + x_k \cdot f_{\rho^{(2)}}(1) + x_k \cdot f_{\rho^{(2)}}(1) - x_{k+1} \cdot f_{\rho^{(2)}}(0) \ge 0 .
\end{equation}
Dividing by $x_k$, after rearranging we get condition (iv).

\item If we use $F_{\rho^{(2)}}(x_k, x_j) = x_{k \wedge j} \cdot f_{\rho^{(2)}}(|k-j|)$ in \eqref{F_cond_from_rho_nonneg} if $k > j$, we get
\begin{equation}\label{four_terms_signiture}
-x_j \cdot f_{\rho^{(2)}}(k-j) + x_j \cdot f_{\rho^{(2)}}(k+1-j) + x_{j+1} \cdot f_{\rho^{(2)}}(k-j-1)
- x_{j+1} \cdot f_{\rho^{(2)}}(k-j) \ge 0 .
\end{equation}
Let $n := k-j$. If we divide \eqref{four_terms_signiture} by $x_j$, after rearranging we get the required inequality.
\end{enumerate}

In the other direction, assume that $f: \mathbb{N} \to \mathbb{R}$ satisfies conditions (i)-(v) of Lemma \ref{lemma:f_conditions}. Our goal is to show that
there exists a unique probability measure $\rho^{(2)} \in \mathcal{M}^{(2)}_{\theta}$ such that $f$ is its signature. As a first step, we define
\begin{equation}\label{f_F_again}
F(x_k, x_j) := x_{k \wedge j} \cdot f(|k-j|), \qquad j, k \in \mathbb{N}
\end{equation}
 and show that the conditions of Lemma \ref{lemma:F_conditions} hold for $F$.
Conditions (i), (ii), (iii) of Lemma \ref{lemma:f_conditions} on $f$ imply respectively conditions 1., 2.\ and 3.\ of Lemma \ref{lemma:F_conditions}.
 Condition 4.\ of Lemma \ref{lemma:F_conditions} of $F$ is straightforward from \eqref{f_F_again}.
Condition 5.\ of Lemma \ref{lemma:F_conditions} follows from condition (iv) of Lemma \ref{lemma:f_conditions} (in the $k=j$ case) and condition (v) of Lemma \ref{lemma:f_conditions} (in the $k > j$ case, and, by symmetry, in the $j > k$ case).

We can thus apply Lemma \ref{lemma:F_conditions} to infer that there exists a probability measure $\rho^{(2)} \in \mathcal{P}^{(2)}_\theta$ such that $F \equiv F_{\rho^{(2)}}$ holds.
In fact $\rho^{(2)} \in \mathcal{M}^{(2)}_{\theta}$ by Lemma \ref{lemma:scale_inv} and the scale invariance of $F$ (c.f.\ \eqref{F_scale_inv}), which is straightforward from \eqref{f_F_again}.
Finally $f \equiv f_{\rho^{(2)}}$ follows from Definition \ref{def:rho->f}, Lemma \ref{lemma:F_and_f} and \eqref{f_F_again}.
   Uniqueness is clear since the signature of a bivariate measure in
$\mathcal{M}^{(2)}_{\theta}$ uniquely determines its bivariate signature by Lemma \ref{lemma:F_and_f}, which in turn uniquely determines the bivariate measure by Lemma \ref{lemma:F_characterizes}.

\end{Proof}

\subsection{\texorpdfstring{Basic properties of $f_{\theta, c}(n)$}{Basic properties of f theta c(n)}}\label{sec:f_unique}

The main goal of Section \ref{sec:f_unique} is to prove Lemmas \ref{lemma:imp_eq_for_f_theta_c} and \ref{lemma:f_continuous}, but we will also collect some other useful properties of $f_{\theta, c}(n)$ in Corollary \ref{corr_f_inc_dec}. We will first define an auxiliary function $g_{\theta, c}(n), n \in \mathbb{N}$ and later we will identify $f_{\theta, c}(n)$ as $f_{\theta, c}(n)=\theta^n g_{\theta, c}(n)$. In order to construct $g_{\theta, c}(n)$, we need the following definition.

\begin{definition}[Recursion map $\psi_{\theta,c}$]\label{def_psi}
 Given some $\theta \in (0,1)$ and $c \geq 0$,  let us define the function $\psi_{\theta,c}: \mathcal{D}_{\theta,c} \to \mathbb{R}$  by
 \begin{equation}\label{def_eq__psi_theta_c}
  \psi_{\theta,c}(x)=\frac{1+\sqrt{(2x-1)^2-4c \cdot (1-\theta^2)}}{2\theta}, \quad \mathcal{D}_{\theta, c}=(\sqrt{(1-\theta^2)c}+1/2,  +\infty).
 \end{equation}
\end{definition}
Note that $x \in \mathcal{D}_{\theta,c} $ if and only if $2x-1\geq 0$ and  $(2x-1)^2-4c \cdot (1-\theta^2) > 0$.

\begin{lemma}[Recursive definition of $g_{\theta, c}(n)$]\label{lemma_rec_def_of_g}
For any $\theta \in (0,1)$ and $c \geq 0$, the recursion
\begin{equation}\label{g_recursion}
  g_{\theta, c}(0)=\frac{1+\sqrt{1+8c(1+\theta)^2}}{2(1+\theta)}, \quad
  g_{\theta, c}(n)= \psi_{\theta,c}\big(g_{\theta, c}(n-1)\big), \quad n \geq 1
\end{equation}
has a solution (i.e., $g_{\theta, c}(n) \in \mathcal{D}_{\theta,c}$ holds for all $n \geq 0$). Moreover, the solution satisfies
\begin{equation}
\label{g_increasing}
 g_{\theta, c}(n)  \geq g_{\theta, c}(n-1), \qquad n \geq 1.
\end{equation}
\end{lemma}
\begin{proof} We first check that $g_{\theta, c}(0) \in \mathcal{D}_{\theta,c}$ holds. In order to do so, let us denote $\gamma=4c(1+\theta)^2$.
After some rearrangements, we only need to check that $\sqrt{1+2\gamma}> \sqrt{(1-\theta^2)\gamma}+\theta$ holds. Taking the square of both sides and rearranging,
we want to show $(1-\theta^2)+(1+\theta^2)\gamma>2\theta\sqrt{(1-\theta^2)\gamma}$. Taking the square of both sides again and rearranging, we need
\begin{equation}
  (1-\theta^2)^2 + 2(1-\theta^2)^2\gamma +(1+\theta^2)^2 \gamma^2 >0,
\end{equation}
and this inequality indeed holds, since all of the terms are non-negative for any choice of $\theta \in (0,1)$ and $c \geq 0$. We have thus checked $g_{\theta, c}(0) \in \mathcal{D}_{\theta,c}$.

 Next we observe that $x \mapsto  \psi_{\theta,c}(x)$ is an increasing and concave function of $x \in \mathcal{D}_{\theta,c}$, moreover $\frac{\mathrm{d}}{\mathrm{d}x}\psi_{\theta,c}(x) > 1/\theta >1$ holds for all $x \in \mathcal{D}_{\theta,c}$. This implies that the equation
$\psi_{\theta,c}(x)=x$ has at most one solution in $\mathcal{D}_{\theta,c}$.
 Let $y_0:=\sqrt{(1-\theta^2)c}+1/2$ denote the left endpoint of $\mathcal{D}_{\theta, c}$. One easily checks that $ \psi_{\theta,c}(y_0)\geq  y_0$ holds if and only if $c \leq \frac{1-\theta}{4 \theta^2 (1+\theta)}$ holds. We will prove Lemma \ref{lemma_rec_def_of_g} by treating the cases $ \psi_{\theta,c}(y_0) \geq y_0$ and $ \psi_{\theta,c}(y_0)< y_0$ separately.

If $ \psi_{\theta,c}(y_0) \geq y_0$ then $ \psi_{\theta,c}(x)>x$  for every $x \in \mathcal{D}_{\theta,c}$ follows from the above listed properties of $\psi_{\theta,c}$.
Now it follows from \eqref{g_recursion} by induction on $n$ that $g_{\theta, c}(n) \in \mathcal{D}_{\theta,c}$ and \eqref{g_increasing} hold for all $n \geq 0$.

If $ \psi_{\theta,c}(y_0)< y_0$ then the above listed properties of $\psi_{\theta,c}$ imply that there exists a unique $x^*_0 \in \mathcal{D}_{\theta,c}$ for which $\psi_{\theta,c}(x^*_0)=x^*_0$, moreover $x\geq x^*_0$ implies
$\psi_{\theta,c}(x)\geq x$. One checks that $x^*_0:=\frac{1+\sqrt{1+4c(1+\theta)^2}}{2(1+\theta)}$, thus  $g_{\theta, c}(0) \geq x^*_0$ holds.
It is enough to prove that $g_{\theta, c}(n) \geq x^*_0$ for all $n \geq 0$ to conclude that $g_{\theta, c}(n) \in \mathcal{D}_{\theta,c}$  for all $n \geq 0$.
  Now both $g_{\theta, c}(n) \geq x^*_0$  and \eqref{g_increasing} follow by induction on $n$  using the recursive  definition \eqref{g_recursion} of $g_{\theta, c}(n)$.
 \end{proof}

Now we are ready to prove the existence and uniqueness of $f_{\theta,c}(n), n \in \mathbb{N}$.

\begin{proof}[Proof of Lemma \ref{lemma:imp_eq_for_f_theta_c}] We will show by induction on $n$ that
\begin{equation}\label{def_eq_g}
f_{\theta, c}(n) = g_{\theta, c}(n) \theta^n , \qquad n \in \mathbb{N}
\end{equation}
is the unique solution of the system of equations \eqref{quadratic_eq_for_f_0}-\eqref{f_theta_c_n_geq_ugly}.
The induction hypothesis holds for $n=0$, since \eqref{quadratic_eq_for_f_0} is a quadratic equation for $f_{\theta,c}(0)$, which has two solutions,
one of them is equal to $g_{\theta, c}(0) \theta^0$, while the other solution is less then or equal to zero, therefore only
$g_{\theta, c}(0) \theta^0$ satisfies \eqref{f_theta_c_n_geq_ugly} for $n=0$.

 Now assume that $n \geq 1$ and \eqref{def_eq_g} holds for $n-1$, i.e.,
we have $f_{\theta, c}(n-1) = g_{\theta, c}(n-1) \theta^{n-1}$. We can view \eqref{quadratic_eq_for_f_n_and_n_minus_1} as a quadratic equation for $f_{\theta,c}(n)$ which has two solutions:
\begin{equation}\label{two_solutions-for_f_n}
 \widetilde{x}_{1,2}= \frac{\theta^{n-1} \pm \sqrt{(2f_{\theta, c}(n-1) - \theta^{n-1})^2 - 4c \cdot \theta^{2n-2} \cdot (1-\theta^2)}}{2}.
\end{equation}
Now $\widetilde{x}_1=\theta^n \psi_{\theta,c}\left(g_{\theta, c}(n-1)\right)=g_{\theta, c}(n) \theta^n $ follows from $f_{\theta, c}(n-1) = g_{\theta, c}(n-1) \theta^{n-1}$, \eqref{def_eq__psi_theta_c} and \eqref{g_recursion}, moreover $\widetilde{x}_1 > \theta^{n-1}/2$, while $\widetilde{x}_2 < \theta^{n-1}/2$, thus only $\widetilde{x}_1$ satisfies
\eqref{f_theta_c_n_geq_ugly} and therefore  \eqref{def_eq_g} holds.
\end{proof}

\begin{corollary}[Recursion for $f_{\theta,c}$]\label{corr_f_inc_dec} For any $\theta \in (0,1]$ and $c \geq 0$ we have
\begin{align}
&f_{\theta, c}(0) = \frac{1+\sqrt{1+8c(1+\theta)^2}}{2(1+\theta)}, \label{f0_with_c}\\
&f_{\theta, c}(n) = \frac{\theta^{n-1} + \sqrt{(2f_{\theta, c}(n-1) - \theta^{n-1})^2 - 4c \cdot \theta^{2n-2} \cdot (1-\theta^2)}}{2}, \quad n \ge 1. \label{fn_with_c}
\end{align}
Moreover, the function $f_{\theta, c}(n)$ decreases in $n$:
\begin{equation}
\label{f_decreasing}
f_{\theta, c}(n)  \le  f_{\theta, c}(n-1), \quad n \geq 1.
\end{equation}
\end{corollary}
\begin{Proof}
The identities \eqref{f0_with_c} and \eqref{fn_with_c} follow from \eqref{g_recursion} and \eqref{def_eq_g}.

 In order to prove \eqref{f_decreasing},
we  need to show that $g_{\theta, c}(n)\leq g_{\theta, c}(n-1)/\theta  $ holds for any $n \geq 1$: this inequality follows from the fact that
$\psi_{\theta,c}(x)\leq \psi_{\theta,0}(x) = x/\theta$ holds for any $x \in \mathcal{D}_{\theta,c}$.
\end{Proof}

\begin{Proof}[of Lemma \ref{lemma:f_continuous}] The limit $ f_{\theta, c}(\infty)= \lim_{n \to \infty} f_{\theta, c}(n)$ exists since $f_{\theta, c}(n) $  decreases as $n$ increases (c.f.\ \eqref{f_decreasing}) and $f_{\theta, c}(n) \geq 0$. It follows from \eqref{f0_with_c} and \eqref{fn_with_c} by induction on $n$ that for each $n$ the function
$c \mapsto f_{\theta, c}(n)$ is continuous. Thus, in order to prove that $c \mapsto f_{\theta, c}(\infty)$ is continuous, we only need to check that the functions $c \mapsto f_{\theta, c}(n)$ converge uniformly as $n \to \infty$ on $[0,c_0]$ for any $0 \leq c_0 <+\infty$.
 In order to achieve this, we will show
\begin{equation}\label{cauchy_ingredient}
 f_{\theta, c}(n-1)-f_{\theta, c}(n) \leq \frac{1}{2} \sqrt{4c \theta^{2n-2}(1-\theta^2)}, \qquad n \geq 1, c \geq 0.
\end{equation}
By \eqref{def_eq_g} we only need to prove $g_{\theta,c}(n-1)-\theta g_{\theta,c}(n)  \leq \frac{1}{2} \sqrt{4c (1-\theta^2)} $. By
\eqref{def_eq__psi_theta_c} and \eqref{g_recursion} it is enough to show that for all $ x \in \mathcal{D}_{\theta, c}$
we have $(2x-1)-\sqrt{(2x-1)^2 -4c(1-\theta^2) } \leq \sqrt{4c(1-\theta^2)}$, but this inequality easily follows using the properties
of $\mathcal{D}_{\theta, c}$ listed below \eqref{def_eq__psi_theta_c}.

 It follows from \eqref{cauchy_ingredient} that  for any $c_0>0$ the functions $c \mapsto f_{\theta, c}(n)$ form a Cauchy
sequence with respect to the sup-norm on $[0,c_0]$. From this the desired uniform convergence readily follows.
\end{Proof}

\subsection{Signature of a scale invariant solution of the bivariate RDE}\label{sec:equivalent_system}

\paragraph{}Our next goal is to prove Lemma \ref{lemma:equivalent_system}. First we show a formula in Lemma \ref{lemma:F_recursion}, which characterizes the distribution of the right-hand side of the bivariate RDE \eqref{bivar_RDE} in terms of the bivariate signature $F_{\rho^{(2)}}$ (c.f. Definition \ref{def:rho->F}). Using this we get an equation for the univariate signature $f_{\rho^{(2)}}$ of a scale invariant measure $\rho^{(2)}$ in Lemma \ref{lemma:f_recursion}, which holds if and only if $\rho^{(2)}$ is a solution of the bivariate RDE.
In Lemma \ref{lemma_properties_of_f_theta_c} we show that $f_{\theta,c}$ (defined in Lemma \ref{lemma:imp_eq_for_f_theta_c}) satisfies a very similar equation.
Finally we conclude the proof of Lemma \ref{lemma:equivalent_system}.

Recall the notion of $\mathcal{P}^{(2)}_\theta$ from Definition \ref{biv_with_rho_marginals}.

\begin{definition}[Definition of $\tilde{\rho}^{(2)}$]\label{def:rho_tilde}
Let $\theta \in (0,1)$, $\rho^{(2)} \in \mathcal{P}^{(2)}_\theta$. Denote by $\tilde{\rho}^{(2)}$ the law of
\begin{equation}
(\chi[\tau, \kappa](Y_1,Y_2), \chi[\tau, \kappa](Y_1^*,Y_2^*)),
\end{equation}
where the function $\chi$ is defined in \eqref{chi_def} and the other notation are defined in \eqref{bivar_RDE}, so $(Y_1, Y_1^*) \sim \rho^{(2)}$, $(Y_2, Y_2^*) \sim \rho^{(2)}$, $\tau \sim \text{UNI}[0,1]$, $\kappa$ is a random variable such that $\p(\kappa = 1) = \p(\kappa = 2) = \frac 12$ and $(Y_1, Y_1^*), (Y_2, Y_2^*), \tau$ and $\kappa$ are mutually independent.
\end{definition}

\begin{lemma}[Expressing $F_{\tilde{\rho}^{(2)}}$ in terms of $F_{\rho^{(2)}}$]\label{lemma:F_recursion}
If $\theta \in (0,1), \rho^{(2)} \in \mathcal{P}^{(2)}_\theta$, then for every $j, k \in \mathbb{N}$ we have
\begin{multline} \label{F_recursion}
F_{\tilde{\rho}^{(2)}}(x_k, x_j)
= F_{\rho^{(2)}}(x_k, x_j) - \frac 12 F_{\rho^{(2)}}(x_k, x_j)^2 + \frac{x_k^2}{2(1+\theta)^2} + \frac{x_j^2}{2(1+\theta)^2} +\\
+ \frac{ 1-\theta}{2\theta} \sum \limits_{t= k \vee j + 1}^\infty \bigg[ \theta^t \Big( F_{\rho^{(2)}}(x_k, x_j) - F_{\rho^{(2)}}(x_t, x_j)
- F_{\rho^{(2)}}(x_k, x_t) + F_{\rho^{(2)}}(x_t, x_t) \Big) \bigg],
\end{multline}
where $F_{\rho^{(2)}}$, $F_{\tilde{\rho}^{(2)}}$ are the bivariate signatures of $\rho^{(2)}$ and $\tilde{\rho}^{(2)}$ respectively (c.f.\ Definition \ref{def:rho->F}).
\end{lemma}

\begin{Proof}
Let us use the notation
\begin{equation}
(\tilde{Y}, \tilde{Y}^*) := (\chi[\tau, \kappa](Y_1,Y_2), \chi[\tau, \kappa](Y_1^*,Y_2^*)), \text{ so that } (\tilde{Y}, \tilde{Y}^*) \sim \tilde{\rho}^{(2)}.
\end{equation}
 Let us also use the shorthand $F=F_{\rho^{(2)}}$, $\widetilde{F}=F_{\tilde{\rho}^{(2)}}$ in this proof.
Thus we have
\begin{equation}\label{kappa_1_or_2}
\widetilde{F}(x_k,x_j)= \p(\kappa = 1, \tilde{Y} \leq x_k \text{ or } \, \tilde{Y}^* \leq x_j) + \p(\kappa = 2, \tilde{Y} \leq x_k \text{ or } \, \tilde{Y}^* \leq x_j).
\end{equation}

Here
\begin{multline}\label{incl_excl}
\p(\kappa = 1, \tilde{Y} \leq x_k \text{ or } \, \tilde{Y}^* \leq x_j) = \\
\frac 12 \big[ \p(\tilde{Y} \le x_k \, |\, \kappa = 1) +  \p(\tilde{Y}^* \le x_j\, |\, \kappa = 1) - \p(\tilde{Y} \le x_k, \tilde{Y}^* \le x_j \,|\, \kappa = 1)\big].
\end{multline}

By the definition of $\chi$ in \eqref{chi_def} we will calculate all the three terms on the r.h.s.\ of \eqref{incl_excl}.
\begin{align}
&\p(\tilde{Y} \le x_k \, |\, \kappa = 1) = \p(Y_1 \le x_k, Y_1 > \tau) = \sum \limits_{l=k}^\infty \p(Y_1 = x_l, x_l > \tau)  \nonumber\\
&= \sum \limits_{l=k}^\infty x_l \cdot c_l = \sum \limits_{l=k}^\infty  \frac{1-\theta}{1+\theta} \cdot \theta^{2l} = x_k^2 \cdot \frac{1-\theta}{1+\theta} \cdot \sum \limits_{l=0}^\infty \theta^{2l} = \frac{x_k^2}{(1+\theta)^2}.
\end{align}

Similarly, we have
\begin{align}
&\p(\tilde{Y}^* \le x_j \, | \, \kappa = 1) = \frac{x_j^2}{(1+\theta)^2}.
\end{align}

\begin{multline}
\p(\tilde{Y} \le x_k, \tilde{Y}^* \le x_j \, |\, \kappa = 1) = \p(Y_1 \le x_k, Y_1^* \le x_j, \tau < Y_1 \wedge Y_1^*) = \\
 \sum \limits_{t = k \vee j + 1}^\infty \p(x_t < Y_1 \le x_k, x_t < Y_1^* \le x_j, \tau \in [x_t, x_{t-1})) = \\
 \sum \limits_{t = k \vee j + 1}^\infty ((F(x_t, x_j) + F(x_k, x_t) - F(x_k, x_j) - F(x_t, x_t) )  (x_{t-1} - x_t) = \\
  \frac{\theta-1}{\theta}  \sum \limits_{t= k \vee j + 1}^\infty \bigg[ \theta^t \Big( F(x_k, x_j) - F(x_t, x_j)
- F(x_k, x_t) + F(x_t, x_t) \Big) \bigg].
\end{multline}

Now we calculate the other term of \eqref{kappa_1_or_2}:
\begin{align}
\p(\kappa &= 2, \tilde{Y} \leq x_k \text{ or } \, \tilde{Y}^* \leq x_j) = \frac 12 \cdot \big( 1 - \p(Y_1 \wedge Y_2 > x_k, Y_1^* \wedge Y_2^* > x_j)\big)  \nonumber \\
& \stackrel{(*)}{=} \frac 12 \cdot \big( 1 - (1-F(x_k, x_j))^2 \big)= F(x_k, x_j) - \frac 12 \cdot F(x_k, x_j)^2, \label{kappa_2_case}
\end{align}
where $(*)$ holds by the independence of $(Y_1, Y^*_1 )$ and $(Y_2, Y^*_2 )$.
Now \eqref{F_recursion} follows if we substitute \eqref{incl_excl}-\eqref{kappa_2_case} into \eqref{kappa_1_or_2}.
\end{Proof}

Recall the notation of $\Mi^{(2)}_\tet$ from below (\ref{rhosca}).

\begin{lemma}[Equation for the signature]\label{lemma:f_recursion}
 $\rho^{(2)} \in \mathcal{M}^{(2)}_{\theta}$ is a solution of the bivariate RDE \eqref{bivar_RDE} if and only if its signature $f_{\rho^{(2)}}$ satisfies
\begin{multline}
f_{\rho^{(2)}}(n)^2 =
 \frac{1}{(1+\theta)^2} +
  \theta^n \cdot f_{\rho^{(2)}}(n) - (1-\theta) \cdot \sum \limits_{t=n}^\infty \theta^{t} f_{\rho^{(2)}}(t+1) +
  \\
+\left( \frac{1}{(1+\theta)^2} + \frac{\theta \cdot f_{\rho^{(2)}}(0)}{1+\theta} - (1-\theta) \cdot \sum \limits_{t=0}^\infty \theta^t f_{\rho^{(2)}}(t+1) \right) \cdot \theta^{2n}. \label{f_recursion}
\end{multline}
\end{lemma}

\begin{Proof} By
 Lemma \ref{lemma:F_characterizes} the measure  $\rho^{(2)}$ is a solution of the bivariate RDE \eqref{bivar_RDE} if and only if $F_{\tilde{\rho}^{(2)}}=F_{\rho^{(2)}}$.
 We will prove that if $\rho^{(2)} \in \mathcal{M}^{(2)}_{\theta}$ then \eqref{f_recursion} holds if and only if $F_{\tilde{\rho}^{(2)}}=F_{\rho^{(2)}}$.

We have $F_{\tilde{\rho}^{(2)}}=F_{\rho^{(2)}}$ in \eqref{F_recursion} if and only if
\begin{multline}\label{F_fix}
F_{\rho^{(2)}}(x_k, x_j)^2 = \frac{x_k^2}{(1+\theta)^2} + \frac{x_j^2}{(1+\theta)^2} + \\ \frac{ 1-\theta}{\theta} \sum \limits_{t= k \vee j + 1}^\infty \bigg[ \theta^t \Big( F_{\rho^{(2)}}(x_k, x_j)
- F_{\rho^{(2)}}(x_t, x_j) - F_{\rho^{(2)}}(x_k, x_t) + F_{\rho^{(2)}}(x_t, x_t) \Big) \bigg]
\end{multline}
holds for every $j, k \in \mathbb{N}$.

By symmetry of $F_{\rho^{(2)}}$ we can assume that $0 \le j \le k$ (so $x_k \le x_j$) and let $n := k-j$. We have  $F_{\rho^{(2)}}(x_l, x_i) = x_{l \wedge i} f(|l-i|), l, i \in \mathbb{N}$ (c.f.\ \eqref{F_and_f}). Hence
\begin{align}
F_{\rho^{(2)}}(x_k, x_j) &= x_j \cdot f_{\rho^{(2)}}(n),
\\
\sum \limits_{t= k \vee j + 1}^\infty \theta^t F_{\rho^{(2)}}(x_k, x_j) &= \sum \limits_{t= k+1}^\infty \theta^t \cdot x_j \cdot f_{\rho^{(2)}}(n) =
 \frac{x_{k+j+1} \cdot f_{\rho^{(2)}}(n)}{1-\theta},\\
\notag \sum \limits_{t= k \vee j + 1}^\infty \theta^t F_{\rho^{(2)}}(x_t, x_j) &= \sum \limits_{t= k+1}^\infty \theta^t \cdot x_j \cdot f_{\rho^{(2)}}(t-j) =\\
&= x_{k+j+1} \cdot \sum \limits_{t=0}^\infty \theta^t f_{\rho^{(2)}}(t+n+1)
= x_{k+j+1} \cdot \sum \limits_{t=n}^\infty \theta^{t-n} f_{\rho^{(2)}}(t+1),\\
\sum \limits_{t= k \vee j + 1}^\infty \theta^t F_{\rho^{(2)}}(x_k, x_t) &= \sum \limits_{t= k+1}^\infty \theta^t  x_k  f_{\rho^{(2)}}(t-k) = x_{2k+1}  \sum \limits_{t=0}^\infty \theta^t f_{\rho^{(2)}}(t+1),\\
\sum \limits_{t= k \vee j + 1}^\infty \theta^t F_{\rho^{(2)}}(x_t, x_t) &=
 \sum \limits_{t= k+1}^\infty \theta^t \cdot x_t \cdot f_{\rho^{(2)}}(0) = \frac{x_{2(k+1)} \cdot f_{\rho^{(2)}}(0)}{1-\theta^2}.
\end{align}

If we substitute all of the above into \eqref{F_fix} and divide by $x_j^2 =  \theta^{2j}$, we get:
\begin{multline}
f_{\rho^{(2)}}(n)^2 = \frac{\theta^{2n}}{(1+\theta)^2} + \frac{1}{(1+\theta)^2} + \theta^n f_{\rho^{(2)}}(n) -  (1-\theta)
 \theta^n
 \cdot \sum \limits_{t=n}^\infty \theta^{t-n} f_{\rho^{(2)}}(t+1)
- \\
(1-\theta) \cdot \theta^{2n} \cdot \sum \limits_{t=0}^\infty \theta^t f_{\rho^{(2)}}(t+1) + \frac{\theta^{2n+1} f_{\rho^{(2)}}(0)}{1+\theta}.
\end{multline}
We get \eqref{f_recursion} after rearranging above formula, so $F_{\tilde{\rho}^{(2)}}=F_{\rho^{(2)}}$ in \eqref{F_recursion} if and only if $f_{\rho^{(2)}}$ satisfies \eqref{f_recursion}, therefore we proved the lemma.
\end{Proof}

Next we consider the sequence $f_{\theta,c}(n), n\in \mathbb{N}$ (c.f.\ Lemma \ref{lemma:imp_eq_for_f_theta_c}) and derive some formulas analogous to \eqref{f_recursion}.

\begin{lemma}[Properties of $f_{\theta,c}$]\label{lemma_properties_of_f_theta_c}
Given some $\theta \in (0,1)$ and $c \geq 0$, let us assume that $\lim \limits_{n \to \infty} f_{\theta,c}(n) = \frac{1}{1+\theta}$ holds. Under these conditions we  have
\begin{equation}\label{f_theta_c_sum_eq}
f_{\theta,c}(n)^2 = \frac{1}{(1+\theta)^2} + \theta^n \cdot f_{\theta,c}(n) - (1-\theta) \cdot \sum \limits_{t=n}^\infty \theta^{t} f_{\theta,c}(t+1) + c \cdot \theta^{2n}, \quad n \in \mathbb{N},
\end{equation}
\begin{equation}\label{c_f_theta_c_formula}
c = \frac{1}{(1+\theta)^2} + \frac{\theta \cdot f_{\theta,c}(0)}{1+\theta} - (1-\theta) \cdot \sum \limits_{t=0}^\infty \theta^t f_{\theta,c}(t+1),
\end{equation}
\begin{equation}\label{f_theta_c_0_in_interval}
  \frac{1}{1+\theta} \wedge \frac{2 \theta}{1+\theta} \leq f_{\theta,c}(0) \leq \frac{1}{1+\theta} \vee \frac{2 \theta}{1+\theta},
\end{equation}
\begin{equation}\label{c_leq_bbbound}
c \leq   0 \vee \frac{\theta \cdot (2\theta - 1)}{(1+\theta)^2}.
\end{equation}
\end{lemma}
\begin{Proof}
To prove \eqref{f_theta_c_sum_eq} let us denote by $\beta_n$ the difference of the r.h.s.\ and the l.h.s.\ of \eqref{f_theta_c_sum_eq}. Our goal is to show $\beta_n \equiv 0$.
For every $n \ge 1$ we  have
\begin{align}
\beta_n - \beta_{n-1} =& f_{\theta,c}(n-1)^2 - f_{\theta,c}(n)^2 - \theta^{n-1} f_{\theta,c}(n-1) + \theta^n f_{\theta,c}(n) +\nonumber\\
&+ (1-\theta) \theta^{n-1} f_{\theta,c}(n) - c \cdot \theta^{2n-2}(1 - \theta^2).\label{equivalent_system_other_direction}
\end{align}
The right-hand side of \eqref{equivalent_system_other_direction} is $0$ by \eqref{quadratic_eq_for_f_n_and_n_minus_1}.
Therefore the sequence $\beta_n$ is constant, but we also have $\lim \limits_{n \to \infty} \beta_n=0$ by the definition of $\beta_n$  and our assumption $\lim \limits_{n \to \infty} f_{\theta,c}(n) = \frac{1}{1+\theta}$. So $\beta_n \equiv 0$, thus we get \eqref{f_theta_c_sum_eq}.

Next we show \eqref{c_f_theta_c_formula}.
If we take \eqref{f_theta_c_sum_eq} at $n=0$, we obtain
\begin{equation}\label{c_equation_1}
f_{\theta,c}(0)^2 = \frac{1}{(1+\theta)^2} + f_{\theta,c}(0) - (1-\theta) \cdot \sum \limits_{t=0}^\infty \theta^{t} f_{\theta,c}(t+1) + c.
\end{equation}
If we take the difference of \eqref{quadratic_eq_for_f_0} and \eqref{c_equation_1} and rearrange, we get \eqref{c_f_theta_c_formula}.

Next we prove \eqref{f_theta_c_0_in_interval}. From our assumption $\lim \limits_{n \to \infty} f_{\theta,c}(n) = \frac{1}{1+\theta}$ and \eqref{f_decreasing}
we obtain that $f_{\theta,c}(n) \ge \frac{1}{1+\theta}$ for every $n \in \mathbb{N}$, hence
\begin{equation}
c  \stackrel{\eqref{c_f_theta_c_formula} }{\le} \frac{1}{(1+\theta)^2} + \frac{\theta  f_{\theta,c}(0)}{1+\theta} - (1-\theta) \sum \limits_{t=0}^\infty   \frac{\theta^t}{1+\theta} = \frac{\theta  ((1+\theta)  f_{\theta,c}(0) - 1)}{(1+\theta)^2}. \label{c_inequality}
\end{equation}
Putting together \eqref{quadratic_eq_for_f_0} and \eqref{c_inequality}, we obtain the inequality
\begin{equation}
 f_{\theta,c}(0)^2 - \frac{1}{1+\theta}f_{\theta,c}(0) \leq  2 \frac{\theta  ((1+\theta)  f_{\theta,c}(0) - 1)}{(1+\theta)^2} ,
\end{equation} which implies
\eqref{f_theta_c_0_in_interval}, since the roots of the polynomial $x^2 - \frac{x}{1+\theta}- 2 \frac{\theta  ((1+\theta)  x - 1)}{(1+\theta)^2} $ are $\frac{1}{1+\theta}$ and $\frac{2 \theta}{1+\theta}$.
Finally, \eqref{c_leq_bbbound} follows by plugging the upper bound of \eqref{f_theta_c_0_in_interval} into \eqref{c_inequality}.
\end{Proof}

\begin{Proof}[of Lemma \ref{lemma:equivalent_system}]
 First assume that $\rho^{(2)} \in \mathcal{M}^{(2)}_{\theta}$ is a solution of the bivariate RDE \eqref{bivar_RDE} and let us define
\begin{equation}\label{c_def}
c := \frac{1}{(1+\theta)^2} + \frac{\theta \cdot f_{\rho^{(2)}}(0)}{1+\theta} - (1-\theta) \cdot \sum \limits_{t=0}^\infty \theta^t f_{\rho^{(2)}}(t+1).
\end{equation}
We will prove that $f_{\rho^{(2)}}(n) = f_{\theta, c}(n)$ holds for this $c$ for every $n \in \mathbb{N}$. By Lemma \ref{lemma:imp_eq_for_f_theta_c}, it is enough
to show that $f_{\rho^{(2)}}(n)$ satisfies \eqref{quadratic_eq_for_f_0}-\eqref{f_theta_c_n_geq_ugly}.

By Lemma \ref{lemma:f_recursion} the equation \eqref{f_recursion} holds. If we plug the definition \eqref{c_def} of $c$ into equation \eqref{f_recursion}, we get
\begin{equation}\label{f_recursion_with_c}
f_{\rho^{(2)}}(n)^2 = \frac{1}{(1+\theta)^2} + \theta^n \cdot f_{\rho^{(2)}}(n) - (1-\theta) \cdot \sum \limits_{t=n}^\infty \theta^{t} f_{\rho^{(2)}}(t+1) + c \cdot \theta^{2n}.
\end{equation}

If we take \eqref{f_recursion_with_c} at $n = 0$,  subtract $\frac{1}{1+\theta} \cdot f_{\rho^{(2)}}(0)$ from both sides and again use the definition \eqref{c_def} of $c$, we get
$f_{\rho^{(2)}}(0)^2 - \frac{1}{1+\theta}f_{\rho^{(2)}}(0) = 2c$, i.e., that \eqref{quadratic_eq_for_f_0} holds. Now let $n \geq 1$.
 If we take the difference of \eqref{f_recursion_with_c} at $n-1$ and at $n$, we obtain
\begin{align}
&f_{\rho^{(2)}}(n-1)^2 - f_{\rho^{(2)}}(n)^2 =\\
&= \theta^{n-1} f_{\rho^{(2)}}(n-1) - \theta^n f_{\rho^{(2)}}(n) - (1-\theta) \theta^{n-1} f_{\rho^{(2)}}(n) + c \cdot \theta^{2n-2}(1 - \theta^2),\nonumber
\end{align}
therefore \eqref{quadratic_eq_for_f_n_and_n_minus_1} holds. Both inequalities required by \eqref{f_theta_c_n_geq_ugly} can be proved using $f_{\rho^{(2)}}(n) \geq \frac{1}{1+\theta}$
 (which holds by Lemma \ref{lemma:f_conditions}), also using  $\frac{1}{1+\theta} >\frac{1}{2}\geq  \frac{\theta^{n-1}}{2}$  in the proof of the $n \geq 1$ case of \eqref{f_theta_c_n_geq_ugly}.

We also need that $c \ge 0$: this follows from $f_{\rho^{(2)}}(0)^2 - \frac{1}{1+\theta}f_{\rho^{(2)}}(0) = 2c$ and $f_{\rho^{(2)}}(n) \geq \frac{1}{1+\theta}$.

In the other direction, we assume that for some $\rho^{(2)} \in \mathcal{M}^{(2)}_{\theta}$ we have $f_{\rho^{(2)}}(n) = f_{\theta, c}(n)$ for every $n \in \mathbb{N}$ for some $c \ge 0$, and we have to show that $\rho^{(2)}$ is a solution of the bivariate RDE. By Lemma \ref{lemma:f_recursion} it is enough to show that \eqref{f_recursion} holds for every $n \in \mathbb{N}$. In order to do so, we use Lemma \ref{lemma_properties_of_f_theta_c} (the conditions of which do hold, since $\lim \limits_{n \to \infty} f_{\rho^{(2)}}(n) = \frac{1}{1+\theta}$ by Lemma \ref{lemma:f_conditions}): the identity  \eqref{f_recursion} follows by putting  \eqref{f_theta_c_sum_eq} and \eqref{c_f_theta_c_formula} together. This completes the proof of statement (i) of Lemma \ref{lemma:equivalent_system}. Also,  \eqref{c_upper_bound} follows from \eqref{c_leq_bbbound}, i.e.,  statement (ii) of Lemma \ref{lemma:equivalent_system} also holds.
The proof of Lemma \ref{lemma:equivalent_system} is complete.
\end{Proof}

\subsection{\texorpdfstring{Definition of $\theta^*$}{Definition of theta*}}\label{sec:lim_f_c=0}

\paragraph{}In this section our goal is to prove Lemmas \ref{L:tetdef} and \ref{lemma:lim_f_c=0}.  We will give an explicit formula for $\tilde{f}_\theta(\infty) = \lim \limits_{n \to \infty} \left( \frac{\partial}{\partial c} f_{\theta, c}(n) \right) \big|_{c=0_+}$ in Lemma \ref{lemma:f_tilde} and  prove that it is strictly decreasing in $\theta$. As we will see, this fact implies Lemmas  \ref{L:tetdef} and  \ref{lemma:lim_f_c=0}. This is a key point:  we will see in Sections \ref{sec:theta<=theta*} and \ref{sec:theta>theta*} that the sign of $\tilde{f}_\theta(\infty)$  determines whether or not we have a non-diagonal scale invariant solution of the RDE.

\medskip

Recall from Corollary \ref{corr_f_inc_dec} that $f_{\theta,c}$ satisfies \eqref{f0_with_c} and \eqref{fn_with_c}. If we differentiate these equations with respect to $c$, we obtain
\begin{align}
&\frac{\partial}{\partial c} f_{\theta, c}(0) = \frac{2(1+\theta)}{\sqrt{1+8c (1+\theta)^2}},\label{f0_diff_c} \\
&\frac{\partial}{\partial c} f_{\theta, c}(n) = \frac{\frac{\partial}{\partial c} f_{\theta, c}(n-1) \cdot (2f_{\theta, c}(n-1) - \theta^{n-1}) - \theta^{2n-2} \cdot (1-\theta^2)}{\sqrt{\left(2f_{\theta, c}(n-1) - \theta^{n-1} \right)^2 - 4c \cdot \theta^{2n-2} \cdot (1-\theta^2)}}, \quad n \geq 1. \label{fn_diff_c}
\end{align}

For any $\theta \in (0,1)$, let us define
\begin{equation}\label{gamma_n_theta}
\gamma_n(\theta):=\frac{\theta^{2n-2} \cdot (1-\theta^2)}{\frac{2}{1+\theta} - \theta^{n-1}}=\frac{(1+\theta)^2 \theta^{n-1}}{ 1 +2 \sum_{k=1}^{n-1} \theta^{-k} }, \quad
n \geq 1.
\end{equation}
Note that the equality of the two formulas in \eqref{gamma_n_theta} holds for all $\theta \in (0,1)$, but the second formula for $\gamma_n(\theta)$  extends continuously to $\theta=1$ as well.

Recall the notations $\tilde{f}_\theta(n)$ and $\tilde{f}_\theta(\infty)$ of Definition \ref{def:f_tilde}.

\begin{lemma}[Formulas for  $\tilde{f}_\theta$]\label{lemma:f_tilde}
We have
\begin{align}
\tilde{f}_\theta(0) &= 2(1+\theta), \label{f_tilde_0} \\
\tilde{f}_\theta(n) &= \tilde{f}_{\theta}(n-1) - \gamma_n(\theta)=2(1+\theta)-\sum_{k=1}^n  \gamma_k(\theta), \quad n \geq  1, \label{f_tilde_n}\\
\tilde{f}_\theta(\infty) &= 2(1+\theta) -\sum_{k=1}^\infty  \gamma_k(\theta) \stackrel{ \eqref{gamma_n_theta} }{=} 1-\theta^2 -\sum_{k=2}^\infty  \gamma_k(\theta) . \label{lim_f_tilde}
\end{align}
\end{lemma}

\begin{Proof}
Substituting $c=0$ into \eqref{f0_diff_c}, we get \eqref{f_tilde_0}.
Similarly, if we substitute $c=0$ into \eqref{fn_diff_c} using that $f_{\theta, 0}(n) = \frac{1}{1+\theta}$ for every $n \in \mathbb{N}$ (see \eqref{f_theta_null}), we get \eqref{f_tilde_n}.
From \eqref{f_tilde_n} we get \eqref{lim_f_tilde} by the definition of $\tilde{f}_\theta(\infty)$ (c.f.\ \eqref{eq_tilde_f_theta_n}).
\end{Proof}

\begin{Proof}[of Lemmas \ref{L:tetdef} and \ref{lemma:lim_f_c=0}] Note that the function $ \theta \mapsto \tilde{f}_\theta(\infty)$ defined in \eqref{lim_f_tilde} coincides with the function $\theta \mapsto g(\theta)$ defined in Lemma \ref{L:tetdef}.

 First we show that  $\tilde{f}_\theta (\infty)$ is a decreasing function of $\theta \in [0,1)$.
We begin by observing that $\gamma_n(\theta)$ is an increasing function of $\theta \in [0,1)$ by the second formula for $\gamma_n(\theta)$ in \eqref{gamma_n_theta}.
Thus by the second formula for $\tilde{f}_\theta (\infty)$ in \eqref{lim_f_tilde} we obtain
  that  $\tilde{f}_\theta (\infty)$ is a decreasing function of $\theta \in [0,1) $.
The function $\theta \mapsto \tilde{f}_\theta (\infty)$ is also continuous on any compact sub-interval of $[0,1)$, since it
is the uniform limit of continuous functions.

 In order to complete the proof of Lemmas \ref{L:tetdef} and \ref{lemma:lim_f_c=0}, we just have to show that $\tilde{f}_{1/2} (\infty)>0$ and $\tilde{f}_{1-\varepsilon} (\infty)<0$ for some $\varepsilon>0$. Indeed, $\tilde{f}_{1/2} (1)=\frac{3}{4}$ and
  \begin{equation}
  \tilde{f}_{1/2} (\infty) \stackrel{ \eqref{lim_f_tilde} }{=} 1-\left(\frac{1}{2} \right)^2  -\sum_{k=2}^\infty  \frac{2^{2-2k} \cdot (1-\frac{1}{4})}{\frac{2}{1+1/2} - 2^{1-k}} \geq
  \frac{3}{4}-\sum_{k=2}^\infty  \frac{2^{2-2k} \cdot (1-\frac{1}{4})}{\frac{2}{1+1/2} - 1/2}= \frac{9}{20} >0.
  \end{equation}
 On the other hand,  $\tilde{f}_1 (2)=-4/3$ by \eqref{gamma_n_theta} and \eqref{f_tilde_n}, moreover $\theta \mapsto \tilde{f}_\theta(2)$ is a continuous function on $[0,1]$, therefore $\tilde{f}_{1-\varepsilon}(2)<0$ for some $\varepsilon>0$, from which $\tilde{f}_{1-\varepsilon} (\infty)<0$ follows,
 since $\tilde{f}_\theta(\infty) \leq \tilde{f}_\theta(2)$ by \eqref{f_tilde_n} and \eqref{lim_f_tilde}. The proofs of Lemmas \ref{L:tetdef} and \ref{lemma:lim_f_c=0} are complete.
\end{Proof}

\subsection{\texorpdfstring{The $\theta \leq \theta^* $ case}{The theta leq theta* case}}\label{sec:theta<=theta*}

\paragraph{}The goal of this section is to prove Lemma \ref{lemma:theta<=theta*}.


\begin{lemma}[Lower bound on $f_{\theta,c}(n)$]\label{lemma:theta<theta*}
If $\theta \in \left(\frac 12, \theta^* \right]$ and $c \in \left(0, \frac{\theta \cdot (2\theta - 1)}{(1+\theta)^2}\right]$, then
\begin{equation}\label{f_1_strict_ineq}
f_{\theta,c}(1)- \tilde{f}_\theta(1) c > \frac{1}{1+\theta},
\end{equation}
\begin{equation}\label{induction_f_n_ineq}
f_{\theta,c}(n)-\tilde{f}_\theta(n) c \geq f_{\theta,c}(1)- \tilde{f}_\theta(1) c, \qquad n \geq 1.
\end{equation}

\end{lemma}

Before we prove Lemma \ref{lemma:theta<theta*}, let us deduce Lemma \ref{lemma:theta<=theta*} from it.

\begin{Proof}[of Lemma \ref{lemma:theta<=theta*}]
 By Lemma \ref{lemma:lim_f_c=0} we have $\lim_{n \to \infty} \tilde{f}_\theta(n)= \tilde{f}_\theta(\infty) \geq 0$ for any $\theta \leq \theta^*$, where $\tilde{f}_\theta(\infty)$ is defined in Definition \ref{def:f_tilde}. Thus
 \begin{equation}
 \lim_{n \to \infty} f_{\theta,c}(n)  \geq
   \lim_{n \to \infty}\left( f_{\theta,c}(n)- \tilde{f}_\theta(n) c \right)
  \stackrel{\eqref{induction_f_n_ineq}  }{\geq}
 f_{\theta,c}(1)- \tilde{f}_\theta(1) c\stackrel{ \eqref{f_1_strict_ineq} }{>}\frac{1}{1+\theta}
 \end{equation}
 holds for any $\theta \in \left(\frac 12, \theta^* \right]$ and $c \in \left(0, \frac{\theta \cdot (2\theta - 1)}{(1+\theta)^2}\right]$. The proof of Lemma \ref{lemma:theta<=theta*} is complete.
\end{Proof}

\begin{remark} We will prove \eqref{induction_f_n_ineq} by induction on $n$.
 We have to start the induction from $n=1$, since it can be easily seen that the analogue of \eqref{f_1_strict_ineq} does not hold in the $n=0$ case, i.e., we have
$f_{\theta,c}(0)- \tilde{f}_\theta(0) c < \frac{1}{1+\theta}$.
\end{remark}

\begin{Proof}[of \eqref{f_1_strict_ineq}.] By \eqref{gamma_n_theta} and \eqref{f_tilde_n} we have $\tilde{f}_\theta(1)=1-\theta^2$, so by \eqref{fn_with_c} we need to show
$\frac{1}{2} \left( 1+\sqrt{(2 f_{\theta,c}(0)-1)^2 - 4c(1-\theta^2)  } \right) -(1-\theta^2)c > \frac{1}{1+\theta}$. Applying a series of equivalent transformations, we see that we need
\begin{equation} f_{\theta,c}(0) > \frac{1}{2}\left( 1+ \sqrt{ \frac{(1-\theta)^2}{(1+\theta)^2} +8(1-\theta)c + 4 (1-\theta^2)^2 c^2  }  \right).
\end{equation}
 Substituting the formula \eqref{f0_with_c}
for $f_{\theta,c}(0)$ into this, we obtain after some rearrangements that we need to show
\begin{equation*}
\sqrt{1+8(1+\theta)^2c}-\theta > \sqrt{ (1-\theta)^2 +8(1-\theta)(1+\theta)^2c + 4 (1-\theta^2)^2(1+\theta)^2 c^2  }.
\end{equation*}
Taking the square of both sides of this inequality, introducing the notation $\alpha=(1+\theta)^2 c$ and rearranging a bit, we obtain that we need to show that
$8 \theta \alpha > 2\theta (\sqrt{1+8\alpha}-1)+4(1-\theta)^2 \alpha^2$ holds. Introducing the notation $\beta=\sqrt{1+8\alpha}-1$, we may equivalently rewrite this and obtain that we need to show $ \beta < \frac{4 \sqrt{\theta}}{1-\theta}-2$. Using the definition of $\alpha$ and $\beta$,  our assumption $c \leq \frac{\theta \cdot (2\theta - 1)}{(1+\theta)^2}$ becomes $\beta \leq \sqrt{ (1-4\theta)^2 }-1$. Using that $\theta > \frac{1}{2}$ we see that we have $\beta \leq 4\theta -2$, so it is enough to show
$4\theta < \frac{4 \sqrt{\theta}}{1-\theta} $ to conclude the desired inequality $ \beta < \frac{4 \sqrt{\theta}}{1-\theta}-2$. Now $\theta < \frac{ \sqrt{\theta}}{1-\theta} $ does hold for all $\theta \in (0,1)$ (therefore it holds for $\theta \in \left(\frac 12, \theta^* \right]$), completing the proof of \eqref{f_1_strict_ineq}.
\end{Proof}

\begin{Proof}[of \eqref{induction_f_n_ineq}.]
We prove \eqref{induction_f_n_ineq} by induction on $n$. The $n=1$ case trivially holds.  Let $n \geq 2$. Let us denote $q=\tilde{f}_\theta(n-1) c+ f_{\theta,c}(1)- \tilde{f}_\theta(1) c$. By our induction hypothesis we know that $f_{\theta,c}(n-1)\geq q$ holds, and we want to show that \eqref{induction_f_n_ineq} also holds, or, equivalently, we want
 $f_{\theta,c}(n)\geq q -\gamma_n(\theta)c $ to hold
(c.f.\ \eqref{gamma_n_theta}, \eqref{f_tilde_n}). Let us note that we have
\begin{equation}\label{q_gamma_theta_ineq}
q -\gamma_n(\theta)c \stackrel{\eqref{f_tilde_n}, \eqref{f_1_strict_ineq}}{\geq} \tilde{f}_\theta(n) c +\frac{1}{1+\theta} \stackrel{(*)}{\geq} \frac{1}{1+\theta},
\end{equation}
where $(*)$ holds since our assumption $\theta \leq \theta^*$ and  Lemma \ref{lemma:lim_f_c=0} together imply  $\tilde{f}_\theta(\infty)\geq 0$  and the formulas \eqref{gamma_n_theta}, \eqref{f_tilde_n} and
\eqref{lim_f_tilde} together imply $\tilde{f}_\theta(n)\geq \tilde{f}_\theta(\infty)$.
Using our induction hypothesis and  \eqref{fn_with_c}, we see that it is enough to prove
\begin{equation}\label{enough_to_prove_ind_q}
\frac{1}{2} \left( \theta^{n-1} +\sqrt{ (2q-\theta^{n-1})^2-4c\gamma_n(\theta)\left(\frac{2}{1+\theta}-\theta^{n-1} \right)  } \right) \geq q-\gamma_n(\theta)c
\end{equation}
in order to arrive at the desired $f_{\theta,c}(n)\geq q -\gamma_n(\theta)c $. We will now show \eqref{enough_to_prove_ind_q}.
 We first show that the expression under the square root is non-negative:
\begin{multline}\label{expr_under_sq_root_non-neg}
    (2q-\theta^{n-1})^2-4c\gamma_n(\theta)\left(\frac{2}{1+\theta}-\theta^{n-1} \right) \stackrel{ \eqref{q_gamma_theta_ineq} }{\geq}
     \left( \left(\frac{2}{1+\theta}-\theta^{n-1}\right) + 2c\gamma_n(\theta) \right)^2- \\ 4c\gamma_n(\theta)\left(\frac{2}{1+\theta}-\theta^{n-1} \right)=
      \left(\frac{2}{1+\theta}-\theta^{n-1}\right)^2 +\left(  2c\gamma_n(\theta) \right)^2 \geq 0.
 \end{multline}
Using this
 we can rearrange \eqref{enough_to_prove_ind_q} and see that it is equivalent to
\begin{equation}
  (2q-\theta^{n-1})^2-4c\gamma_n(\theta)\left(\frac{2}{1+\theta}-\theta^{n-1} \right) \geq \left( (2q-\theta^{n-1}) -2 \gamma_n(\theta)c  \right)^2,
\end{equation}
 which is in turn equivalent to
  $2\left(q-\frac{1}{1+\theta}\right) \geq \gamma_n(\theta)c $, and
this inequality indeed holds by \eqref{q_gamma_theta_ineq}. The proof of the induction step is complete.
\end{Proof}

The proof of Lemma \ref{lemma:theta<theta*} is complete.

\begin{remark} Our assumption
$c \in \left(0, \frac{\theta \cdot (2\theta - 1)}{(1+\theta)^2}\right]$ that appears in the statement of Lemma \ref{lemma:theta<theta*} (or something similar to it) seems
indispensable, because numerical simulations suggest that the conclusions of Lemma \ref{lemma:theta<theta*} do not hold for big values of $c$.
\end{remark}

\subsection{\texorpdfstring{The $\theta> \theta^*$ case}{The theta>theta* case}}\label{sec:theta>theta*}

\paragraph{}In this section we prove Lemma \ref{lemma:theta>theta*}. First we show Lemma \ref{lemma:f_lower_bound}, which implies that $f_{\theta, c}(\infty)$ is large if $c$ is large. We will also argue that $f_{\theta, c}(\infty)<\frac{1}{1+\theta}$ if
$\theta > \theta^*$ and $c$ is small.
We then combine these facts to show that there exists a $\hat{c} >0$ for which $ f_{\theta, \hat{c}}(\infty) = \frac{1}{1+\theta}$. After that we will see in Lemma \ref{lemma:conditions_fulfil} and \ref{lemma:every_alpha} that this $f_{\theta, \hat{c}}$ satisfies the conditions of Lemma \ref{lemma:f_conditions} (and therefore it is the signature of a non-diagonal solution $\hat{\rho}^{(2)} \in \mathcal{M}^{(2)}_{\theta}$ of the bivariate RDE \eqref{bivar_RDE}).

\begin{lemma}[Lower bound on $f_{\theta, c}$]\label{lemma:f_lower_bound}
If $\theta \in (0,1)$  and $c \ge 4$, then
\begin{equation}\label{lower_bound_on_f}
f_{\theta, c}(n) \ge \frac{\theta^n}{2} + \sqrt{ \left( \frac{1}{2} +\theta^{2n} \right) \cdot c}, \qquad n \in \mathbb{N}.
\end{equation}
\end{lemma}

\begin{Proof}
We prove \eqref{lower_bound_on_f} by induction on $n$. The $n=0$ case holds, since
\begin{equation}
f_{\theta, c}(0) \stackrel{\eqref{f0_with_c}}{\geq } \frac{1+\sqrt{8c(1+\theta)^2}}{2(1+\theta)} \geq \frac{1}{4} + \sqrt{2 c}
\stackrel{(*)}{\geq}
\frac{1}{2} + \sqrt{ \frac{3}{2}  c} =\frac{\theta^0}{2} + \sqrt{ \left( \frac{1}{2} +\theta^{2\cdot 0} \right) \cdot c} ,
\end{equation}
where $(*)$ holds if $c \geq 4$. Now assume that $n \geq 1$ and \eqref{lower_bound_on_f} holds for $n-1$, and we want to deduce that \eqref{lower_bound_on_f}
also holds with $n$ as well:
\begin{multline*}
f_{\theta, c}(n) \stackrel{\eqref{fn_with_c}}{=}
 \frac{\theta^{n-1} + \sqrt{(2f_{\theta, c}(n-1) - \theta^{n-1})^2 - 4c \cdot \theta^{2n-2} \cdot (1-\theta^2)}}{2}
\stackrel{(**)}{\geq} \\
\frac{\theta^{n} + \sqrt{  4 (\frac{1}{2} + \theta^{2(n-1)} ) c    - 4c \cdot \theta^{2n-2} \cdot (1-\theta^2)}}{2} =
 \frac{\theta^n}{2} + \sqrt{ \left( \frac{1}{2} +\theta^{2n} \right) \cdot c},
\end{multline*}
where in $(**)$ we used the induction hypothesis and also that $\theta^{n-1}\geq \theta^n$.
\end{Proof}

\begin{lemma}[$f_{\theta, c}$ satisfies necessary conditions]\label{lemma:conditions_fulfil}
If $\theta \in (0,1)$ and $c \ge 0$ are arbitrary, then $f_{\theta, c}$ satisfies conditions (iii), (iv) and (v) of Lemma \ref{lemma:f_conditions}, i.e.
\begin{enumerate}[1.]
\item $f_{\theta, c}(n)$ is non-increasing in $n$,
\item $(1+\theta) \cdot f_{\theta, c}(0) \le 2 f_{\theta, c}(1)$,
\item $(1+\theta) \cdot f_{\theta, c}(n) \le \theta \cdot f_{\theta, c}(n-1) + f_{\theta, c}(n+1)$ for every $n \geq 1$.
\end{enumerate}
\end{lemma}

\begin{Proof}
 $ $

\begin{enumerate}[1.]
\item We have already seen this in Corollary \ref{corr_f_inc_dec}.

\item Recalling the notation introduced at the beginning of Section \ref{sec:f_unique} (see in particular \eqref{g_recursion} and \eqref{def_eq_g}), we want to show
\begin{equation}\label{g_measure_requirement}
 (1+\theta)g_{\theta,c}(0) \leq 2 \theta \psi_{\theta,c}\big(g_{\theta,c}(0)\big).
\end{equation}
Using the definition  \eqref{def_eq__psi_theta_c} of $\psi_{\theta,c}$ and $\mathcal{D}_{\theta,c}$ one deduces that
\begin{align}
\label{psi_incr_concave} \text{the function $x \mapsto 2 \theta \psi_{\theta,c}(x)$ is increasing and concave on $\mathcal{D}_{\theta,c}$,}\\
\label{psi_slope_big}  \frac{\mathrm{d}}{\mathrm{d}x} 2 \theta \psi_{\theta,c}(x) = 2 \frac{2x-1}{\sqrt{(2x-1)^2 -4c\cdot(1-\theta^2) }} > 2 >1+\theta , \qquad x \in \mathcal{D}_{\theta,c},\\
\label{psi_bigger_than_that_line_for_large_x} \lim_{x \to \infty} 2 \theta \psi_{\theta,c}(x) - (1+\theta)x = +\infty.
\end{align}
It follows from \eqref{psi_incr_concave} and \eqref{psi_slope_big} that
\begin{equation}\label{atmost_one_solutionn}
\text{ the equation $2 \theta \psi_{\theta,c}(x)=(1+\theta)x$ has at most one solution in $\mathcal{D}_{\theta,c}$.}
\end{equation}
 Let $y_0:=\sqrt{(1-\theta^2)c}+1/2$ denote the left endpoint of $\mathcal{D}_{\theta, c}$.
  One easily checks that $ 2\theta \psi_{\theta,c}(y_0)\geq (1+\theta) y_0$ holds if and only if $c \leq \frac{1-\theta}{4(1+\theta)^3}$ holds. We will prove \eqref{g_measure_requirement} by treating the cases $ 2 \theta \psi_{\theta,c}(y_0) \geq (1+\theta) y_0$ and $ 2 \theta \psi_{\theta,c}(y_0)< (1+\theta)y_0$ separately.

If $ 2 \theta \psi_{\theta,c}(y_0) \geq (1+\theta) y_0$ then $  2 \theta \psi_{\theta,c}(x) \geq (1+\theta)x $  for every $x \in \mathcal{D}_{\theta,c}$ follows from \eqref{psi_slope_big}, and in particular \eqref{g_measure_requirement} holds.

If $ 2 \theta \psi_{\theta,c}(y_0)< (1+\theta)y_0$ then this inequality, \eqref{psi_bigger_than_that_line_for_large_x} and \eqref{atmost_one_solutionn} together imply that
 there exists a unique $\tilde{x} \in \mathcal{D}_{\theta,c}$ such that
  $2\theta \psi_{\theta,c}(\tilde{x})=(1+\theta)\tilde{x}$, moreover we obtain using \eqref{psi_slope_big} that $ \tilde{x} \leq x$ implies   $  2\theta \psi_{\theta,c}(x) \geq (1+\theta)x $.
   One easily finds that $\tilde{x}=\frac{1+\sqrt{4(\theta+3)(\theta+1)c+1  }}{\theta+3}$, thus we only need to check $\tilde{x} \leq g_{\theta,c}(0)$, i.e., by the definition \eqref{g_recursion} of $g_{\theta,c}(0)$  we need to check that $\alpha_\theta(c)\leq \beta_\theta(c)$ holds for all $c \geq 0$, where
\begin{equation}
\alpha_\theta(c):= \frac{1+\sqrt{4(\theta+3)(\theta+1)c+1  }}{\theta+3}, \qquad \beta_\theta(c):= \frac{1+\sqrt{1+8(1+\theta)^2 c} }{2(\theta+1)}.
\end{equation}
The inverse functions of both  $ c \mapsto \alpha_\theta(c)$ and $ c \mapsto \beta_\theta(c)$ are quadratic polynomials:
\begin{equation} \alpha^{-1}_\theta(y)=\frac{((\theta+3)y-1 )^2-1}{4(\theta+3)(\theta+1)}, \qquad \beta^{-1}_\theta(y)=\frac{ (2(\theta+1)y-1)^2-1 }{8(\theta+1)^2}.
\end{equation}
 It is enough to check that $\alpha^{-1}_\theta(y)\geq \beta^{-1}_\theta(y)$ holds for all $y \in \mathbb{R}$, and indeed we have $\alpha^{-1}_\theta(y)-\beta^{-1}_\theta(y)=\frac{(1-\theta)y^2}{4(\theta+1)}$, which is nonnegative for all $\theta \in (0,1], y \in \mathbb{R}$.

\item We have to show $f_{\theta, c}(n) - f_{\theta, c}(n+1) \le \theta \cdot (f_{\theta, c}(n-1) - f_{\theta, c}(n))$
for every $n \ge 1$. Rewriting this using the notation introduced in Section \ref{sec:f_unique} as well as \eqref{def_eq_g}, we need to show that the inequality
 $g_{\theta,c}(n)-\theta g_{\theta,c}(n+1)\leq g_{\theta,c}(n-1)-\theta g_{\theta,c}(n)$ holds. Since $g_{\theta,c}(n+1)=\psi_{\theta,c}(g_{\theta,c}(n))$ and
  $g_{\theta,c}(n)=\psi_{\theta,c}(g_{\theta,c}(n-1))$ by \eqref{g_recursion}, moreover we know  $g_{\theta,c}(n)\geq g_{\theta,c}(n-1)$ (c.f.\ \eqref{g_increasing}),  it is enough
  to show that $\varphi_{\theta,c}(x)$ is a decreasing function of $x$, where $\varphi_{\theta,c}(x):=x-\theta \psi_{\theta,c}(x)$. This is indeed the case, since
  we have $\varphi'_{\theta,c}(x)=1-\frac{2x-1}{\sqrt{(2x-1)^2 -4c(1-\theta^2) }}<0$ for every $x  $ in the domain $\mathcal{D}_{\theta,c}$ of $\varphi_{\theta,c}(\cdot)$.
\end{enumerate}
\end{Proof}

\begin{lemma}[Upper bound on $f_{\theta, \hat{c}}(0)$]\label{lemma:every_alpha}
If $\theta \in (0,1)$ and $f_{\theta, \hat{c}}(\infty) = \frac{1}{1+\theta}$, then $f_{\theta, \hat{c}}(0) \le 1$.
\end{lemma}

\begin{Proof} The conditions of Lemma \ref{lemma_properties_of_f_theta_c} are fulfilled for $f_{\theta, \hat{c}}$, thus we may use
\eqref{f_theta_c_0_in_interval} to conclude $f_{\theta, \hat{c}}(0) \leq \frac{1}{1+\theta} \vee \frac{2 \theta}{1+\theta} \leq 1$.
\end{Proof}

\begin{Proof}[of Lemma \ref{lemma:theta>theta*}]
We will show that the function $c \mapsto  f_{\theta, c}(\infty) -\frac{1}{1+\theta} $ takes both positive and negative values. This is enough to conclude the proof of the first statement of Lemma \ref{lemma:theta>theta*},  since this function is continuous by Lemma \ref{lemma:f_continuous}.

We know from \eqref{f_theta_null} that $ f_{\theta, 0}(n) = \frac{1}{1+\theta}$ for all $n \in \mathbb{N}$.
 By the $\theta > \theta^*$ case of Lemma \ref{lemma:lim_f_c=0} we have $\tilde{f}_\theta(\infty) = \lim \limits_{n \to \infty} \tilde{f}_\theta (n) < 0$, therefore we can fix an
 $n \in \mathbb{N}$ such that $\tilde{f}_\theta (n) < 0$.
  Recall from Definition \ref{def:f_tilde} that $\tilde{f}_\theta (n)$ denotes $ \frac{\partial}{\partial c} f_{\theta, c}(n)  \big|_{c=0_+}$.
  We can thus fix a small but positive value of $c$ such that $f_{\theta, c}(n) < f_{\theta, 0}(n)  =\frac{1}{1+\theta}$.  Now $f_{\theta, c}(\infty) < \frac{1}{1+\theta} $ follows from the fact that $f_{\theta, c}(n)$ decreases as $n$ increases (c.f.\ \eqref{f_decreasing}).

    Next we show that there exists a $c > 0$ for which $f_{\theta, c}(\infty) > \frac{1}{1+\theta}$. This follows from Lemma \ref{lemma:f_lower_bound}, since for $c \ge 4$ we have
\begin{equation}
\lim \limits_{n \to \infty} f_{\theta, c}(n) \stackrel{\eqref{lower_bound_on_f}}{\ge} \lim \limits_{n \to \infty} \left(
\frac{\theta^n}{2} + \sqrt{ \left( \frac{1}{2} +\theta^{2n} \right) \cdot c} \right) = \sqrt{\frac{1}{2}c}  > \frac{1}{1+\theta}.
\end{equation}

Therefore there exists $\hat{c} > 0$ such that $f_{\theta, \hat{c}}(\infty) = \frac{1}{1+\theta}$.

Now we prove the second statement of Lemma \ref{lemma:theta>theta*}. Since $ f_{\theta, \hat{c}}(\infty) = \frac{1}{1+\theta}$, condition (ii) of Lemma \ref{lemma:f_conditions} holds and condition (i) also holds by Lemma \ref{lemma:every_alpha}. By Lemma \ref{lemma:conditions_fulfil} we also know that conditions (iii), (iv) and (v) of Lemma \ref{lemma:f_conditions} are true. So we can conclude that $f_{\theta, \hat{c}}$ satisfies all of the conditions of Lemma \ref{lemma:f_conditions}.
\end{Proof}

\begin{remark}\label{remark:conj_unique}
In Figure \ref{fig:85} we can see $ f_{0.85, c}(\infty)$ as a function of $c$, where $c$ is an element of the interval $c \in \left[0, \frac{0.85 \cdot (2 \cdot 0.85 - 1)}{(1+0.85)^2}\right]$. The horizontal red line is the constant $\frac{1}{1+\theta} \stackrel{\theta = 0.85}{=} \frac{20}{37}$. We see that first it is decreasing, then it is increasing and goes to infinity, thus there exists $\hat{c} > 0$ for which $ f_{0.85, \hat{c}}(\infty) = \frac{20}{37}$. We get a similar picture for every $\theta \in \left(\theta^*, 1\right)$.

We also note that Figure \ref{fig:85} suggests that Conjecture \ref{conj:unique} holds, since this conjecture is equivalent with the fact that there exists exactly one $\hat{c}>0$ for which $ f_{\theta, \hat{c}}(\infty) = \frac{1}{1+\theta}$.
\end{remark}

\begin{figure}[!ht]
	\centering
	\includegraphics[scale=0.31]{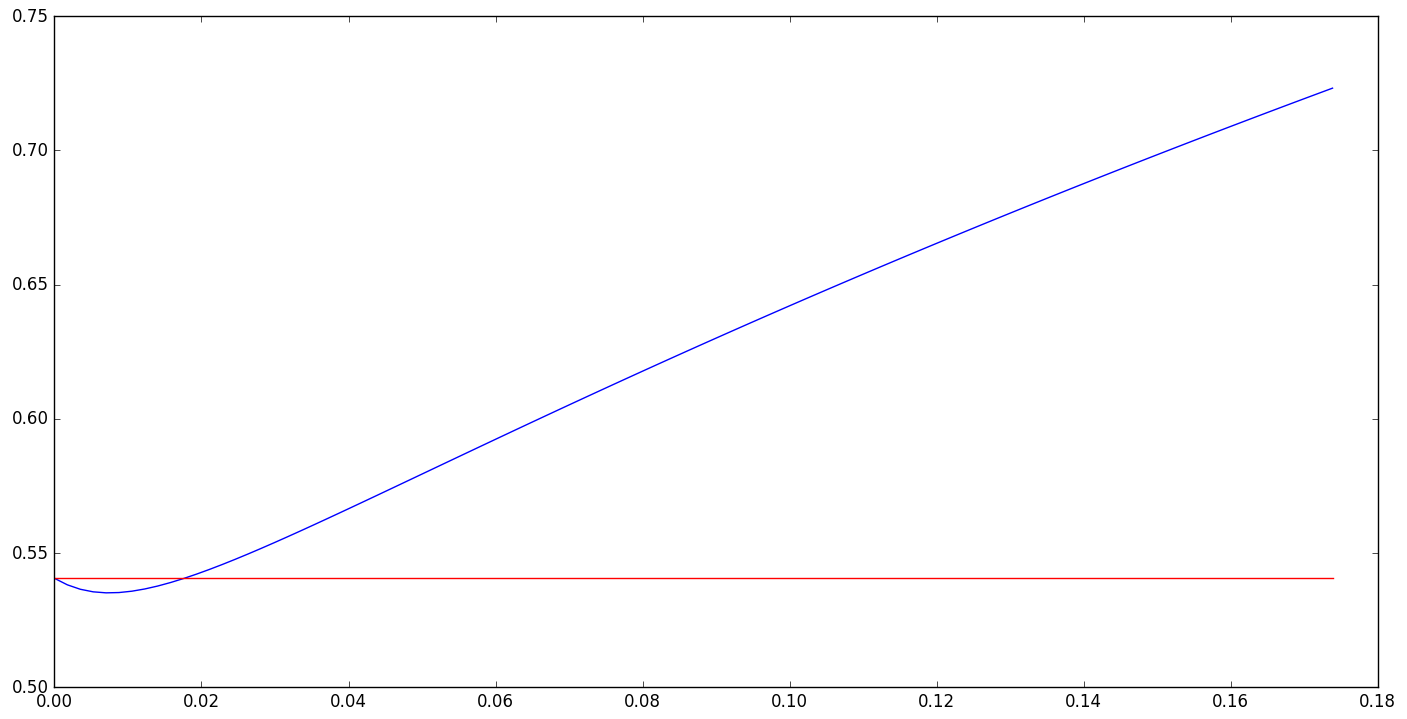}
	\caption{$ f_{0.85, c}(\infty)$}
	\label{fig:85}
\end{figure}


\begin{thebibliography}{RST19}

\bibitem[AB05]{AB05}
D.J.~Aldous and A.~Bandyopadhyay.
A survey of max-type recursive distributional equations.
\emph{Ann.\ Appl.\ Probab.}~15(2) (2005), 1047--1110.




\bibitem[Ald00]{Ald00}
D.J.~Aldous.
The percolation process on a tree where infinite clusters are frozen.
\emph{Math.\ Proc.\ Cambridge Philos.\ Soc.}~128 (2000), 465--477.















\bibitem[BT01]{BT01}
J.~van den Berg, B.~T\'oth.
A signal-recovery system: asymptotic properties, and construction of an
infinite-volume process.
\emph{Stochastic Process.\ Appl.}~96(2) (2001), 177--190.
























\bibitem[MSS18]{MSS18}
T.~Mach, A.~Sturm, and J.M.~Swart.
A new characterization of endogeny.
\emph{Math.\ Phys.\ Anal.\ Geom.}~21(4) (2018), no.~30.

\bibitem[MSS20]{MSS20}
T.~Mach, A.~Sturm, and J.M.~Swart.
Recursive tree processes and the mean-field limit of stochastic flows.
\emph{Electron.\ J.\ Probab.}~25 (2020) paper No.~61, 1--63.





\bibitem[RST19]{RST19}
B.~R\'ath, J.M.~Swart, and T.~Terpai.
Frozen percolation on the binary tree is nonendogenous.
\emph{Ann.\ of Probab.} Vol.~49(5)  (2021) 2272 -- 2316.

\bibitem[SSS14]{SSS14}
E.~Schertzer, R.~Sun and J.M.~Swart.
Stochastic flows in the Brownian web and net.
\emph{Mem.\ Am.\ Math.\ Soc.} Vol.~227 (2014), Nr.~1065.



\end{thebibliography}
\end{document}